\documentclass[11pt, letterpaper]{amsart}

\usepackage{enumerate}
\usepackage{enumitem}
\usepackage{dsfont}
\usepackage{amsxtra}
\usepackage{amsmath}
\usepackage{amscd}
\usepackage{amssymb}
\usepackage{amsfonts}
\usepackage[all]{xy}
\usepackage{mathrsfs}
\usepackage{amsthm} 
\usepackage{color}
\usepackage{upgreek}
\usepackage{latexsym}
\usepackage[english]{babel} 
\usepackage{hyperref}
\usepackage{stmaryrd}
\usepackage{nicefrac}
\usepackage{euscript}
\usepackage{mathabx}
\usepackage[foot]{amsaddr}

\usepackage{amscd}

\newcommand{\HC}{{\mathcal H}(C)}

\newcommand{\g}{{\mathfrak g}}

\newcommand{\Tr}{{\bf Tr}}

\newcommand\bfhc{{\bfh\bfc}}

\newcommand{\Lg}{{\check\g}}

\newcommand{\LG}{{\check G}}
\newcommand{\LB}{{\check B}}
\newcommand{\LN}{{\check N}}
\newcommand{\LLbd}{{\check\Lambda}}
\newcommand{\LP}{{\check P}}
\newcommand{\LQ}{{\check Q}}

\newcommand{\LT}{{\check T}}
\newcommand{\LZ}{{\check Z}}
\newcommand{\Lpi}{{\check\pi}}

\newcommand{\Lrho}{{\check\rho}}
\newcommand{\Lalpha}{{\check\alpha}}

\newcommand{\LPhi}{{\check\Phi}}
\newcommand{\LGamma}{{\check\Gamma}}

\newcommand{\Om}{{\Omega}}
\newcommand{\LOm}{{\check\Omega}}

\newcommand{\Lbd}{\Lambda}

\renewcommand{\P}{{\mathcal P}}
\newcommand{\R}{{\mathcal R}}
\newcommand{\Fl}{{{\mathcal F}\ell}}
\newcommand{\Gr}{{{\mathcal G}{\mathfrak{r}} }}

\usepackage{amsfonts}
\usepackage{color}
\usepackage{upgreek}
\usepackage{verbatim}
\usepackage{enumerate}
\usepackage{latexsym}
\usepackage{graphicx}
\usepackage{tikz}
\usepackage{todonotes}
\usepackage{setspace}
\usepackage{hyperref}
\usepackage{stmaryrd}
\usepackage{nicefrac}
\usepackage{euscript}
\usetikzlibrary{decorations.pathmorphing}

 \definecolor{darkred}{HTML}{993333}
\makeatletter

\newcommand{\arxiv}[1]{\href{http://arxiv.org/abs/#1}{\tt
arXiv:\nolinkurl{#1}}}

\theoremstyle{plain}
\newtheorem{theorem}{Theorem}[section]
\newtheorem{lemma}[theorem]{Lemma}
\newtheorem{lemma-definition}[theorem]{Lemma-Definition}
\newtheorem{definition-lemma}[theorem]{Definition-Lemma}

\newtheorem{proposition}[theorem]{Proposition}
\newtheorem{conjecture}[theorem]{Conjecture}
\newtheorem{corollary}[theorem]{Corollary}

\newtheorem*{theorem-A}{Theorem A}
\newtheorem*{theorem-B}{Theorem B}
\newtheorem*{theorem-C}{Theorem C}
\newtheorem*{theorem-D}{Theorem D}
\newtheorem*{conjecture-A}{Conjecture A}
\newtheorem*{conjecture-B}{Conjecture B}
\newtheorem*{conjecture-C}{Conjecture C}
\newtheorem*{conjecture-D}{Conjecture D}

\theoremstyle{definition}

\theoremstyle{remark}
\newtheorem{remark}[theorem]{Remark}

\numberwithin{equation}{section}

\xyoption{all}
\def\A{\mathrm{A}}

\def\D{\mathrm{D}}
\def\E{\mathrm{E}}

\def\I{\mathrm{I}}

\def\L{\mathrm{L}}
\def\M{\mathrm{M}}

\def\P{\mathrm{P}}
\def\Q{\mathrm{Q}}
\def\R{\mathrm{R}}

\def\U{\mathrm{U}}
\def\V{\mathrm{V}}
\def\W{\mathrm{W}}
\def\X{\mathrm{X}}

\def\u{\mathrm{u}}

\def\bbA{\mathbb{A}}

\def\bbC{\mathbb{C}}

\def\bbF{\mathbb{F}}
\def\bbG{\mathbb{G}}

\def\bbN{\mathbb{N}}

\def\bbP{\mathbb{P}}
\def\bbQ{\mathbb{Q}}
\def\bbR{\mathbb{R}}

\def\bbZ{\mathbb{Z}}

\def\scrU{\mathscr{U}}

\def\frakB{\mathfrak{B}}

\def\calO{\mathfrak{O}}
\def\frakP{\mathfrak{P}}

\def\frakR{\mathfrak{R}}

\def\frakU{\mathfrak{U}}

\def\frakZ{\mathfrak{Z}}

\def\calA{\mathcal{A}}

\def\calC{\mathcal{C}}
\def\calD{\mathcal{D}}

\def\calI{\mathcal{I}}

\def\calK{\mathcal{K}}
\def\calL{\mathcal{L}}

\def\calO{\mathcal{O}}
\def\calP{\mathcal{P}}

\def\calX{\mathcal{X}}

\def\frakb{\mathfrak{b}}

\def\frakg{\mathfrak{g}}

\def\frakn{\mathfrak{n}}

\def\frakp{\mathfrak{p}}

\def\frakt{\mathfrak{t}}
\def\fraku{\mathfrak{u}}

\def\bfa{\mathbf{a}}
\def\bfb{\mathbf{b}}
\def\bfc{\mathbf{c}}

\def\bfh{\mathbf{h}}

\def\bfk{\mathbf{k}}

\def\bfA{\mathbf{A}}

\def\bfC{\mathbf{C}}

\def\bfF{\mathbf{F}}

\def\bfK{\mathbf{K}}

\def\simto{\overset{\sim}\to}

\def\k{{\operatorname{k}\nolimits}}

\def\mod{\operatorname{-mod}\nolimits}

\def\proj{\operatorname{proj}\nolimits}

\def\Mod{\operatorname{-Mod}\nolimits}
\def\Rep{\operatorname{rep}\nolimits}

\def\Vect{\operatorname{Vect}\limits}

\def\Im{\operatorname{Im}\nolimits}
\def\Fun{\operatorname{Fun}\nolimits}
\def\Hom{\operatorname{Hom}\nolimits}

\def\End{\operatorname{End}\nolimits}

\def\Ind{\operatorname{Ind}\nolimits}
\def\Ext{\operatorname{Ext}\nolimits}

\def\id{\operatorname{id}\nolimits}

\def\Spec{\operatorname{Spec}\nolimits}

\def\Tr{\operatorname{Tr}\nolimits}
\def\da{\downarrow}
\def\St{\operatorname{St}\nolimits}
\def\Ad{\operatorname{Ad}\nolimits}

\def\Sym{\operatorname{Sym}\nolimits}
\def\Tr{\operatorname{Tr}\nolimits}

\def\can{\operatorname{can}\nolimits}
\def\modulo{\operatorname{mod}\nolimits}

\def\pt{\operatorname{pt}\nolimits}

\def\af{{\operatorname{af}\nolimits}}
\def\ex{{\operatorname{ex}\nolimits}}

\def\re{{\operatorname{re}\nolimits}}
\def\sc{{\operatorname{sc}\nolimits}}

\def\hyb{{\operatorname{hb}\nolimits}}

\def\HC{{\operatorname{HC}\nolimits}}
\def\Fr{{\operatorname{Fr}\nolimits}}

\def\res{\mathrm{res}}

\def\Ub{\U^{\hyb,b}}

\def\Ut{\frakU_\zeta^t}

\usepackage[toc,page]{appendix}

\usepackage{etoolbox}
\appto\appendix{\addtocontents{toc}{\protect\setcounter{tocdepth}{1}}}

\title[Center of the small quantum group]
{A geometric realization of the center of the small quantum group}
\author{Roman Bezrukavnikov$^1$}
\address{\scriptsize{$^1$~Massachusetts Institute of Technology, 77 Massachusetts Ave, Cambridge, MA 02139.}}

\author{Pablo Boixeda Alvarez$^2$}
\address{\scriptsize{$^2$~Yale University, New Haven, 06520 CT, USA, Partially supported by NSF grant DMS-1926686}}
\author{Peng Shan$^3$} 
\address{\scriptsize{$^3$~Yau Mathematical Sciences Center, Tsinghua University, 100084, Beijing, China.}}
\author{Eric Vasserot$^4$} 
\address{\scriptsize{$^4$~Universit\'e de Paris, 75013 Paris, France, UMR7586 (CNRS), ANR-18-CE40-0024,
Institut Universitaire de France (IUF).
}}

\begin{document}
\maketitle

\begin{abstract}
We propose a new geometric model for the center of the small quantum group using the  
cohomology of certain affine Springer fibers. More precisely, we establish an isomorphism between the equivariant cohomology of affine Spaltenstein fibers for a split element and the center of the deformed graded modules for the small quantum group. We also obtain an embedding from the invariant part of the nonequivariant cohomology under the action of the extended affine Weyl group to the $\LG$-invariant part of the center of the small quantum group, which we conjecture to be an isomorphism. Finally, we give a dimension formula for the invariants on the cohomology side, thus providing a lower bound for the dimension of the center.
\end{abstract}

%\tableofcontents

\bigskip

\section{Introduction}

\subsection{The main results}
In this paper, we propose a new geometric model for the center of small quantum groups. It is given by the cohomology of certain affine Spaltenstein varieties. The main results of the paper are the follows.

\smallskip

Let $\LG$ be a quasi-simple, connected, simply-connected algebraic group. Fix $\zeta\in\bbC$ a root of unity of order $\ell$. We assume that $\ell$ is odd,  greater than the Coxeter number, and coprime to the determinant of the Cartan matrix of $\LG$.
Let $\u_\zeta$ be the small quantum group. It is a finite dimensional subalgebra of the Lusztig's quantum group $\U_\zeta$. We denote by $Z(\u_\zeta)$ the center of $\u_\zeta$. The adjoint action of $\U_\zeta$ on itself preserves $Z(\u_\zeta)$. It factorises through the quantum Frobenius map, and yields an action of $\LG$ on $Z(\u_\zeta)$. The study 
of the center $Z(\u_\zeta)$ as well as the subspace of $\LG$-invariant elements in it has captured a lot of attention recently, see, e.g., \cite{LQ1}, \cite{LQ2}.

\smallskip

To study the center, we consider several module categories related to $\u_\zeta$.
Let $\LLbd$ be the weight lattice for $\LG$. 
Let $\u_\zeta\mod^\LLbd$ be the category of finite dimensional $\u_\zeta$-modules
equipped with $\LLbd$-graded structure which is compatible with the $\LT$-action on $\u_\zeta$. 
%This category is a quantum analog of the so-called $G_1T$-modules. It has been studied in \cite{AJS}. 
It admits a deformation, studied in \cite{AJS}, which we now describe. 
Let $\frakU_\zeta$ be the De Concini-Kac quantum group. 
The center of $\frakU_\zeta$ contains the ring of regular functions $\bbC[\LG^*]$ on the Poisson dual group 
$\LG^*$ of $\LG$. 
The specialisation of $\frakU_\zeta$ at $1\in \LG^*$ is the small quantum group
$\u_\zeta$.
The algebra $\frakU_{\zeta}^t=\frakU_\zeta \otimes _{\bbC[\LG^*]} \bbC[\LT]$ is a deformation of $\u_\zeta$. 
Let $S$ be the completion of $\bbC[\LT]$ at $1\in \LT$. 
The category $\frakU^t_\zeta\mod^\LLbd_S$ of $\LLbd$-graded $\frakU_{\zeta}^t$-modules over $S$ 
is a deformation of $\u_\zeta\mod^\LLbd$. 
The blocks of $\frakU^t_\zeta\mod^\LLbd_S$ are naturally in bijection with those of $\u_\zeta\mod^\LLbd$. We will study both the centers  $Z(\frakU^t_\zeta\mod^\LLbd_S)$ and $Z(\u_\zeta\mod^\LLbd)$. They carry natural actions of the lattice 
$\ell\LLbd$. There is a canonical isomorphism $Z(\u_\zeta\mod^\LLbd)^{\ell\LLbd}=Z(\fraku_\zeta)^{\LT}$.

\smallskip

Now, let us introduce the main objects on the geometric side. 
Let $G$ be the Langlands dual group of $\LG$ and let $T$ be the dual torus of $\LT$. Let $\Gr$ be the affine Grassmannian of the group $G$. We consider the $\bbG_m$-action on $\Gr$ by loop rotation. Let $\Gr^\zeta$ be the fixed points locus of the cyclic subgroup $\mu_\ell\subset\bbG_m$ generated by $\zeta$. 
Let $\gamma_\ell=s\otimes t^\ell$, where $s$ is a regular element in the Lie algebra of $T$. We consider the affine Spaltenstein variety ${}^0\Gr^\zeta_{\gamma_\ell}$ which is the closed ind-subscheme of $\Gr^\zeta$ consisting of points $x$ such that $\gamma_\ell$ belongs to the radical of the Lie algebra of the stabilizer of $x$ in $G(\bbC(\!(t)\!))$. 
Note that ${}^0\Gr^\zeta_{\gamma_\ell}$ is a disjoint union of the Spaltenstein subvarieties
${}^0\!\Fl^\omega_{\ell,\gamma_\ell}$ of the partial affine flag manifold $\Fl^\omega_{\ell}$
for the loop group $G(\bbC(\!(t^\ell)\!))$ of type $\omega$.
The space ${}^0\Gr^\zeta_{\gamma_\ell}$ inherits a $T(\!(t^\ell)\!)$-action coming from the left multiplication on $\Gr$. The 
cohomology of ${}^0\Gr^\zeta_{\gamma_\ell}$ is $T$-equivariantly formal and satisfies the so-called GKM(=Goresky-Kottwitz-MacPherson) condition. Thus the cohomology ring  $H^\bullet_T({}^0\Gr^\zeta_{\gamma_\ell},\bbC)$
admits a combinatorial description, see Lemma \ref{lem:GKM0}. 
Further, let $W_{\ell,\ex}=W\ltimes \ell\LLbd$ be the extended affine Weyl group, and $W_{\ell,\af}=W\ltimes \ell\LQ$ be the affine Weyl group. The space $H^\bullet_T({}^0\Gr^\zeta_{\gamma_\ell},\bbC)$ has a left $W_{\ell,\ex}$-action, with the action of the lattice $\ell\LLbd$ induced by the 
$T(\!(t^\ell)\!)$-action and the action of the finite Weyl group $W$ induced by some monodromy operators.
For a component of ${}^0\Gr^\zeta_{\gamma_\ell}$ isomorphic to the affine flag variety $\Fl_\ell$, the 
corresponding cohomology also has a right $W_{\ell,\af}$-action given by the Springer correspondence. 

\smallskip

We identify $S$ with the completion of $H^\bullet_T(\pt)$ with respect to the augmentation ideal.
The first main result of our paper is the following.

\begin{theorem-A} [Theorem \ref{prop:P2}] There is an $S$-algebra isomorphism
$$\bfa: H^\bullet_T({}^0\Gr^\zeta_{\gamma_\ell},\bbC)_{\hat 0}\simto Z(\Ut\mod^\LLbd_S).$$
It specializes to an algebra embedding
$H^\bullet({}^0\Gr^\zeta_{\gamma_\ell},\bbC)\hookrightarrow Z(\u_\zeta\mod^\LLbd)$, which restricts to an embedding
$$H^\bullet({}^0\Gr^\zeta_{\gamma_\ell},\bbC)^{\ell\LLbd}\hookrightarrow Z(\u_\zeta)^\LT.$$
\end{theorem-A}

Moreover, the morphism $\bfa$ and its specialization are compatible with the direct sum decomposition of the 
cohomology according to the connected components of $^0\Gr^\zeta_{\gamma_\ell}$ and the decomposition 
of $Z(\u_\zeta\mod^\LLbd)$ according to the blocks.

Further, we conjecture that the embeddings in Theorem A are isomorphisms, see Conjecture \ref{conj-G1T}.

\smallskip

Our second goal is to compare the isomorphism $\bfa$ with an algebra homomorphism from the cohomology of 
$\Gr^\zeta$ to the center of $\U_\zeta$. Indeed, we may identify the $\LG$-invariants in $Z(\u_\zeta)$ as the intersection $Z(\U_\zeta)\cap \u_\zeta$, see \eqref{Ginv}. We have the following theorem.

\begin{theorem-B} [Theorem \ref{thm:main2}]
\begin{enumerate}
\item[$\mathrm{(a)}$] There is an algebra homomorphism
$$\bfc:H^\bullet(\Gr^\zeta,\bbC)\to Z(\U_\zeta),$$
whose image is contained in $Z(\u_\zeta)^\LG$. 

\item[$\mathrm{(b)}$]  There is a commutative diagram of algebra homomorphisms
\begin{align*}
\xymatrix{
H^\bullet(\Gr^\zeta,\bbC)\ar[rr]^{\bfc}
\ar[d]^{i^\ast} &&Z(\u_\zeta)^\LG\ar@{^{(}->}[d]\\
H^\bullet({}^0\Gr^\zeta_{\gamma_\ell},\bbC)^{\ell\LLbd}\ar@{^{(}->}[rr]^-\bfa &&Z(\u_\zeta)^\LT.
}
\end{align*}
All the morphisms are compatible with the block decompositions.
\end{enumerate}
\end{theorem-B}

Note that the image of $i^\ast$ lands in the $W_{\ell,\ex}$-invariant part of $H^\bullet({}^0\Gr^\zeta_{\gamma_\ell},\bbC)$. 
In type $A$, we show that $i^\ast$ maps surjectively onto this part, see Corollary \ref{cor: surjA}. Thus, in this case, the commutative diagram above yields an embedding
$$H^\bullet({}^0\Gr^\zeta_{\gamma_\ell},\bbC)^{W_{\ell,\ex}}\hookrightarrow Z(\u_\zeta)^\LG.$$

We propose the following conjecture (for arbitrary type), which is an extension of a previous conjecture of Bezrukavnikov-Qi-Shan-Vasserot for the principal block case in type A.

\begin{conjecture-A} [Conjecture \ref{conj-G1T} and Conjecture \ref{conj:B}] 
The algebra embedding 
$$\bfa:H^\bullet({}^0\Gr^\zeta_{\gamma_\ell},\bbC)\hookrightarrow Z(\u_\zeta\mod^\LLbd)$$
is an isomorphism. Moreover, it restricts to isomorphisms
\begin{eqnarray*}
H^\bullet({}^0\Gr^\zeta_{\gamma_\ell},\bbC)^{\ell\LLbd}&\simeq& Z(\u_\zeta)^\LT,\\
H^\bullet({}^0\Gr^\zeta_{\gamma_\ell},\bbC)^{W_{\ell,\ex}}&\simeq& Z(\u_\zeta)^\LG.
\end{eqnarray*}
\end{conjecture-A}

\medskip

Our third main result is an explicit formula for the dimension of the left hand side.

\begin{theorem-C}[Theorem \ref{prop:formula}] Denote by $h$ the Coxeter number of $G$, 
by $r$ its rank,  and by $e_1,...,e_r$ the exponents of the Weyl group $W$. Then we have
$$
\dim H^\bullet({}^0\Gr^\zeta_{\gamma_\ell},\bbC)^{W_{\ell,\ex}}=
\frac{1}{|W|}\prod_{i=1}^r\big((h+1)\ell-h+e_i\big).
$$
\end{theorem-C}

In particular, if the Conjecture A is true, this gives an explicit formula for the dimension of $Z(\u_\zeta)^\LG$.
In type A, this formula is compatible with the conjecture in \cite{LQ2}, see Remark \ref{rem:LQ}.

After our paper was written, we were informed by A. Lachowska, N. Hemelsoet and O. Kivinen that they have independently proved Theorem C, partly motivated by Conjecture A that we announced earlier.

\subsection{Ingredients of proofs}

The proof of Theorem A uses some localization properties of the $S$-linear category $\frakU^t_\zeta\mod^\LLbd_S$. 
By \cite{AJS}, the category $\Ut\mod^\LLbd_S$ is generically semi-simple, and when localizing 
$S$ at height one prime ideals generated by the root vectors, it is equivalent to the corresponding category for a rank one 
subgroup of $G$. The center in the rank one situation can be computed explicitly. The center of the category $\Ut\mod^\LLbd_S$ can be obtained as the intersection of their localizations at height one prime ideals inside the generic center. We obtain in this way a GKM-type  description of the center of $\Ut\mod^\LLbd_S$, which matches the GKM-type description of the equivariant 
cohomology of ${}^0\Gr^\zeta_{\gamma_\ell}$

\smallskip

The proof of Theorem B consists of two parts. 
The first one is to construct the morphism $\bfc$, and the second one is to 
show that $\bfc$ fits into the commutative diagram in the theorem. For the construction of $\bfc$, we use an 
algebraic description of the ring $H^\bullet_{\bbG_m}(\Gr^\zeta,\bbC)$ following \cite{BF08}. More precisely, for any closed subscheme $Y$ inside an affine scheme $X$, we denote by $\bbC[\widetilde{N}_Y(X)]$ the Rees algebra associated with the ideal $I_Y$, that is the subring in $\bbC[\hbar^{\pm 1}][X]$ generated by $\bbC[\hbar][X]$ and $\hbar^{-1}I_Y$. We write $\bbC[N_Y(X)]$ for its specialization at $\hbar=0$.
In Proposition \ref{prop:N}, we show that the completion of the ring of $H^\bullet_{\bbG_m}(\Gr^\zeta,\bbC)$ 
with respect to the augmentation ideal is isomorphic to $\bbC[\widetilde{N}_\Omega(T/W)]$ for a certain subscheme $\Omega$ inside $T/W$. Analogous description for the $T$-equivariant cohomology of $\Gr^\zeta$ is also given. Using this description, we show that the 
Harish-Chandra isomorphism extends to an algebra homomorphism from $\bbC[\widetilde{N}_\Omega(T/W)]$ to 
the center of the Lusztig quantum group $\U_{\hat\zeta}$ completed at the parameter $\zeta$. Then specializing the parameter produces a morphism $\bbC[N_\Omega(T/W)]\to Z(\U_\zeta)$. Finally, by analysing the image of certain central elements 
introduced by Drinfeld, we show that this map lands in the $\LG$-invariant part of $Z(\u_\zeta)$. In this way, we get 
the map $\bfc$. There is also a geometric approach for $\bfc$ in the case of principal blocks which is indicated in 
Appendix \ref{appendixB}.

\iffalse%%%%%%%%%%%%%%%%%%
In the appendix we explain a geometric approach to construct the morphism $\bfc$ in the case of principal blocks.
We want to define a morphism
$$\bfc':H^\bullet(\Fl)\to Z(\Rep(\U_\zeta)\!^{\,\hat 0})$$
where $\Fl$ is the affine flag manifold of $G$.
To be able to use Frobenius weights, we consider $\U_\zeta$ over the field $\bar\bbQ_l$ and
$\Fl$ over an algebraic closure of the finite field $\bbF_q$ 
of characteristic $p$ which is large enough and prime to $l$.
Let ${}^0\calI$ be the pro-unipotent 
radical of the Iwahori group $\calI$.
Let $D^b_{m,{}^0\calI}(\Gr)$ be the 
${}^0\calI$-equivariant derived category of mixed complexes of $\bar\bbQ_l$-sheaves
on $\Gr$ and $D^b_{m,IW}(\Fl)$ be the Iwahori-Whittaker derived category of 
mixed complexes of $\bar\bbQ_l$-sheaves
on $\Fl$. 
%Forgetting the mixed structures we get the categories $D^b_{{}^0\calI}(\Gr)$ and $D^b_{IW}(\Fl)$.
According to \cite{BY}, there is an equivalence of categories
$D^b_{m,{}^0\calI}(\Gr)\to D^b_{m,IW}(\Fl).$
It yields a graded vector space homomorphism
$$
H^\bullet(\Fl)\to\Hom\big(\id_{D^b_{m,{}^0\calI}(\Gr)}\,,\,\id_{D^b_{m,{}^0\calI}(\Gr)}(-\bullet\!/2)\,\big).$$
Composing it with the degrading functor 
$D^b_{m,{}^0\calI}(\Gr)\to D^b_{{}^0\calI}(\Gr),$ we get a map
$H^\bullet(\Fl)\to Z(D^b_{{}^0\calI}(\Gr)).$
Then using the equivalence  
$D^b_{{}^0\calI}(\Gr)\simto D^b(\Rep(\U_\zeta)\!^{\,\hat 0})$
in \cite{ABG}, we get the map $\bfc'$.
\fi%%%%%%%%%%%%%%%%%%%%%

The proof of the commutativity of the diagram uses a module category of an hybrid version of 
the quantum group at roots of unity, which provides a bridge between representations of Lusztig's quantum group 
$\U_\zeta$ and the category $\u_\zeta\mod^{\LLbd}$. This hybrid quantum group first appeared in the work of 
Gaitsgory \cite{G18}, in view of extending the Kazhdan-Lusztig equivalence between perverse sheaves on 
$\Gr$ and representations of $\U_\zeta$ to a derived equivalence between
perverse sheaves on $\Fl$ and the category $\mathcal{O}$ for the hybrid quantum group. 
In \S\ref{s:hybrid}, we 
study some basic properties of this category. The main point of this section is the diagram \eqref{diagram}, which 
provides the link between the hybrid category $\calO$, $\u_\zeta\mod^{\LLbd}$ and $\Rep(\U_\zeta)$. Next, we 
give in Proposition \ref{prop:P1} an algebra embedding
$$\bfb: H^\bullet_T(\Gr^\zeta,\bbC)_{\hat 0}\to Z(\calO^{\hyb}_{\zeta,S}).$$
Using the natural embedding of $\Rep(\U_\zeta)$ into $\calO^{\hyb}_{\zeta,S}$ and its left adjoint functor, we produce a morphism 
$Z(\calO^{\hyb}_{\zeta, S})\to Z(\Rep(\U_\zeta))$. In Proposition \ref{prop:P00}, we show that $\bfb$, composed with 
this morphism, is equal to $\bfc$ composed with the specialisation map on 
the cohomology. In some sense, the category $\calO^{\hyb}_{\zeta, S}$ plays the role of a deformation of (an enlargement of) the category 
$\Rep(\U_\zeta)$. This fact allows us in \S\ref{sec:main}
to make the link to the map $\bfa$ constructed by a deformation argument.

\smallskip

For Theorem C, we first obtain an upper bound of the dimension for 
$H^\bullet(\Fl_\gamma,\bbC)^{W_\ex}$, which 
corresponds to the center of a principal block, see Proposition 2.9. The inequality is an equality in type $A$. 
Using this upper bound, a theorem of Boixeda-Alvarez and Losev yields an isomorphism of
right $W$-modules $H^\bullet(\Fl_\gamma,\bbC)^{W_\ex}=\bbC[\LLbd/(h+1)\LLbd]$. 
Then, using a formula of Sommers, we compare $H^\bullet(\Fl_\gamma,\bbC)^{W_\ex}$ and the cohomology of 
an elliptic affine Springer fibre, as well as their parabolic versions involving
affine Spaltenstein varieties to obtain the main result.

\subsection{Perspectives}
Most of the results in this paper have counterparts in positive characteristic. 
Let $\bar\bbF_p$ be the algebraic closure of the finite field $\bbF_p$ for a good prime $p$. 
We replace the small quantum group $\u_\zeta$ by the restricted enveloping algebra $\u_p$ of $\LG$ 
over $\bar\bbF_p$, replace the category $\u_\zeta\mod^\LLbd$ by the category 
$\u_p\mod^\LLbd$ of $\LG_1\LT$-modules, and replace
$\Ut\mod^\LLbd_S$ by the $S$-deformed version $\frakU^t_p\mod^\LLbd_S$
of $\u_p\mod^\LLbd$ defined in \cite{AJS}, 
where $S$ is the completion of $\bar\bbF_p[\frakt]$ at zero.  
Then:

\begin{enumerate}
\item there is an $S$-algebra isomorphism
$$\bfa: H^\bullet_T({}^0\Gr^{\mu_p}_{\gamma_p},\bar\bbF_p)_{\hat 0}\simto Z(\frakU^t_p\mod^\LLbd_S)$$
which specializes to an inclusion
$H^\bullet({}^0\Gr^{\mu_p}_{\gamma_p},\bar\bbF_p)\hookrightarrow Z(\u_p\mod^\LLbd)$;

\item there is an algebra homomorphism
$$\bfc:H^\bullet(\Gr^{\mu_p},\bar\bbF_p)\to Z(\text{Dist}(\LG_{\bar\bbF_p})),$$
where $\text{Dist}(\LG_{\bar\bbF_p})$ is the distribution algebra of $\LG_{\bar\bbF_p}$;

\item there is a positive characteristic analogue of the map $\bfb$ above,
so that one can check a positive characteristic analogue of Theorem B.
\end{enumerate}

We'll come back to this elsewhere.

\smallskip

The results proved in this paper suggest that the Koszul dual of the category of $\u_\zeta$-modules should have a 
geometric construction in terms of microlocal perverse sheaves attached to the Springer fiber.
This is the subject of an on going project of the first two authors with M. McBreen and Z. Yun.

 \medskip

 \subsection{Notation}
 Let $\bfk$ be an algebraically closed field of characteristic zero. 
 Unless stated otherwise, we'll assume $\bfk=\bbC$.
 For a $\bfk$-algebra $A$ let $Z(A)$ be its center,
 and $A\Mod$, $A\mod$ be the categories of all left 
 $A$-modules and all finitely generated ones.  
 The center $Z(\calC)$ of a category $\mathcal{C}$ is the endomorphism ring of the identity functor. Basic properties of the center are recalled in Appendix A. 

Given a lattice $\Lbd$, 
 a $\Lbd$-graded $\bfk$-vector space is a $\bfk$-vector space $M$ equipped with a decomposition 
$M=\bigoplus_{\mu\in\Lbd} M_\mu$. 
For any $\Lbd$-graded $\bfk$-algebra $A$, 
by a $\Lbd$-graded $A$-module we mean an $A$-module in the category of
$\Lbd$-graded $\bfk$-vector spaces. Let $A\mod^\Lbd$ be the category of $\Lbd$-graded finitely 
generated $A$-modules.
For any set $X$, let $\Fun(X,A)=\prod_{x\in X}A$ be the ring 
whose product is the point-wise multiplication of functions.

For any complex variety $X$ with an action of a complex linear group $G$,
 let $H_G^\bullet(X)$ and $H^G_\bullet(X)$  be the $G$-equivariant cohomology
 and $G$-equivariant Borel-Moore homology of $X$ with coefficient in $\bfk$. 
 We write $H^\bullet_G$ for the $G$-equivariant cohomology of a point. 
 Let $R_G$ be the representation ring of $G$ tensored with $\bfk$.
We have obvious identifications $H^\bullet_{\bbG_m}=\bfk[\hbar]$
and $R_{\bbG_m}=\bfk[q^{\pm 1}]$, where $q$ is the linear character and $\hbar$ is its first Chern class.
For any commutative ring $R$ let $\Vect(R)$ be the category consisting of the free $R$-modules of finite rank.

Let  $e$ be a positive integer. 
Let $q_e=q^{1/e}$ be a formal variable, and $\zeta$ a primitive $\ell$-th root of unity. 
Fix $\zeta_e=\zeta^{1/e}$ in $\bfk^\times$.
 We write 
 $\calA=\bfk[\hbar],$
 $\bfA=\bfk[q_e^{\pm 1}],$
 and $\bfF=\bfk(q_e).$ 
 We may identify the ring $\calA$ with $\bfk[\bbG_a]$ and view
 $\bfA$ as an extension of the ring $\bfk[\bbG_m]$.
Let $\bfA_{\hat\zeta}$ be the completion of $\bfA$
at the point $q_e=\zeta_e$, and 
$\calA_{\hat 0}$ be the completion of $\calA$ at $\hbar=0$.
The assignment $q_e=\zeta_e+\hbar$ yields canonical isomorphisms 
$\bfA_{\hat\zeta}=\calA_{\hat 0}=\bfk[[\hbar]].$
Set
$\calK_\ell=\bbC((t^\ell))$ and
$\calO_\ell=\bbC[[t^\ell]].$
 We abbreviate
 $\calK=\calK_1$ and
$ \calO=\calO_1$.

 \subsection{Acknowledgement}
 The initial stage of this project was based on several discussions with Qi You. We would like to thank him for all the helpful discussions.

\bigskip

\section{Cohomology of $\zeta$-fixed points in affine Grassmannians}

\subsection{Affine Grassmannians and $\zeta$-fixed points}

\subsubsection{Root data}\label{sec:RD}
Let $G$ be a connected complex reductive algebraic group over $\bbC$, with maximal torus $T$, 
Borel subgroup $B$, opposite Borel subgroup $B^-$ and root datum
$(\X^\ast(T), \X_\ast(T), \Phi, \LPhi)$.
Let $N$, $N^-$ be the unipotent radical of $B$, $B^-$. 
Let $r$ be the rank of $G$ and $h$ its Coxeter number.
Let $G_\sc$ be the simply connected cover of $G$ and $T_\sc$ its maximal torus.
Let $\LG$ be the Langlands dual  group of $G$, defined over the field $\bbC$. 
Let $\LT$ be the dual torus, and $\LB$ the corresponding Borel subgroup.
Let $\frakg$, $\frakb$, $\frakt$ be the Lie algebras of $G$, $B$, $T$. 
Let $\Phi_\af$ and $\Phi_\re=\Phi\oplus\bbZ\delta$ be the sets of affine roots and affine real roots. 
Let $\Sigma$ and $\Sigma_\af=\Sigma\sqcup\{\alpha_0\}$ be the sets of simple finite and affine roots.
Let $\Lalpha$ denote the affine coroot associated with the affine root $\alpha$.
We abbreviate $\LLbd=\X_\ast(T)$, $\Lbd=\X^\ast(T)$ and $Q=\bbZ\Phi$, $\LQ=\bbZ\LPhi$.
Let $\langle-,-\rangle$ be the canonical pairing on $\Lbd\times\LLbd$.
Set also $P=\LQ^*$ and $\LP=Q^*$.
Recall that $Q\subset\Lbd\subset P$ and $\LQ\subset\LLbd\subset\LP$.
Let $e$ be the order of the group $\LLbd/\LQ$.
Let $Z=Z(G)$, $\LZ=Z(\LG)$ be the centers
and $\pi_1=\pi_1(G)$, $\Lpi_1=\pi_1(\LG)$ the fundamental groups.
We have
$$\pi_1=\X^\ast(\LZ)=\LLbd/\LQ=P/\Lbd,\qquad\Lpi_1=\X^\ast(Z)=\Lbd/Q=\LP/\LLbd.$$

For any commutative ring $R$ we abbreviate $\frakt_R=\LLbd\otimes_\bbZ R$ and
$\frakt^*_R=\Lbd\otimes_\bbZ R$. 
We fix a $W$-invariant positive definite symmetric bilinear form $(-,-)$ on $\frakt_\bbR$
such that $(\Lalpha,\Lalpha)=2$ for each short coroot.
We have $(\LLbd,\LLbd)\subset \frac{1}{e}\bbZ$ and $(\LLbd,\LQ)\subset \bbZ$.
We may identify $\alpha=2\Lalpha/(\Lalpha,\Lalpha)$ for each root $\alpha$.
We write $\Sigma=\{\alpha_i\}$.
Let $a_{ij}$ be the $(i,j)$-th entry of the Cartan matrix of $G$.
For any root $\alpha$ we set
$d_\alpha=(\alpha,\alpha)/2$.
%We abbreviate $d=d_\alpha$ if $\alpha$ is long.
%So $d=\max\{d_{\alpha_i}\}=1,2$ or $3$.
We have $\alpha_i=d_{\alpha_i}\Lalpha_i$ and $(\alpha_i,\alpha_j)=d_{\alpha_i}a_{ij}$.
Let $\Lrho$ be half the sum of the positive coroots,
$\delta$ the minimal positive imaginary root, and $\theta$ the highest root.
%Set $c=\Lalpha_0+\check\theta$.

Let $W$ be the Weyl group. Let $W_\af=W\ltimes\LQ$ and
$W_\ex=W\ltimes\LLbd$ be the affine Weyl group and the extended one.
For each $\mu\in\LLbd$, we write $\tau_\mu=1\ltimes\mu$.
The group $W_\ex$ acts on the lattice $\LLbd\oplus\bbZ \partial$ by
\begin{align}\label{Wact}w(\lambda+m \partial)=w(\lambda)+m \partial,\qquad
\tau_\mu(\lambda+m \partial)=\lambda+m\mu+m\partial.\end{align}
for $w\in W$ and $\lambda,\mu\in\LLbd$.
The dual action on $\Lambda\oplus\bbZ\delta$ is such that
\begin{align}\label{WactD}w(\beta+k\delta)=w(\beta)+k\delta,\qquad
\tau_\mu(\beta+k\delta)=\beta+k\delta-\langle\beta,\mu\rangle\delta,
\end{align}
for $w\in W$, $\beta\in\Lambda$, and $\mu\in\LLbd.$

The group $W$ acts on the torus $T$, hence on $\bbC[T]$.
When viewed as a function on  $T$, 
we'll denote the character $\beta\in\Lambda$ by the symbol $e^\beta$.
Similarly, for each cocharacter $\lambda\in\LLbd$ and each $z\in\bbC^\times$
we write $z^\lambda$ for the element $\lambda(z)$ in $T$.
We have 
$$\bbC[T\times\bbG_m]=\bbC[q^{\pm 1}][e^{\beta+m\delta}\,;\,\beta\in\Lbd\,,\,m\in\bbZ]\,/\,(q-e^{\delta}),
\qquad
\X^*(\bbG_m)=\bbZ\delta.$$
The $W_\ex$-action on $\Lambda\oplus\bbZ\delta$ yields a $W_\ex$-action on $T\times\bbG_m$ 
such that 
\begin{align}\label{WactT}w(t,z)=(w(t),z),\qquad
\tau_\mu(t,z)=(tz^\mu,z).\end{align}
for $w\in W$, $\mu\in\LLbd$, $t\in T$, and $z\in\bbC^\times$.
Let $(x,f)\mapsto {}^x\!f$ be the induced action on the algebra $\bbC[T\times\bbG_m]$.

Under the obvious identification $\LLbd\simeq\LLbd+\partial$ the group $W_\ex$ acts $\LLbd$ by
\begin{align}\label{Wact1}w(\lambda)=w(\lambda),\qquad
\tau_\mu(\lambda)=\lambda+\mu.\end{align}
for $w\in W$, $\lambda,\mu\in\LLbd$.
For each affine real root $\alpha+m\delta$ let $s_{\alpha+m\delta}$ 
be the reflection in $W_\af$ relatively to the affine hyperplane
$$H_{\alpha+m\delta}=\{\lambda\in \frakt_\bbR\,;\, \langle\alpha,\lambda\rangle+m=0\}.$$
We have $s_{\alpha+m\delta}=s_\alpha\tau_{m\Lalpha}$ and
$s_{\alpha+m\delta}(\lambda)=\lambda-(\langle\alpha,\lambda\rangle+m)\Lalpha$
for each $\lambda\in\LLbd$.
The $\bullet$-action of $W_\ex$ on $\LLbd$ is
given by $x\bullet\lambda=x(\lambda+\Lrho)-\Lrho$ for each $x\in W_\ex$.

The set of dominant coweights and restricted dominant coweights are
\begin{align*}
&\LLbd^+=\{\lambda\in\LLbd\,;\,0\leqslant\langle\alpha,\lambda\rangle,\,\forall\alpha\in\Sigma\},\\
&\LLbd_\ell=\{\mu\in\LLbd\,;\,0\leqslant\langle\alpha,\mu\rangle<\ell,\,\forall\alpha\in\Sigma\}.
\end{align*}
For each $x\in W_\af$ the coweight $x\bullet 0$ is dominant if and only if  $x$ is
minimal in the coset $W\setminus W_\af$.
For each $\Lalpha\in\LQ$ let $w_\Lalpha$
be the minimal element in the coset $W\tau_\Lalpha$.
We get a bijection 
$$\LQ\to\LLbd^+\cap W_\af\bullet 0,\qquad
\Lalpha\mapsto w_\Lalpha\bullet 0=w\bullet\Lalpha,$$ where $w\in W$ is such that
$w\bullet\Lalpha$ is dominant.

For each affine real root $\alpha+m\delta$ let $U_{\alpha+m\delta}$ 
be the unipotent root subgroup in $G(\calK)$
associated with the affine real root $\alpha+m\delta$.
It is the image of the group homomorphism
$$\bbG_a\to G(\calK),\qquad z\mapsto n_{\alpha+m\delta}(z):=n_\alpha(zt^m),$$
where $n_\alpha:\bbG_a\to U_\alpha$ is a fixed group isomorphism.
Let $\fraku_\alpha$ be the Lie algebra of $U_\alpha$ and $e_\alpha=d_1n_\alpha(1)$.

\smallskip

\subsubsection{Affine flag manifolds}
Let $\Fl= G(\calK )/\calI$ denote the affine flag manifold, with $\calI$ the usual Iwahori subgroup.
For any subset $J\subset \Sigma_\af$ let $P^J\subset G(\calK)$ be the standard parahoric subgroup
of type $J$ and $\frakp^J$ its Lie algebra. 
Let ${}^0P^J$ be the pro-unipotent radical of $P^J$ and ${}^0\frakp^J$ its Lie algebra.  
Let $G_J$ be the Levi component and $\frakg_J$ its Lie algebra.
Let $\Phi_\re^J$ and $\Phi_{\re,J}$ be the sets of all real affine roots which occur in
${}^0\frakp^J$ and in $\frakg_J$ respectively.
Let $W_J\subset W_\af$ be the subgroup generated by $J$ and 
$W^J_\af=W_\af/W_J$, $W^J_\ex=W_\ex/W_J$.
The partial affine flag manifold of type $J$ is the fpqc quotient $\Fl^J=G(\calK )/P^J$. 
It is an ind-scheme which is studied in \cite{PR}.
The connected components of $\Fl^J$ are parametrized by the algebraic fundamental group 
$\pi_1$ of $G$. We have canonical identifications
$$\pi_1=N_{G(\calK)}(\calI)/\calI,
\qquad W_\ex=W_\af\rtimes\pi_1$$
where $\pi_1$ acts on $W_\af$ by Dynkin diagram automorphisms.
Let $\Fl^{J,\circ}$ be the connected component of the based point $P^J/P^J$.
The set of fixed points for the left $T$-action on $\Fl^J$ is 
$$(\Fl^J)^T=\{\delta_x\,;\,x\in W^J_\ex\},\qquad \delta^J_x=xP^J/P^J.$$
If no confusion is possible we abbreviate $\delta_x=\delta^J_x$.
We may also abbreviate $\delta_\mu=\delta_{\tau_\mu}$ for each weight
$\mu\in\LLbd$.
If $J=\Sigma$ then $\Fl^J=\Gr$ is the affine Grassmannian and the set
$W^J_\ex$ is identified with the
lattice $\LLbd$. 
The set of $T$-fixed points in $\Gr$ is 
$$\Gr^T=\{\delta_\mu\,;\, \mu\in\LLbd\},\qquad
\delta_\mu=t^\mu G(\calO)/G(\calO).$$ 

\smallskip

\subsubsection{The $\ell$-affine Weyl group}\label{sec:Xi}
Fix an odd integer $\ell\geqslant h$ which is prime to $e$ and to $3$ if $G$ contains a component of type $G_2$. 
Let $\mu_\ell\subset\bbC^\times$ be the set of $\ell$-th roots of unity and $\zeta\in\mu_\ell$ a primitive element.
Let $W_{\ell,\af}=W\ltimes \ell\LQ$ and 
$W_{\ell,\ex}=W\ltimes \ell\LLbd$ be the $\ell$-affine Weyl group and
the $\ell$-extended Weyl group.
Let $\Sigma_{\ell,\af}=\{s_{\alpha+\ell m\delta}\,;\, s_{\alpha+m\delta}\in\Sigma_\af\}$.
The fundamental $\ell$-alcove in $\frakt_\bbR$ is
$$A_{\ell,1}=\{\mu\in\frakt_\bbR\,;\, 0< \langle\alpha,\mu\rangle< \ell\,,\,\forall\alpha\in\Phi^+\}.$$
Let $\bar A_{\ell,1}$ be its closure. It is a fundamental domain for the action of $W_{\ell,\af}$ on $\frakt_\bbR$.
The facet of type $J$ in $\bar A_{\ell,1}$ is the subset
$$F_J=\bar A_{\ell,1}\cap\bigcap_{s_{\alpha+m\delta}\in J}H_{\alpha+\ell m\delta}.$$
Let $W_{\ell,J}$ be the subgroup of $W_{\ell,\af}$ generated by $J$.
 We write $W_{\ell,\af}^{J}=W_{\ell,\af}/W_{\ell,J}$ and $W_{\ell,\ex}^{J}=W_{\ell,\ex}/W_{\ell,J}$.
 We'll abbreviate $\ell\LLbd=W_{\ell,\ex}^\Sigma.$
We consider the $\bullet$-action of $W_{\ell,\ex}$ on $\LLbd$.
Let 
$$\Xi_\sc=\LLbd/W_{\ell,\af},\qquad\Xi=\LLbd/W_{\ell,\ex}.$$
Since $\bar A_{\ell,1}-\Lrho$ is a fundamental domain for the $\bullet$-action of $W_{\ell,\af}$ on $\frakt_\bbR$,
each coset in $\Xi_\sc$ is uniquely represented by an element in the set
$(\bar A_{\ell,1}-\Lrho)\cap\LLbd$. Further, we have
$W_{\ell,\ex}=W_{\ell,\af}\rtimes\pi_{\ell,1}$ where $\pi_{\ell,1}\subset W_{\ell,\ex}$ is the subgroup
stabilizing $(\bar A_{\ell,1}-\Lrho)\cap\LLbd$. The group $\pi_{\ell,1}$ is canonically identified with
$\pi_1$, yielding a $\pi_1$-action
on $\Xi_\sc$ such that $\Xi=\Xi_\sc/\pi_1$.
Since the integer $\ell$ is prime to $e$, an easy computation yields 
\begin{align}\label{PQ}\ell\LQ=\LQ\cap\ell\LLbd.\end{align}
Hence, the $\pi_1$-action on $\Xi_\sc$ is free.
Given $\omega$ in $\Xi_\sc$,
let $J_\omega$ be the set of affine reflections $s\in\Sigma_{\ell,\af}$ such that  $s\bullet\omega=\omega$,
where $\omega$ is viewed as an element in $\LLbd$.
We abbreviate
 $W_{\ell,\omega}=W_{\ell,J_\omega}$ and
 $W_{\ell,\af}^\omega=W_{\ell,\af}^{J_\omega}$.
 We have
 $$W_{\ell,\omega}=\{x\in W_{\ell,\af}\,;\,x\bullet\omega=\omega\}.$$
 We define
 $$W_{\zeta^\omega}=\{w\in W\,;\,w\bullet\zeta^\omega=\zeta^\omega\}.$$
 Using \eqref{PQ}, the obvious group homomorphism $W_{\ell,\ex}\to W$ implies that
 \begin{align}\label{Womega}
 W_{\ell,\omega}=\{x\in W_{\ell,\ex}\,;\,x\bullet\omega=\omega\}\simeq 
 W_{\zeta^\omega}.
 \end{align}
 
 \smallskip

 \subsubsection{Fixed points}
 For any subset $J\subset \Sigma_\af$, let $\Fl^J_\ell$ be the 
partial affine flag manifold with $t$ replaced by $t^\ell$.
We have $\Fl^J_\ell=G(\calK_\ell)/P^J_\ell$
where $P^J_\ell\subset G(\calK_\ell)$ is the parahoric subgroup of type $J$.
Note that 
\begin{align}\label{para}U_{\alpha+\ell m\delta}\subset P^J_\ell\iff
F_J\subset H_{\alpha+\ell m\delta}.\end{align}
For $\omega \in\Xi_\sc$ we abbreviate
 $$
 \Fl^\omega_\ell=\Fl^{J_\omega}_\ell,\qquad
 P^\omega_\ell=P^{J_\omega}_\ell,\qquad
 \frakp^\omega_\ell=\frakp^{J_\omega}_\ell,\quad\text{etc.}
 $$
Note that $J_{\omega\gamma}=\gamma(J_\omega)$ for any element $\gamma\in\pi_1$, where
$\gamma$ acts on $\Sigma_\af$ through an automorphism of the affine Dynkin graph. 
Hence, the partial affine flag manifold $\Fl^\omega_\ell$ depends only on the class of 
$\omega$ in the set $\Xi=\Xi_\sc\,/\,\pi_1$.
Let $\Gr^\zeta$ be the set of fixed points of the $\mu_\ell$-action on $\Gr$ by loop rotation.
 
\begin{lemma}\label{lem:isom1}
We have 
$\Gr^\zeta =\bigsqcup_{\omega\in \Xi} \Fl^\omega_\ell$ and
$\pi_0(\Gr^\zeta)=\Xi_\sc.$
\end{lemma}

\begin{proof}
Recall that an element $\omega$ of $\Xi_\sc$ is identified with an element in $\LLbd$.
The adjoint action by the element $t^{\omega}$ of $G(\calK)$ on $G(\calO)$ 
yields a well-defined subgroup in $G(\calK)$.
From \eqref{para} we deduce that $U_{\alpha+\ell m\delta}\subset \Ad_{t^{-\omega}}(G(\calO))$
if and only if $U_{\alpha+\ell m\delta}\subset P^\omega_\ell$.
Hence, we have
\begin{align}\label{parahoric}
G(\calK_\ell)\cap\Ad_{t^{-\omega}}(G(\calO))=P^\omega_\ell.
\end{align}
Hence, we have an isomorphism, see, e.g., \cite[\S 4]{RW} for details,
$$\bigsqcup_{\omega\in \Xi} \Fl^\omega_\ell\to\Gr^\zeta,\qquad
g\delta^\omega_0\mapsto g\delta_\omega,\qquad
g\in G(\calK_\ell).$$
For the second formula, recall that the connected components of
$\Fl^\omega_\ell$ are parametrized by the fundamental group $\pi_1$.
\end{proof}

\medskip

 \subsection{The equivariant cohomology of affine Grassmannians}

Given a $G$-variety $X$ over $\bbC$,
the equivariant cohomology group $H^\bullet_G(X)$ is a $\bfk$-algebra.
Its spectrum is a $\bfk$-scheme. We'll relate it to the maximal torus $T$, viewed as an algebraic group over the 
field $\bfk$.

\subsubsection{The deformation to the normal cone}
Let $A$ be a commutative $\bfk$-algebra and let $I$ be an ideal in $A$. 
Let $\hbar$ be a formal variable. 
We define
\begin{align*}
&\frakR_I(A)=\sum_{n\geqslant 0}A[\hbar](\hbar^{-1}I)^n\subset A[\hbar^{\pm 1}],
\\
&\frakB_I(A)=\frakR_I(A)/\hbar\,\frakR_I(A)=\sum_{n\geqslant 0}I^n/I^{n+1},\text{ with }
I^0=A.
\end{align*}
For $X=\Spec(A)$ and $Y=\Spec(A/I)\subset X$, we write
$$N_Y(X)=\Spec(\frakB_I(A)),\qquad  \widetilde{N}_Y(X)=\Spec(\frakR_I(A)).$$
The scheme $\widetilde{N}_Y(X)$ is 
the deformation of $X$ to the normal cone of $Y$.
The scheme $N_Y(X)$ is the normal cone. 
Note that  $\widetilde{N}_Y(X)$ is a scheme over $X\times\bbG_a$ and $N_Y(X)$ is its fiber at $0\in\bbG_a$.

\smallskip

\iffalse
The set $\Xi_\sc=\bar A_\ell\cap\LLbd-\Lrho$ is a fundamental domain, i.e., 
we have $\Xi_\sc=\LLbd/W_{\ell,\af}$.
It is the set of $\bbC$-points of the following scheme
$$\pmb\Om=\{1\}\times_{T/W}T_\sc/W,$$
where right morphism is the composition of the obvious projection $T_\sc/W\to T/W$
and the map 
$$[\ell]:T/W\to T/W,\qquad t\,\text{mod}\,W\mapsto t^\ell\,\text{mod}\,W,\qquad t\in T.$$
We also consider the scheme $\Omega$ such that
$$\Xi=\LLbd/W_{\ell,\ex},\qquad\Omega=\{1\}\times_{T/W}T/W.$$
Thus $\Om$ is the scheme-theoretic preimage of $\{1\}$ by the map $[\ell]$.
The scheme $\Om$ is not connected. Let $\Om_\omega$ be the connected 
component containing the closed point $\omega\in\Omega(\bbC)$.
We have $\Om=\bigsqcup_{\omega\in\Omega(\bbC)}\Omega_\omega$.
 Note that $\Om=\pmb\Om/\pi_1$.
 For each point $\omega\in\Xi$ represented by an element of $\bar A_\ell$, let
 $J_\omega$ be the set of affine reflections in $\Sigma_{\ell,\af}$ which fix $\omega$.
 For instance, we have $J_0=\emptyset$.
 We abbreviate 
 $$W_{\ell,\omega}=W_{\ell,J_\omega},\qquad
 W_{\ell,\af}^\omega=W_{\ell,\af}^{J_\omega},\qquad
 \Fl^\omega_\ell=\Fl^\omega_\ell,\qquad
 P^\omega_\ell=P^{J_\omega}_\ell,\qquad
 \frakp^\omega_\ell=\frakp^{J_\omega}_\ell,\quad\text{etc.}
 $$
\fi

\subsubsection{The schemes $\Omega$ and $\Gamma$}
Consider the $\bfk$-scheme homomorphism
$$[\ell]:T/W\to T/W,\qquad t\,\text{mod}\,W\mapsto t^\ell\,\text{mod}\,W,\qquad t\in T.$$
Let $\Omega$ be the scheme-theoretical preimage of the closed point $1\in T/W$ under the map $[\ell]$. 
Let $\Omega_\sc=\Omega\times_{T/W}T_\sc/W$, where $T_\sc\to T=T_\sc/\pi_1$ is the obvious projection. 
We have $\Omega\simeq \Omega_\sc\,/\,\pi_1$.
Let $\bfC=\bfk[\frakt]/\bfk[\frakt]^W_+$ be the coinvariant ring.
By \eqref{Womega} the group $W_{\ell,\omega}$ embeds into $W$. Hence, it acts on the ring $\bfC$.
We define   
$\Omega_\omega= \Spec(\bfC^{W_{\ell,\omega}})$.
We have the following, see, e.g., \cite[thm.~4.1]{L03}.

\begin{lemma}
\hfill
\begin{itemize}[leftmargin=8mm]
\item[$\mathrm{(a)}$] 
$\Omega(\bfk)=\{W(\zeta^{2(\omega+\Lrho)})\,;\,\omega\in\Xi\}\simeq\Xi$.
\item[$\mathrm{(b)}$] 
$\Omega=\bigsqcup_{\omega\in\Xi}\Omega_\omega$.
\end{itemize}
\end{lemma}

\iffalse%%%%%%%%%%%%%%%%%%%%%%%%
\begin{proof}
The morphism $[\ell]$ is finite. The scheme $\Omega$ is non-reduced of dimension zero. 
The closed points in $T$ are of the form $z^\lambda$, with $\lambda\in \LLbd$ and $z\in\bfk^\times$. 
The $W$-orbits of $\zeta^\lambda$ and $\zeta^\mu$ coincide 
if and only if $\lambda$ and $\mu$ belong to the same $W_{\ell,\ex}$-orbit on $\LLbd$, proving (a).
Let $\Omega_\omega$ be the connected 
component containing the closed point 
$$\zeta^{2(W_{\ell,\ex}\bullet\omega+\Lrho)}=W(\zeta^{2(\omega+\Lrho)})\in T(\bfk)/W.$$
The scheme $\Omega_\omega$ is also the fibre at $1$ of the restriction of the map $[\ell]$ to the 
formal neighbourhood of $W(\zeta^{2(\omega+\Lrho)})$ in $T/W$. We have a map
$$(T/W)_{\widehat{W(\zeta^{2(\omega+\Lrho)})}}\to (T/W)_{\hat 1}.$$
Under the canonical isomorphism of the left hand side with  $(\frakt/W_{\ell,\omega})_{\hat 0}$
and the right hand side with
$(\frakt/W)_{\hat 0}$, this map is identified with the obvious map
$(\frakt/W_{\ell,\omega})_{\hat 0}\to (\frakt/W)_{\hat 0}$, and $\Omega_\omega$ is the fibre at zero.
Thus, we have $\bfk[\Omega_\omega]=\bfC^{W_{\ell,\omega}}$.
\end{proof}
\fi%%%%%%%%%%%%%%%%%%%%%%%%%%%

We consider the fibre product of schemes
$$\Gamma=T\times_{T/W}T/W$$
relative to the map $T\to T/W$, $t\mapsto t^\ell\,\text{mod}\, W$ and the map 
$[\ell]:T/W\to T/W$, $t\,\text{mod}\,W\mapsto t^\ell\,\text{mod}\, W$.
We have
$$\Gamma=\Spec\Bigl(\bfk[\LLbd]\otimes_{\bfk[\ell\LLbd]^W}\bfk[\LLbd]^W\Bigr).$$
The set $\Gamma(\bfk)$ can be described as follows: consider the union of graphs of the 
elements in $W$ in $T^2$, take its preimage under the map $T^2\to T^2$ given by $(z_1,z_2)\mapsto (z_1^\ell, z_2^\ell)$,
then $\Gamma(\bfk)$ is the image of this preimage in $T(\bfk)\times(T(\bfk)/W)$.
The formal neighborhood $\Gamma_{\hat 1}$ of $\Gamma$
along the fibre at $1\in T$ is a disjoint union 
\begin{align}\label{Gamma}\Gamma_{\hat 1}=\bigsqcup_{\omega\in\Xi}\Gamma_\omega.
\end{align}
Each connected component
$\Gamma_\omega$ is isomorphic to $\mathfrak{t}\times\Omega_\omega$.
Let 
$$I_{\Om}\subset \bfk[T/W],\qquad I_{\Gamma}\subset\bfk[T\times T/W],$$
be the vanishing ideals of $\Omega$ and $\Gamma$.
Let 
$$\bfk[\widetilde{N}_{\Gamma}(T\times T/W)]_{\hat 0},\qquad
\bfk[\widetilde{N}_{\Gamma}(T\times T/W)]_{\widehat{1,0}},\qquad
\bfk[N_{\Gamma}(T\times T/W)]_{\hat 1},$$ 
be respectively the completions of the 
$\calA[T]$-algebra $\bfk[\widetilde{N}_{\Gamma}(T\times T/W)]$ at the point $0\in\bbG_a$, 
its completion at $(1,0)\in T\times\bbG_a$, and the completion of the $\bfk[T]$-algebra
$\bfk[N_{\Gamma}(T\times T/W)]$ at the point $1\in T$.

\smallskip

\subsubsection{The equivariant $K$-theory of the affine Grassmannian}
We first recall a few facts about the equivariant $K$-theory of the affine Grassmannian.
All results here are standard, and can be found in \cite{LSS} for instance.
For $H=T\times\bbG_m$, $T$ or $\bbG_m$, let
$K^0_H(\Gr)$ be the $H$-equivariant $0$-th $K$-cohomology group 
of the affine Grassmannian tensored with $\bfk$.
It can be realized as the limit  in the category of 
$R_H$-modules of the $\bfk$-linear Grothendieck groups
$K_0^H(U)$ of $H$-equivariant coherent sheaves over $U$, where $U$ runs over the set of all finite unions of 
$G(\calO)$-orbits in the thick affine Grassmannian of $G$. 
As a module over the ring $R_H$, we have
\begin{align}\label{topofree}
K^0_H(\Gr)=\prod_{x\in \LLbd}R_H\cdot[\calO^x],
\end{align}
where $\calO^x$ is the structural sheaf of the finite codimensional Schubert variety labelled by $x$.
Recall that $\LLbd=W_\ex/W$.
For a future use, we allow $e$-th roots of the element $q$ in $\bbC[T\times\bbG_m]$, and we define
\begin{align}\label{free}\begin{split}
K_{T\times\bbG_m}(\Gr)&=\bigoplus_{x\in \LLbd}\bfA[T]\cdot[\calO^x]\subset 
K^0_{T\times\bbG_m}(\Gr)\otimes_{R_{T\times\bbG_m}}\bfA[T]
,\\
K_{T}(\Gr)&=\bigoplus_{x\in \LLbd}\bfk[T]\cdot[\calO^x]\subset
K^0_{T}(\Gr)\otimes_{R_T}\bfk[T]
\end{split}
\end{align}

The $W_\ex$-action on the ring $\bbC[T\times\bbG_m]$ in \eqref{WactT} extends to 
a $W_\ex$-action on $\bfA[T]$.
The $G(\calK)$-action on the affine Grassmannian by left multiplication yields a $W_\ex$-action on 
the space $K_{T\times\bbG_m}(\Gr)$.
This action is called the left action. 
It is compatible with the $W_\ex$-action on the ring $\bfA[T]$.
The restriction to the $T\times\bbG_m$-fixed point $\delta_x\in\Gr$ yields a 
$W_\ex$-equivariant map
$$\res_x:K_{T\times\bbG_m}(\Gr)\to\bfA[T],\qquad
\forall\ x\in\LLbd.$$
In the rest of this section, except in Proposition \ref{prop:N}, we'll assume that the group $G$ is simply connected.
By Kostant-Kumar, taking the product over all $x$'s yields an 
$\bfA[T]$-algebra embedding
\begin{align}\label{pi1}K_{T\times\bbG_m}(\Gr)\to\Fun(\LLbd\,,\,\bfA[T]),\end{align}
which takes the class $[\calO^0]$ to the constant function $1$.
We view the set of maps $\Fun(\LLbd\,,\,\bfA[T])$ as an 
$(\bfA[T]\,,\,\bfA[T/W])$-bimodule such that
$$(f\cdot \psi)(x)=\psi(x)f,\qquad (\psi\cdot g)(x)=\psi(x)\,{}^x\! g,$$
for $x\in \LLbd$, $f\in \bfA[T]$, $g\in\bfA[T/W]$, $\psi\in\Fun(\LLbd\,,\,\bfA[T])$.
The image of the map \eqref{pi1} is an $(\bfA[T]\,,\,\bfA[T/W])$-subbimodule containing the constant function 1. 
We consider the algebra homomorphism
\begin{align}\label{gamma}
\bfA[T\times T/W]\to K_{T\times\bbG_m}(\Gr),\quad f\otimes g\mapsto f\cdot 1\cdot g.
\end{align}
Forgetting the $T$-action we also have a restriction map
$K_{\bbG_m}(\Gr)\to\Fun(\LLbd\,,\,\bfA)$
and an algebra homomorphism
$\bfA[T/W]\to K_{\bbG_m}(\Gr)$.
Let $K_{T\times\bbG_m}(\Gr)_{\hat\zeta}$ and $K_{\bbG_m}(\Gr)_{\hat\zeta}$ 
be the completions of the $\bfA$-algebras
$K_{T\times\bbG_m}(\Gr)$ and $K_{\bbG_m}(\Gr)$ at the point $\hbar=0.$

\smallskip

\begin{lemma} \label{lem:extension}
Let $G$ be simply connected. \hfill
\begin{itemize}[leftmargin=8mm]
\item[$\mathrm{(a)}$]  
There is an algebra embedding
$\bfk[\widetilde{N}_{\Gamma}(T\times T/W)]_{\hat 0}\to 
K_{T\times\bbG_m}(\Gr)_{\hat\zeta}$
such that the composed map
$$\xymatrix{
\calA[T\times T/W]_{\hat 0}\ar[r]&
\bfk[\widetilde{N}_{\Gamma}(T\times T/W)]_{\hat 0}\ar[r]& K_{T\times\bbG_m}(\Gr)_{\hat\zeta}
\ar[r]^-{\res_x}&\calA[T]_{\hat 0}
}$$
takes $f\otimes g$ to $f\cdot{}^x\!g$ for each $f\in\bfk[T]$, $g\in\bfk[T/W]$  and $x\in\LLbd$.

\item[$\mathrm{(b)}$] 
There is an algebra embedding 
$\bfk[\widetilde{N}_{\Om}(T/W)]_{\hat 0}\to K_{\bbG_m}(\Gr)_{\hat\zeta}$
such that the composed map
$$\xymatrix{
\calA[T/W]_{\hat 0}\ar[r]&
\bfk[\widetilde{N}_{\Om}(T/W)]_{\hat 0}\ar[r]& K_{\bbG_m}(\Gr)_{\hat\zeta}
\ar[r]^-{\res_x}&\bfk[[\hbar]]
}$$
takes $g$ to ${}^x\!g(1)$ for each $g\in\bfk[T/W]$  and $x\in\LLbd$.
\end{itemize}
\end{lemma}

\begin{proof}
Let us concentrate on (a) because the proof of (b) is similar.
We must check that the map \eqref{gamma} extends through the canonical embedding
$\bfA[T\times T/W]\to \bfk[\widetilde{N}_{\Gamma}(T\times T/W)]_{\hat 0}$ to an algebra homomorphism
$\bfk[\widetilde{N}_{\Gamma}(T\times T/W)]_{\hat 0}\to 
K_{T\times\bbG_m}(\Gr)_{\hat\zeta}$.
To do that, we first note that the map \eqref{gamma} yields an algebra homomorphism
\begin{align}\label{pi3}\calA[T\times T/W]_{\hat 0}
\to K_{T\times\bbG_m}(\Gr)_{\hat\zeta}.\end{align}
Taking its specialization at $\hbar=0$,
we get a map
$$\bfk[T\times T/W]\to K_{T\times\bbG_m}(\Gr)_\zeta.$$
We claim that this map takes $I_{\Gamma}$ to $\{0\}$.
From \eqref{free}, we deduce that $K_{T\times\bbG_m}(\Gr)$ is a free $\bfA$-module. 
Hence, the map \eqref{pi3} takes $I_{\Gamma}$ to 
$\hbar\, K_{T\times\bbG_m}(\Gr)_{\hat\zeta}$. Thus, we obtain the desired morphism
$\bfk[\widetilde{N}_{\Gamma}(T\times T/W)]_{\hat 0}\to K_{T\times\bbG_m}(\Gr)_{\hat\zeta}.$
To prove that it is injective, it is enough to observe that its composition with the map
\eqref{pi1} yields an inclusion
\begin{align}\label{pi4}\bfk[\widetilde{N}_{\Gamma}(T\times T/W)]_{\hat 0}
\to\text{Fun}(\LLbd\,,\,\bfk[T][[\hbar]]).\end{align}
To prove the claim, note that the ideal $I_\Gamma$ of $\bfk[T\times T/W]$ is generated
by the set $\{f\otimes 1-1\otimes f\,;\,f\in[\ell]^*\bfk[T/W]\}$, that
$\ell\LLbd$ acts as the identity on $\bfk[T]$ by \eqref{WactD}, and that 
the restriction yields an embedding of $K_{T\times\bbG_m}(\Gr)_\zeta$
into $K_T(\Gr^\zeta).$
\end{proof}

\smallskip

\subsubsection{The equivariant cohomology of $\Gr$ and $\Gr^\zeta$}

Recall that
$$H^\bullet_{T\times\bbG_m}=\bfk[\frakt\times\bbG_a]=\calA[\frakt],
\qquad
H^\bullet_T=\bfk[\frakt].$$
For each subset $J\subset\Sigma_\af$
the equivariant cohomology group $H^\bullet_{T\times\bbG_m}(\Fl^J)$ is a limit in the category of graded 
$H^\bullet_{T\times\bbG_m}$-modules in a similar way as the equivariant $K$-cohomology
group considered above. 
Let $[\Fl^{J,x}]$ be the fundamental class of the finite codimensional Schubert variety labelled by the coset $x$
in $W^J$.
We have
$$H_{T\times\bbG_m}^\bullet(\Fl^J)=\bigoplus_{x\in W^J}H_{T\times\bbG_m}^\bullet\cdot[\Fl^{J,x}].$$
The $G(\calK)$-action on $\Fl^J$ by left multiplication yields a $W_\ex$-action on 
the cohomology space $H^\bullet_{T\times\bbG_m}(\Fl^J)$ which will be referred as the left action. 
The group $W_\ex$ acts also on $H^\bullet_{T\times\bbG_m}(\Fl)$ via the affine Springer action, 
see \S\ref{sec:symmetries} for details.
We'll refer to this action as the right action. 
The group $W_\ex$ acts on 
$$H^\bullet_{T\times\bbG_m}=\bfk[\frakt\times\bbG_a]=\Sym(\Lambda\oplus\bbZ\delta)\otimes\bfk$$
as in \eqref{WactD}.
The right action is $H^\bullet_{T\times\bbG_m}$-linear and the left one is $H^\bullet_{T\times\bbG_m}$-skew-linear.

The cup product by elements of $H^\bullet_{T\times\bbG_m}$
yields a graded ring homomorphism
$l:H^\bullet_{T\times\bbG_m}\to H_{T\times\bbG_m}^\bullet(\Fl).$
The first $T\times\bbG_m$-equivariant Chern classes of the line bundles
$\calL(\lambda)=G(\calK)\times_\calI\lambda$ with $\lambda\in\Lbd\oplus\bbZ\delta$ 
yield a graded ring homomorphism
$r:H^\bullet_{T\times\bbG_m}\to H_{T\times\bbG_m}^\bullet(\Fl).$
Note that $\calL(\delta)$ is a trivial line bundle on $\Fl$ with first $T\times\bbG_m$-equivariant Chern class
$\text{c}_1(\calL(\delta))=l(\hbar)$.
Thus the tensor product $l\otimes r$ is a 
graded ring homomorphism which fits into the commutative diagram
\begin{align*}
\xymatrix{
H^\bullet_{T\times\bbG_m}\otimes_{H^\bullet_{\bbG_m}} H^\bullet_{T\times\bbG_m}\ar@{=}[d]
\ar[r]^-{l\otimes r}& H_{T\times\bbG_m}^\bullet(\Fl)\ar@{=}[d]\\
\bfk[\frakt\times\frakt\times\bbG_a]\ar[r]&\bfk[\frakt\times\frakt\times\bbG_a][g_i]/(l(f_i)-r(f_i)-\hbar g_i).}
\end{align*}
The lower map is the obvious one, the elements $f_1,\dots,f_r$ 
are the homogeneous generators of the ring $\bfk[\frakt]^W$, and $g_i$
has the degree $\deg(g_i)=\deg(f_i)-1$.
Let $\Delta\subset\frakt\times\frakt$ be the pull-back of the diagonal of $\frakt/W$ by the obvious projection
$\frakt\times\frakt\to \frakt/W\times \frakt/W$.
We deduce that the map $l\otimes r$ extends to a graded ring isomorphism
$$\bfk[\widetilde N_\Delta(\frakt\times\frakt)]=\bfk[\frakt\times\frakt\times\bbG_a][g_i]/(l(f_i)-r(f_i)-\hbar g_i)\to H_{T\times\bbG_m}^\bullet(\Fl).$$
The pull-back by the projection $\Fl\to\Fl^J$ yields an isomorphism
$$H_{T\times\bbG_m}^\bullet(\Fl^J)= H_{T\times\bbG_m}^\bullet(\Fl)^{W_J},$$
where the invariants on the right hand side are relative to the right action (=the Springer action) of 
$W_J$ on $H_{T\times\bbG_m}^\bullet(\Fl)$. 
The right $W_\ex$-action on $\frakt\times\frakt\times\bbG_a$ preserves the closed subset
$\Delta\times\{0\}$. Hence $W_\ex$ acts on the scheme 
$\widetilde N_\Delta(\frakt\times\frakt)$.
We deduce that there is a 
graded ring isomorphism
\begin{align}\label{BF}
\bfk[\widetilde N_{\Delta}(\frakt\times\frakt)/W_J]\to H_{T\times\bbG_m}^\bullet(\Fl^J).
\end{align}
Note that if $W_J\subset W$ then 
$\widetilde N_{\Delta}(\frakt\times\frakt)/W_J=\widetilde N_{\Delta_J}(\frakt\times\frakt/W_J)$
where $\Delta_J\subset\frakt\times(\frakt/W_J)$ is the image of $\Delta$.
See, e.g.,  \cite[thm.~1]{BF08} for more details in the special case $W_J=W$.

Now, we consider the equivariant cohomology of $\Gr^\zeta$.
For each coset $x\in W^J$, the restriction to the fixed point $\delta_x$ in $(\Fl^J)^{T\times\bbG_m}$
yields a map
$$\res_x:H^\bullet_{T\times\bbG_m}(\Fl^J)\to \calA[\frakt].$$
Thus, for $x\in\LLbd$, the restriction to $\delta_x$ yields maps
$$\res_x:H^\bullet_{T\times\bbG_m}(\Gr^\zeta)\to\calA[\frakt],\qquad
\res_x:H^\bullet_{\bbG_m}(\Gr^\zeta)\to\calA.$$
Let $H_T^\bullet(\Gr^\zeta)_{\hat 0}$ be the completion of the
$H_T^\bullet$-algebra $H^\bullet_T(\Gr^\zeta)$ at the point $0\in\frakt_\bfk$,
and $H^\bullet_{T\times\bbG_m}(\Gr^\zeta)_{\widehat{0,0}}$ be the completion of the
$H^\bullet_{T\times\bbG_m}$-algebra $H^\bullet_{T\times\bbG_m}(\Gr^\zeta)$ at the point $(0,0)$ in 
$\frakt_\bfk\times\bfk$.

\begin{lemma}\label{lem:BUa}Let $G$ be simply connected.
There are algebra isomorphisms
\begin{align*}
\bfk[\widetilde{N}_{\Gamma}(T\times T/W)]_{\widehat{1,0}}\to
H^\bullet_{T\times\bbG_m}(\Gr^\zeta)_{\widehat{0,0}},\qquad
\bfk[N_{\Gamma}(T\times T/W)]_{\hat 1}\to H^\bullet_T(\Gr^\zeta)_{\hat 0}.
\end{align*}

\end{lemma}

\begin{proof} 
The partition \eqref{Gamma} and the isomorphism of formal schemes
$$
\widetilde{N}_{\Gamma_\omega}(T\times T/W)_{\widehat{1,0}}=
\widetilde{N}_{\Delta}(\frakt\times \frakt)_{\widehat{0,0}}/W_{\ell,\omega}
$$
imply that
\begin{equation}\label{decN}
\bfk[\widetilde{N}_{\Gamma}(T\times T/W)]_{\widehat{1,0}}=
\bigoplus_{\omega\in \Xi}\bfk[\widetilde{N}_{\Delta}(\frakt\times \frakt)/W_{\ell,\omega}]_{\widehat{0,0}}.
%\bigoplus_{\omega\in \Om(\bbC)ega}H^\bullet_{T\times\bbG_m}(\Fl_{J_\omega}^{(\ell)})_{\hat 0}=
%H^\bullet_{T\times\bbG_m}(\Gr^\zeta)_{\hat 0}
\end{equation}
For each subset $J\subset\Sigma_\af$ the composed map
$$\xymatrix{
\bfk[\frakt\times \frakt\times\bbG_a/W_{\ell,J}]\ar[r]&
\bfk[\widetilde{N}_{\Delta}(\frakt\times \frakt)/W_{\ell,J}]\ar[r]^-{\eqref{BF}}& H_{T\times\bbG_m}^\bullet(\Fl^J_\ell)
\ar[r]^-{\res_x}&\calA[\frakt]
}$$
maps $f\otimes g$ to $f\cdot{}^x\!g$ for $f\in\bfk[\frakt]$ and $g\in\bfk[\frakt\times\bbG_a/W_{\ell,J}]$.
Taking the sum over all $\omega$'s, we get a canonical ring isomorphism
$$\bfk[\widetilde{N}_{\Gamma}(T\times T/W)]_{\widehat{1,0}}=
H^\bullet_{T\times\bbG_m}(\Gr^\zeta)_{\widehat{0,0}}.$$
Since $\Gr^\zeta$ is equivariantly formal, specializing this isomorphism to the point $\hbar=0$ in $\bbG_a$, 
we get
$$\bfk[N_{\Gamma}(T\times T/W)]_{\hat 1}
=H^\bullet_{T}(\Gr^\zeta)_{\hat 0}.$$
\end{proof}

\smallskip

In the next proposition, we remove the assumption that $G$ is simply-connected.
Recall that
\begin{align}\label{decomposition}
\Gr^\zeta =\bigsqcup_{\omega\in \Xi} \Fl^\omega_\ell,\qquad
\Gamma_{\hat 1}=\bigsqcup_{\omega\in \Xi}\Gamma_\omega,\qquad
\Omega=\bigsqcup_{\omega\in \Xi}\Omega_\omega.
\end{align}
The following proposition is the main result of this section. It summarizes and completes the results above.

\begin{proposition}\label{prop:N}\hfill
\begin{itemize}[leftmargin=8mm]
\item[$\mathrm{(a)}$]  There are  unique algebra isomorphisms
\begin{align}\label{BUa}
\begin{split}
&\bfk[\widetilde{N}_{\Gamma}(T\times T/W)]^{\oplus\pi_1}_{\widehat{1,0}}=
H^\bullet_{T\times\bbG_m}(\Gr^\zeta)_{\widehat{0,0}},\\
&\bfk[N_{\Gamma}(T\times T/W)]^{\oplus\pi_1}_{\hat 1}= H^\bullet_T(\Gr^\zeta)_{\hat 0}
\end{split}
\end{align}
such that,
under the decompositions $\eqref{decomposition}$, we have the following
algebra isomorphisms
\begin{align*}
&\bfk[\widetilde N_{\Gamma_\omega}(T\times T/W)]^{\oplus\pi_1}_{\widehat {1,0}}= 
H^\bullet_{T\times\bbG_m}(\Fl^\omega_\ell)_{\widehat {0,0}},\\
&\bfk[N_{\Gamma_\omega}(T\times T/W)]^{\oplus\pi_1}_{\hat 1}= H^\bullet_T(\Fl^\omega_\ell)_{\hat 0},
\end{align*}
and the composed map
$$\xymatrix{
\calA[T\times T/W]_{\widehat{1,0}}\ar[r]&
\bfk[\widetilde{N}_{\Gamma}(T\times T/W)]^{\oplus\pi_1}_{\widehat{1,0}}\ar[r]& 
H^\bullet_{T\times\bbG_m}(\Gr^\zeta)_{\widehat{0,0}}
\ar[r]^-{\res_x}&\calA[T]_{\widehat{1,0}}
}$$
takes $f\otimes g$ to $f\cdot{}^x\!g$ for each $f\in\bfk[T]_{\hat 1}$, $g\in\bfk[T/W]$ and
$x\in\LLbd$.

\item[$\mathrm{(b)}$] 
There are unique algebra isomorphisms
\begin{align}\label{BUb}
\bfk[\widetilde N_{\Om}(T/W)]^{\oplus\pi_1}_{\hat 0}= H^\bullet_{\bbG_m}(\Gr^\zeta)_{\hat 0},\qquad
\bfk[N_{\Om}(T/W)]^{\oplus\pi_1}= H^\bullet(\Gr^\zeta)
\end{align}
such that, under the decompositions $\eqref{decomposition}$, we have the following
algebra isomorphisms
\begin{align*}
\bfk[\widetilde N_{\Om_\omega}(T/W)]^{\oplus\pi_1}_{\hat 0}= H^\bullet_{\bbG_m}(\Fl^\omega_\ell)_{\hat 0},\quad
\bfk[N_{\Om_\omega}(T/W)]^{\oplus\pi_1}= H^\bullet(\Fl^\omega_\ell),
\end{align*}
and the composed map
$$\xymatrix{
\calA[T/W]_{\hat 0}\ar[r]&
\bfk[\widetilde{N}_{\Omega}\big(T/W\big)]^{\oplus\pi_1}_{\hat 0}\ar[r]& 
H^\bullet_{\bbG_m}(\Gr^\zeta)_{\hat 0}
\ar[r]^-{\res_x}&\bfk[[\hbar]]
}$$
takes $g$ to ${}^x\!g(1)$ for each weight $x\in\LLbd$.
\end{itemize}
\qed
\end{proposition}

\medskip

\subsection{The equivariant cohomology of affine Springer fibers}

\subsubsection{ Affine Springer fibers}
The set of fixed points is  $(\Fl^J)^T$.
The stabilizer in $T$ of any point of an arbitrary 1-dimensional $T$-orbit $E$ in any $T$-variety
is the kernel of a character  of $T$ which is uniquely determined up to integer multiples. 
We pick such a character $\alpha_E$ with minimal norm. It is unique up to a sign.

Fix an element $\gamma$ in $\frakg_\calK$. Let  ${}^0\!\Fl^J_\gamma$ 
be the closed set of all $x\in\Fl^J$ such that
$\gamma$ belongs to the Lie algebra of the prounipotent radical of the stabilizer of $x$ in $G(\calK)$.
Let $\Fl^J_\gamma$ be the closed set of all $x\in\Fl^J$ such that
$\gamma$ belongs to the Lie algebra of the stabilizer of $x$.
Following the terminology of Borho-MacPherson, we may call  
${}^0\!\Fl^J_\gamma$ an affine Spaltenstein fiber, $\Fl^J_\gamma$ an affine Steinberg fiber and
$\Fl_\gamma={}^0\!\Fl^\emptyset_\gamma=\Fl^\emptyset_\gamma$ an affine Springer fiber.
%We'll abbreviate ${}^0\!\Fl^{J,\circ}_\gamma={}^0\!\Fl^J_\gamma\cap \Fl^{J,\circ}$
%and $\Fl^{J,\circ}_\gamma=\Fl^J_\gamma\cap \Fl^{J,\circ}$.
The torus $T$ centralize $\gamma$, thus it acts on $ {}^0\!\Fl^J_\gamma$. 
From now on let $\gamma_\ell=s\otimes t^\ell$, where $s\in \frakt$ is regular semi-simple,
and set $\gamma=\gamma_1$.
Recall that $(\Fl^J)^T=\{\delta_x\,;\,x\in W^J_\ex\}$.

\begin{lemma}\label{lem:isom2}
\hfill
\begin{itemize}[leftmargin=8mm]
\item[$\mathrm{(a)}$] 
There is an isomorphism of ind-varieties
${}^0\Gr^\zeta_{\gamma_\ell} =\bigsqcup_{\omega\in \Xi} {}^0\!\Fl^\omega_{\ell,\gamma_\ell}.$

\item[$\mathrm{(b)}$] 
${}^0\!\Fl^J_\gamma\cap N(\calK)\delta_x$ is an affine space for each $x\in W^J_\ex$. 
% of dimension $(\sharp\Phi-\sharp\widehat\Phi_{\re,J})/2$.

\item[$\mathrm{(c)}$] 
${}^0\!\Fl^J_\gamma\cap\calI\delta_x$ is an affine space for each $x\in W^J_\ex$.

\item[$\mathrm{(d)}$] $({}^0\!\Fl^J_\gamma)^{T}=(\Fl^J)^T=\{\delta_x\,;\,x\in W^J_\ex\}$.

\item[$\mathrm{(e)}$] 
Two points $\delta_x$ and $\delta_y$ in ${}^0\!\Fl^J_\gamma$ are joined
by a 1-dimensional $T$-orbit in ${}^0\!\Fl^J_\gamma$
if and only if we have $y=s_{\alpha+m\delta}x$ for some real affine root such that
$s_{\alpha+n\delta}\notin W_J$ for all $n\in\bbZ$.
There is a unique such orbit $E$. We have $\alpha_E=\pm\alpha$. 
If $x=\tau_\mu w$ with $\mu\in\LLbd$ and $w\in W$, then
$m=\max\{k\in\bbZ\,;\,
w^{-1}(\alpha)+(k-\langle\alpha,\mu\rangle)\delta\notin\Phi^J_\re\}.$

\end{itemize}
\end{lemma}

\begin{proof} 
Any element $\omega\in\Xi_\sc$ is identified with an element in $\LLbd$. 
So the adjoint action by $t^{\omega}$ on $\frakg_\calO$ 
yields a well-defined 
Lie subalgebra in $\frakg_\calK$.
To prove (a), observe first that \eqref{parahoric} implies that
\begin{align*}
\frakg_{\calK_\ell}\cap\Ad_{t^{-\omega}}(\frakg_\calO)=\frakp^\omega_\ell.
\end{align*}
Since $\frakg_{\calK_\ell}\cap\Ad_{t^{-\omega}}(\frakg)$ is reductive, we deduce that
$\frakg_{\calK_\ell}\cap\Ad_{t^{-\omega}}(t\frakg_\calO)$ is the pro-unipotent radical 
${}^0\frakp^\omega_\ell$ of $\frakp^\omega_\ell.$
Thus, from \eqref{lem:isom1}, we get
\begin{align*}
{}^0\Gr^\zeta_{\gamma_\ell} 
&=\bigsqcup_{\omega\in \Xi}\{g\delta_\omega\,;\, g\in G(\calK_\ell)\,,\,
\Ad_g^{-1}(\gamma_\ell)\in\frakg_{\calK_\ell}\cap\Ad_{t^{-\omega}}(t\frakg_\calO)\}\\
&=\bigsqcup_{\omega\in \Xi}\{g\delta_\omega\,;\, g\in G(\calK_\ell)\,,\,
\Ad_g^{-1}(\gamma_\ell)\in {}^0\frakp^\omega_\ell\}\\
&=\bigsqcup_{\omega\in \Xi} {}^0\!\Fl^\omega_{\ell,\gamma_\ell}.
\end{align*}
Part (d) is obvious.
Now, we concentrate on parts (b) and (e).
The Iwasawa decomposition implies that 
$\Fl^J=\bigsqcup_{x\in W^J_\ex}N(\calK)\delta_x$.
We choose an element $n\in N(\calK)$ of the following form
$$n=\prod_{\alpha\in\Phi^+}n_\alpha(t^{k_\alpha}),\qquad k_\alpha\in\bbZ.$$
Let $x\in W_\ex$ with $x=\tau_\mu w$
and $w\in W$, $\mu\in\LLbd$.
We have
\begin{align}\label{ad}
\Ad_{nx}^{-1}(\gamma)=\Ad_x^{-1}(\gamma)+
\sum_{\alpha\in\Phi^+}(\langle\alpha,s\rangle t^{1+k_\alpha-\langle\alpha,\mu\rangle}+\text{higher\ terms})e_{w^{-1}(\alpha)}.
\end{align}
Since $\Ad_x^{-1}(\gamma)\in {}^0\frakp^J$, we deduce that
\begin{align*}
n\delta_x\in {}^0\!\Fl^J_\gamma&\iff
w^{-1}(\alpha)+(1+k_\alpha-\langle\alpha,\mu\rangle)\delta\in\Phi_\re^J,\ \forall \alpha\in\Phi^+,\\
n\delta_x\neq\delta_x&\iff
w^{-1}(\alpha)+(k_\alpha-\langle\alpha,\mu\rangle)\delta\notin\Phi_{\re,J}\cup\Phi_\re^J,\ 
\forall \alpha\in\Phi^+.
\end{align*}
Hence, there is an isomorphism
$$ {}^0\!\Fl^J_\gamma\cap N(\calK)\delta_x\simeq \prod_{\alpha}U_{\alpha+m_\alpha\delta},$$
where the product is over all roots $\alpha\in\Phi^+$ such that
\begin{align}\label{cond}
(w^{-1}(\alpha)+\bbZ\delta)\cap\Phi_{\re,J}=\emptyset.
\end{align}
The integer $m_\alpha$ is given by
$$m_\alpha=\max\{m\in\bbZ\,;\,
w^{-1}(\alpha)+(m-\langle\alpha,\mu\rangle)\delta\notin\Phi^J_\re\}.$$
By \eqref{WactD}, we have 
$x^{-1}(\alpha+m\delta)=w^{-1}(\alpha)+(m+\langle \alpha,\mu\rangle)\delta$.
Hence the condition \eqref{cond} is equivalent to 
$x^{-1}s_{\alpha+m\delta}x\notin W_J$ for all $m\in\bbZ.$
Part (e) follows (b), because one dimensional $T$-orbits in ${}^0\!\Fl^J_\gamma$ are closure of one dimensional $T$-orbits in the affine spaces ${}^0\!\Fl^J_\gamma\cap N(\calK)\delta_x$, and these are given by the root vectors.

Finally, let us consider the part (c). Fix an order
$$\{\alpha_1+m_1\delta,\dots,\alpha_r+m_r\delta\}=\Phi_\re^+\setminus x(\Phi^J_\re\sqcup\Phi_{\re,J}),$$
such that $\alpha_i-\alpha_j+(m_i-m_j)\delta\notin\Phi_\re^+\ \text{if}\ i<j.$
We have an isomorphism of affine spaces
$$\bbA^r\to  \calI \delta_x,\quad
(z_1,z_2,\dots,z_r)\mapsto n_{\alpha_1+m_1\delta}(z_1)n_{\alpha_2+m_2\delta}(z_2)\cdots n_{\alpha_r+m_r\delta_r}(z_r)\delta_x.$$
For $nx\in \calI \delta_x$ given by the image of $(z_1,...,z_r)$ under the above map, we have
$$\Ad_{nx}^{-1}(\gamma)=\Ad_x^{-1}(\gamma)+
\sum_{k=1}^r(\langle\alpha_k,s\rangle z_k+f_k(z_1,\cdots,z_{k-1}))e_{x^{-1}(\alpha_k)+(m_k+1)\delta},
$$
where $f_k(z_1,\cdots,z_{k-1})$ is a polynomial in $z_1,..,z_{k-1}$.
Thus $\Ad_{nx}^{-1}(\gamma)$ belongs to ${}^0\frakp^J$ if and only if
$\langle\alpha_k,s\rangle z_k+f_k(z_1,\cdots,z_{k-1})=0$ whenever
$x^{-1}(\alpha_k)+(m_k+1)\delta\notin \Phi^J_\re$.
Since $s$ is regular, we have $\langle\alpha_k,s\rangle\neq 0$. Hence this equality determines $z_k$ uniquely.
If $x^{-1}(\alpha_k)+(m_k+1)\delta\in \Phi^J_\re$, then $z_k$ can take any value.
Thus ${}^0\!\Fl^J_\gamma\cap\calI\delta_x$ is an affine space of dimension 
$$\sharp\{\gamma\in\Phi^+_\re\,;\,x^{-1}(\gamma)\notin\Phi^J_\re\sqcup\Phi_{\re,J}\,,\,
x^{-1}(\gamma)\in(\Phi^J_\re-\delta)\}.$$

\end{proof}

\smallskip

\subsubsection{The GKM presentation of $H^\bullet_T({}^0\Gr^\zeta_{\gamma_\ell})$}\label{sec:GKM}
Let $i$ be the obvious inclusion of ${}^0\Gr^\zeta_{\gamma_\ell}$ in $\Gr^\zeta$.
The ind-varieties ${}^0\Gr^\zeta_{\gamma_\ell}$ and $\Gr^\zeta$ are both equivariantly formal.
Hence, by the localization theorem, the restrictions to the fixed points 
subsets yield the following commutative diagram of algebra embeddings
\begin{align}\label{incl1}
\begin{split}
\xymatrix{
H^\bullet_T({}^0\Gr^\zeta)\ar@{_{(}->}[d]_{i^*}\ar@{^{(}->}[r]^-\res&\Fun(\LLbd,H^\bullet_T).\\
H^\bullet_T({}^0\Gr^\zeta_{\gamma_\ell})\ar@{^{(}->}[ru]_-\res&}
\end{split}
\end{align}
We can not apply the GKM method to get a presentation of the cohomology group $H^\bullet_T(\Gr^\zeta)$.
To do so we must consider the $T\times\bbG_m$-action.
However we can apply it to the cohomology group
$H^\bullet_T({}^0\Gr^\zeta_{\gamma_\ell})$.
To do that, we first consider the connected components of 
${}^0\Gr^\zeta_{\gamma_\ell}$ separately.
Using Lemma \ref{lem:isom2} as in \cite[\S 9.4]{GKM2} we get the following

\smallskip

\begin{lemma}\label{lem:GKM0}
For each element $\omega\in\Xi$ there is an $H^\bullet_T$-algebra isomorphism
$$H^\bullet_T( {}^0\!\Fl^\omega_{\ell,\gamma_\ell})=
 \left\{ (a_x)\in \Fun(W_{\ell,\ex}^\omega\,,H^\bullet_T)\,\middle |\,
\begin{array}{c}a_x\equiv a_{s_{\alpha+\ell m\delta}x}\modulo \alpha,\ \forall m\in\bbZ,\\
\ \text{if}\ x^{-1}s_{\alpha+\ell n\delta}x\notin W_{\ell,\omega} 
\ \forall n\in\bbZ\end{array}.\right\}$$
\qed
\end{lemma}

%%%%%%%%%%%%%%%%%%%%%%%%%%%%%%%%%

\smallskip

From Lemmas \ref{lem:isom2}, \ref{lem:GKM0} we get the following.

\smallskip

\begin{proposition}\label{prop:GKM1} 
There is an $H^\bullet_T$-algebra isomorphism
$$H^\bullet_T({}^0\Gr^\zeta_{\gamma_\ell})=
 \left\{ (a_\lambda)\in \Fun(\LLbd\,,H^\bullet_T)\,\middle |\,\begin{array}{c}
a_\lambda\equiv a_{s_{\alpha+\ell m\delta}\bullet\lambda}\modulo\alpha\ \forall m\in\bbZ,\\
\ \text{if}\ \langle\alpha\,,\,\lambda+\Lrho\rangle\notin\ell\bbZ\end{array}\right\}.$$
\end{proposition}

\begin{proof}
Use the following identifications of $W_{\ell,\ex}$-sets
$$%\LLbd=W^\Sigma_\ex=W_\ex/W,\qquad
\LLbd=\bigsqcup_{\omega\in \Xi} W_{\ell,\ex}\bullet\omega=
\bigsqcup_{\omega\in \Xi} W_{\ell,\ex}/W_{\ell,\omega}
=\bigsqcup_{\omega\in \Xi} W_{\ell,\ex}^\omega,$$ 
the decomposition in Lemma \ref{lem:isom2}, and the following equivalences
$$s_{\alpha+\ell n\delta}\bullet\lambda\neq\lambda,\ \forall n\in\bbZ
\iff s_\alpha \bullet\zeta^\lambda\neq\zeta^\lambda\iff \langle\alpha\,,\,\lambda+\Lrho\rangle\notin\ell\bbZ.$$
\end{proof}

\smallskip

\subsubsection{Affine Weyl group symmetries of of $H^\bullet_T({}^0\Gr^\zeta_{\gamma_\ell})$}
\label{sec:symmetries}
The $G(\calK_\ell)$-action on $\Gr^\zeta$ gives rise to the left $W_{\ell,\ex}$-action on $H^\bullet_T(\Gr^\zeta)$.
The $T(\calK_\ell)$-action on the subset ${}^0\Gr^\zeta_{\gamma_\ell}$ yields
a $\ell\LLbd$-action on $H^\bullet_T({}^0\Gr^\zeta_{\gamma_\ell})$.
This action extends to a $W_{\ell,\ex}$-action
on $H^\bullet_T({}^0\Gr^\zeta_{\gamma_\ell})$ called again the left action.
Indeed, recall that under the restriction 
$$\res:H^\bullet_T(\Gr^\zeta)\to\Fun(\LLbd,H^\bullet_T)$$
the left $W_{\ell,\ex}$-action on 
$H^\bullet_T(\Gr^\zeta)$ is given by the following rule
\begin{align}\label{leftaction}
x(a_\lambda)={}^xa_{x\bullet\lambda},\qquad \lambda\in \LLbd,\quad 
x\in W_{\ell,\ex},\quad (a_{\lambda})\in \Fun(\LLbd,H^\bullet_T).
\end{align}
Here the map $f\mapsto{}^xf$ is the $W_{\ell,\ex}$-action on $H^\bullet_T=\bfk[\frakt]$
dual to \eqref{Wact1}.
Under the restriction 
$$\res:H^\bullet_T({}^0\Gr^\zeta_{\gamma_\ell})\to\Fun(\LLbd,H^\bullet_T)$$
we define the left $\ell\LLbd$-action on 
$H^\bullet_T({}^0\Gr^\zeta_{\gamma_\ell})$ by
$$\mu(a_\lambda)=a_{\tau_\mu\bullet\lambda},\qquad \lambda,\mu\in \LLbd,\qquad 
(a_\lambda)\in \Fun(\LLbd,H^\bullet_T).$$
By Proposition \ref{prop:GKM1}, 
this action extends uniquely to a $W_{\ell,\ex}$-action on $H^\bullet_T({}^0\Gr^\zeta_{\gamma_\ell})$
such that \eqref{leftaction} holds. 
The element $\gamma_\ell$ is homogeneous.
Hence, both ind-schemes $\Gr^\zeta$ and 
${}^0\Gr^\zeta_{\gamma_\ell}$ are 
$T$-equivariantly formal by \cite{GKM1}.
Thus, we have
$$H^\bullet_T(\Gr^\zeta)\otimes_{H^\bullet_T}\bfk=H^\bullet(\Gr^\zeta),\qquad
H^\bullet_T({}^0\Gr^\zeta_{\gamma_\ell})\otimes_{H^\bullet_T}\bfk=H^\bullet({}^0\Gr^\zeta_{\gamma_\ell}).$$
The left $W_{\ell,\ex}$-actions on $H^\bullet_T(\Gr^\zeta)$ and
$H^\bullet_T({}^0\Gr^\zeta_{\gamma_\ell})$ are $H^\bullet_T$-skew-linear.
Thus, they specialize to
$W_{\ell,\ex}$-actions on the cohomology spaces $H^\bullet(\Gr^\zeta)$ and
$H^\bullet({}^0\Gr^\zeta_{\gamma_\ell})$.
The restriction
$$i^*:H^\bullet_T(\Gr^\zeta)\to H^\bullet_T({}^0\Gr^\zeta_{\gamma_\ell})$$ commutes with the left
$W_{\ell,\ex}$-action. By formality, it specializes to the restriction
$$i^*:H^\bullet(\Gr^\zeta)\to H^\bullet({}^0\Gr^\zeta_{\gamma_\ell})$$
which is also $W_{\ell,\ex}$-equivariant.
The left $W_{\ell,\af}$-action on $H^\bullet(\Gr^\zeta)$ is trivial, 
because it comes from the left action of the neutral connected component of $G(\calK_\ell)$.
Hence $i^*$ maps into $H^\bullet({}^0\Gr^\zeta_{\gamma_\ell})^{W_{\ell,\af}}$.
The connected components of $\Gr$ are parametrized by the group $\pi_1$, hence
the algebras $H^\bullet(\Gr^\zeta)$ and $H^\bullet({}^0\Gr^\zeta_{\gamma_\ell})$
are both $\pi_1$-graded. 
The canonical inclusion $\pi_1\subset W_{\ell,\ex}$ recalled in \S\ref{sec:Xi}
yields an action of the group $\pi_1$ by algebra automorphisms
on $H^\bullet(\Gr^\zeta)$ and $H^\bullet({}^0\Gr^\zeta_{\gamma_\ell})$.
Since $W_{\ell,\ex}=W_{\ell,\af}\rtimes\pi_1$, 
the subalgebra $H^\bullet({}^0\Gr^\zeta_{\gamma_\ell})^{W_{\ell,\af}}$ of $W_{\ell,\af}$-invariants elements 
is preserved by the $\pi_1$-action, and the map $i^*$ is a $\pi_1$-equivariant algebra homomorphism
\begin{align}\label{invariant-res}
i^*:H^\bullet(\Gr^\zeta)\to H^\bullet({}^0\Gr^\zeta_{\gamma_\ell})^{W_{\ell,\af}}.
\end{align}
In particular, it takes $H^\bullet(\Gr^\zeta)^{\pi_1}$ into $H^\bullet({}^0\Gr^\zeta_{\gamma_\ell})^{W_{\ell,\ex}}$.
Here, for each subset $S\subset W_{\ell,\ex}$, we write
$H^\bullet({}^0\Gr^\zeta_{\gamma_\ell})^S$
for the space of $S$-invariants with respect to the left action.

The group $W_\ex$ acts also on $H^\bullet(\Fl_\gamma)$ via the affine Springer action defined in \cite{Lu}.
According to \cite{GKM2}, under the embedding $i^*$ this action is given by
$$y(a_x)=(a_{xy^{-1}}),\qquad x,y\in W_\ex.$$
We'll refer to the Springer action as the right action. To avoid confusions we'll write
$(a_x)y^{-1}$ instead of $y(a_x)$.

The left and right actions commute with each other.
They both preserve the cohomological grading.
The right action is $H^\bullet_T$-linear and the left one is $H^\bullet_T$-skew-linear.

From now on and until the end of Proposition \ref{prop:PI}, we focus on the component $J=\emptyset$, and safely remove $\ell$ in the discussion. 

\begin{proposition}\label{prop:dimformula}
\hfill
\begin{itemize}[leftmargin=8mm]
\item[$\mathrm{(a)}$]
$\dim H^\bullet(\Fl_\gamma)^{W_\ex}\leqslant (h+1)^r$.
\item[$\mathrm{(b)}$]
$\dim H^\bullet(\Fl_\gamma)^{W_\ex}= (h+1)^r$ in type $A$.
\end{itemize}
\end{proposition}

\begin{proof}
We first concentrate on part (a).
We'll prove the equivalent statement
$$\dim H_\bullet(\Fl_\gamma)_{W_\ex}\leqslant (h+1)^r$$
where the left hand side is the dimension of the space of coinvariant in the Borel-Moore homology 
group for the left action.
We may assume that the group $G$ is simply connected, because the terms of the inequality do not depend on the 
fundamental group. Then $\LLbd$ is the coroot lattice $\LQ$.
For each affine element $x\in W_\ex$ we consider the affine space 
$C_x=\Fl_\gamma\cap\calI\delta_x.$
From the proof of Lemma \ref{lem:isom2} we deduce that
\begin{align*}
&C_x=x(\prod_{\alpha\in E_x}U_\alpha)\delta_1,\\
&E_x=x^{-1}(\Phi_\re^+)\cap\Phi^-_\re\cap(\Phi^+_\re-\delta)=x^{-1}(\Phi_\re^+)\cap((\Phi^+-\delta)\sqcup\Phi^-).
\end{align*}
Hence, for each affine reflexion $s_i$ we have
\begin{align*}
s_iC_x= C_{s_ix}
&\iff E_x=E_{s_ix}\\
&\iff \Phi_\re^+\cap(x(\Phi^+-\delta)\sqcup x(\Phi^-))=s_i(\Phi_\re^+)\cap(x(\Phi^+-\delta)\sqcup x(\Phi^-))\\
&\iff\pm x^{-1}(\alpha_i)\notin (\Phi^+-\delta)\sqcup\Phi^-\\
&\iff x^{-1}(\alpha_i)\notin E_x\cup(-E_{s_ix}).
\end{align*}
We consider the equivalence relation on $W_\ex$ such that 
$$y\sim x\iff y^{-1}C_y=x^{-1}C_x.$$
We have 
\begin{align}\label{sim1}x\sim s_ix\iff  x^{-1}(\alpha_i)\notin E_x\cup(-E_{s_ix}).
\end{align}
Let $\overline{C_x}$ be the closure of $C_x$ in $\Fl_\gamma$, and 
$b_x=[\overline{C_x}]$ be its fundamental class in $H_\bullet(\Fl_\gamma)$.
Then $\{b_x\,;\, x\in W_\ex\}$ is a basis of $H_\bullet(\Fl_\gamma)$
such that $b_y=b_x$ in $H_\bullet(\Fl_\gamma)_{W_\ex}$ whenever $y\sim x$.
Let $M\subset W_\ex$ be the set of all elements which are minimal in their $\sim$-class.
We deduce that
$$\dim H_\bullet(\Fl_\gamma)_{W_\ex}\leqslant \sharp(W_\ex/\sim)\leqslant
\sharp(M).$$
The alcove of $x\in W_\ex$ is
$A_x=x^{-1}(A_1)$ where
$$A_1=\{\mu\in\frakt_\bbR\,;\, 0< \langle\alpha,\mu\rangle< 1\,,\,\forall\alpha\in\Phi^+\}.$$
We write $\langle\alpha,A_{x^{-1}}\rangle=\langle x^{-1}(\alpha),\Lrho\rangle/h$ for each $\alpha\in\Lambda$.
Note that
$$|\langle\Phi,A_1\rangle|\subset(0,1),\qquad \langle\alpha_i,A_{x^{-1}}\rangle<0\iff s_ix<x.$$

Pick an element $x\in M$.
From \eqref{sim1} we deduce that 
$$\langle\alpha_i,A_{x^{-1}}\rangle<0\Rightarrow x^{-1}(\alpha_i)\in E_x\cup(-E_{s_ix}).$$
%$\pm x^{-1}(\alpha_i)\in (\Phi^+-\delta)\sqcup\Phi^-.$
Write $x^{-1}=w^{-1}\tau_\mu$ with $w\in W$ and $\mu\in\LLbd$. 
By \eqref{WactD} we have
\begin{align}\label{sim2}
x^{-1}(\alpha_i)=w^{-1}(\alpha_i)-\langle\alpha_i,\mu\rangle\delta,\qquad
\langle\alpha_i,A_{x^{-1}}\rangle=\langle \alpha_i,A_{w^{-1}}\rangle-\langle\alpha_i,\mu\rangle.
\end{align}

First, assume that $i\neq 0$. 
If $x^{-1}(\alpha_i)\in E_x$ then $x^{-1}(\alpha_i)\in (\Phi^+-\delta)\sqcup\Phi^-$.
Hence,
either $w^{-1}(\alpha_i)\in\Phi^+$ and $\langle\alpha_i,\mu\rangle=1$, 
or $w^{-1}(\alpha_i)\in\Phi^-$ and $\langle\alpha_i,\mu\rangle=0$. 
Hence $\langle \alpha_i,A_{x^{-1}}\rangle<0$ in both cases.
Similarly, if $x^{-1}(\alpha_i)\in -E_{s_ix}$, then $\langle \alpha_i,A_{x^{-1}}\rangle>0$.
Therefore, the computation above and \eqref{sim2} yield
$$\langle\alpha_i,A_{x^{-1}}\rangle<0\Rightarrow
x^{-1}(\alpha_i)\in E_x
\Rightarrow \langle\alpha_i,A_{x^{-1}}\rangle>-1.$$
On the other hand, if $\langle \alpha_i,A_{x^{-1}}\rangle\geqslant 0$ 
then \eqref{sim2} yields either $w^{-1}(\alpha_i)\in\Phi^+$ and $\langle\alpha_i,\mu\rangle\leqslant 0$, 
or $w^{-1}(\alpha_i)\in\Phi^-$ and $\langle\alpha_i,\mu\rangle\leqslant -1$. 
Hence, in both cases we have $\langle\alpha_i,A_{x^{-1}}\rangle>-1.$

Now, let $i=0$. 
If $x^{-1}(\alpha_0)\in E_x$ then $x^{-1}(\alpha_0)\in (\Phi^+-\delta)\sqcup\Phi^-$.
Hence,
either $w^{-1}(\alpha_0)\in\Phi^++\delta$ and $\langle\alpha_i,\mu\rangle=2$, 
or $w^{-1}(\alpha_0)\in\Phi^-+\delta$ and $\langle\alpha_i,\mu\rangle=1$. 
Hence $\langle \alpha_0,A_{x^{-1}}\rangle\in(-1,0)$ in both cases.
Similarly, if $x^{-1}(\alpha_0)\in -E_{s_0x}$, then $\langle \alpha_0,A_{x^{-1}}\rangle\in(0,1)$.
Therefore, the computation above and \eqref{sim2} yield
$$\langle\alpha_0,A_{x^{-1}}\rangle<0\Rightarrow
x^{-1}(\alpha_0)\in E_x
\Rightarrow \langle\alpha_0,A_{x^{-1}}\rangle>-1.$$
On the other hand, if $\langle \alpha_0,A_{x^{-1}}\rangle\geqslant 0$ 
then \eqref{sim2} yields either $w^{-1}(\alpha_0)\in\Phi^++\delta$ and $\langle\alpha_0,\mu\rangle\leqslant 2$, 
or $w^{-1}(\alpha_i)\in\Phi^-+\delta$ and $\langle\alpha_0,\mu\rangle\leqslant 1$. 
Hence, in both cases we have $\langle\alpha_0,A_{x^{-1}}\rangle>-1.$

Therefore, the case-by-case computations above imply that
$$x\in M\Rightarrow\langle\alpha_i,A_{x^{-1}}\rangle>-1\,,\,\forall \alpha_i\in\Sigma_\af.$$
Now, the number of such alcoves is $\leqslant (h+1)^r$.

Now we prove part (b). 
The diagonal coinvariant ring is $DR=\bfk[\frakt\times\frakt]/I_+$, 
where $I_+$ is the ideal
generated by homogeneous diagonal $W$-invariants of positive degrees. In \cite{H}, Haiman proved that
this ring has dimension $(h+1)^r$. 
Using the right $W_\af$-action and multiplications by Chern classes, a $DR$-action
on the Borel-Moore homology group $H_\bullet(\Fl)$ is defined in
\cite[thm.~2]{CO}. Further, let $[G/B]$ be the fundamental class corresponding to $G/B$ in $H_\bullet(\Fl)$. It was proved in loc.~cit. that 
the morphism 
$$DR\to H_\bullet(\Fl),\quad x\mapsto x [G/B]$$
is injective. In particular, the DR-submodule in $H_\bullet(\Fl)$ generated by $[G/B]$ has dimension $(h+1)^r$.
Now, consider $i_\ast: H_\bullet(\Fl_\gamma)\to H_\bullet(\Fl)$ given by push-forward with respect to the embedding 
$i:\Fl_\gamma\to\Fl$. Since $G/B$ is contained in $\Fl_\gamma$, the class $[G/B]$ belongs to the image of $i_\ast$.
Since the $DR$-action is constructed using the right $W_\af$-action and multiplications by Chern classes, it is easy to see that the image of $i_\ast$ is a $DR$-submodule. In particular, it contains the $DR$-submodule generated by $[G/B]$. By the injectivity
We deduce that
$\dim i_*(H_\bullet(\Fl_\gamma))\geqslant (h+1)^r$.

By taking the duals and identifying the dual of the Borel-Moore homology groups
$H_\bullet(\Fl_\gamma)$, $H_\bullet(\Fl)$  
with the cohomology groups
$H^\bullet(\Fl_\gamma)$, $H^\bullet(\Fl)$, we deduce that
$$\dim i^*(H^\bullet(\Fl))\geqslant (h+1)^r.$$
Since the left $W_\ex$-action on $H^\bullet(\Fl)$ is trivial, we have
$$i^*(H^\bullet(\Fl))\subset H^\bullet(\Fl_\gamma)^{W_\ex}.$$ 
But by part (a), we know that $\dim H^\bullet(\Fl_\gamma)^{W_\ex}\leqslant (h+1)^r$.
Thus we deduce 
$\dim H^\bullet(\Fl_\gamma)^{W_\ex}= (h+1)^r.$
\end{proof}

\smallskip

In particular, the proof above shows that in type $A$, we have $i^*(H^\bullet(\Fl))=H^\bullet(\Fl_\gamma)^{W_\ex}$.
We obtain the following corollary

\begin{corollary}\label{cor: surjA}
In type $A$, the pull-back for $i: \Fl_\gamma\to \Fl$ in cohomology yields a surjective map 
$$i^\ast: H^\bullet(\Fl)\to H^\bullet(\Fl_\gamma)^{W_\ex}.$$
\end{corollary}

Let us note the following result of Boixeda Alvarez and Losev, whose proof uses Proposition \ref{prop:dimformula}.
\begin{proposition}[\cite{BL}]\label{prop:PI}
\hfill
\begin{itemize}[leftmargin=8mm]
\item[$\mathrm{(a)}$] 
$H^\bullet(\Fl_\gamma)^{W_\ex}= \bfk[\LLbd/(h+1)\LLbd]$ as a right $W$-module.
\item[$\mathrm{(b)}$] 
$H^\bullet(\Fl_\gamma)^{W_\ex}=H^\bullet(\Fl_\gamma)^\LLbd$ in type $A$.
\end{itemize}
\qed
\end{proposition}

\smallskip

\smallskip

Let $e_1,\dots,e_r$ be the exponents of the group $W$. We'll prove the following.

\begin{theorem} \label{prop:formula} 
Assume that $\LG$ is simply connected. We have
\begin{equation}\label{eq:dim}
\dim H^\bullet({}^0\Gr^\zeta_{\gamma_\ell})^{W_{\ell,\ex}}=
\frac{1}{|W|}\prod_{i=1}^r\big((h+1)\ell-h+e_i\big).
\end{equation}
\end{theorem}

\smallskip

\begin{remark}\label{rem:LQ}
In type  $SL_{r+1}$ it is conjectured in \cite{LQ2} that 
$$\dim Z(\u_\zeta)=\frac{1}{(h+1)\ell}\begin{pmatrix}(h+1)\ell\\h\end{pmatrix}.$$
For $G=PSL_{r+1}$ the formula \eqref{eq:dim} gives
$$\dim H^\bullet({}^0\Gr^\zeta_{\gamma_\ell})^{W_{\ell,\ex}}=
\frac{1}{(r+1)!}\prod_{i=1}^r\big((h+1)\ell-i\big)
=\frac{1}{(h+1)\ell}\begin{pmatrix}(h+1)\ell\\h\end{pmatrix}.
$$
\end{remark}

\smallskip

\subsubsection{Relation with Lusztig's elliptic Springer fibres}
Let $\gamma\in\g\otimes\calK$ be an arbitrary regular semi-simple element.
Fix a subset $J\subset\Sigma_\af$.
Let $\Fl^J_\gamma$ and ${}^0\!\Fl^J_\gamma$ be the affine Steinberg and Spaltenstein 
fibers at $\gamma$.
Let $1_J$, $\varepsilon_J$ be the trivial and signature idempotents in $\bfk W_J$.
The cohomology of ${}^0\!\Fl^J_\gamma$ is given by the following proposition.
Set $d_J=\dim\Fl-\dim\Fl^J$.

\begin{proposition}\label{prop:BM}
%\hfill \begin{itemize}[leftmargin=8mm]
%\item[$\mathrm{(a)}$] 
%There is a commutative diagram
%$$\xymatrix{H^\bullet(\Fl)\ar[d]_-{i^*}&\ar[l]_-{p^*}H^\bullet(\Fl^J)\ar[d]^-{i^*}\\
%H^\bullet(\Fl_\gamma)&\ar[l]_-{q^*}H^\bullet(\Fl^J_\gamma)}$$
%where $i^*$ is the restriction and
%$q^*$ is a graded vector space isomorphism
%$H^\bullet(\Fl_\gamma)\cdot 1_J\simeq H^\bullet(\Fl^J_\gamma)$.
%\item[$\mathrm{(b)}$] 
There is a commutative diagram
$$\xymatrix{H^\bullet(\Fl)\ar[r]^-{p_*}\ar[d]_-{i^*}&H^\bullet(\Fl^J)[-2d_J]\ar[d]^-{i^*}\\
H^\bullet(\Fl_\gamma)\ar[r]^-{q_*}&H^\bullet({}^0\!\Fl^J_\gamma)[-2d_J]}$$
where $i^*$ is the restriction and $p_*$ is the Gysin map.
The map $q_*$ restricts to a graded vector space isomorphism from the space
$H^\bullet(\Fl_\gamma)\cdot\varepsilon_J$ of anti-invariants for the right $W$-action to
$H^\bullet({}^0\!\Fl^J_\gamma)[-2d_J].$
%\end{itemize}
\end{proposition}

\begin{proof}
%Since both parts are very similar, we'll only prove part (b).
The proof uses a sheaf theoretic interpretation of the affine 
Springer representation which goes back to Lusztig's work.
We have the following commutative diagram of stacks with Cartesian squares
$$\xymatrix{
\Fl_\gamma\ar[r]^-{\pi_2}\ar[d]_-{q_1}&\Fl^J_\gamma\ar[d]^-{q_2}&\ar[l]_-{i_2}{}^0\!\Fl^J_\gamma\ar[d]^-{{}^0q_2}\\
[\frakn_J/B_J]\ar[r]^-{\pi_1}&[\frakg^{nil}_J/G_J]&\ar[l]_-{i_1}[\pt/G_J]}
$$
where $B_J=G_J\cap \calI$ and $\frakn_J$ is the nilpotent radical of its Lie algebra.
The map $q_2$ maps $\delta_x$ to the image of $\Ad_x^{-1}(\gamma)$ in $\frakg_J$ for each
$x\in G(\calK)$. The map $q_1$ is defined similarly.
By base change, we have
$$H^\bullet(\Fl_\gamma)=H^\bullet(\Fl^J_\gamma,(\pi_2)_!q_1^*(\underline\bfk))
=H^\bullet(\Fl^J_\gamma,q_2^*(\pi_1)_!(\underline\bfk)).$$
The map $\pi_1$ is semi-small. Thus the complex $(\pi_1)_!(\underline\bfk)$ is semi-simple and contains
the direct summand $(i_1)_!(\underline\bfk)[-2d_J]$. Thus, by base change,
the projection onto this direct summand yields a map
$$q_*:H^\bullet(\Fl_\gamma)\to 
H^\bullet(\Fl^J_\gamma,q_2^*(i_1)_!(\underline\bfk))[-2d_J]$$
Note that the right hand side is equal to
$H^\bullet(\Fl^J_\gamma,(i_2)_!(\underline\bfk))[-2d_J]$, which is the same as $H^\bullet({}^0\!\Fl^J_\gamma)[-2d_J].$
So this map satisfies the properties in the proposition.
\end{proof}

\smallskip

Let $e_1^J,\dots,e_j^J$ be the exponents of the group $W_J$.
Let $\text{ht}(\alpha)$ be the height of a root $\alpha$.
Let $e_\alpha\in\frakg$ be a non-zero root vector of weight $\alpha\in\Phi$.
Lusztig has given a regular semi-simple element $l_n$ depending
on a positive integer $n$ which is elliptic if $(h,n)=1$, and such that
$$l_n=\sum_{hk-\text{ht}(\alpha)=n}t^ke_\alpha,$$
where the sum is over all couples $(k,\alpha)\in\bbZ\times\Phi$ as above.
%$$l_{ah+b}=t^{a+1}\sum_{\alpha\in\Phi_{h-b}}e_\alpha+t^a\sum_{\alpha\in\Phi_{-b}}e_\alpha,
%\qquad 0\leqslant b<h$$
Further, the following is proved in \cite{S97}
\begin{align}\label{sommers}
%H^\bullet(\Fl^J_{l_n})=\bfk[\LLbd/n\LLbd]\cdot 1_J,\qquad
(h,n)=1\Rightarrow \dim H^\bullet(\Fl^J_{l_n})=(n+e^J_1)\dots(n+e^J_j)n^{r-j}\,/\,|W_J|.
\end{align}
We'll abbreviate $\gamma'=\gamma'_1$
and 
$\gamma'_1=l_{h+1}.$
Let $\gamma=ts$ as before.
From Propositions \ref{prop:PI}, \ref{prop:BM} and \eqref{sommers} we deduce the following.

\begin{corollary}\label{cor:compare}
We have $H^\bullet({}^0\!\Fl_\gamma^J)^{W_\ex}=H^\bullet({}^0\!\Fl_{\gamma'}^J)$ as vector spaces.
\qed
\end{corollary}

\smallskip

We can now complete the proof of Theorem \ref{prop:formula}.

\begin{proof}[Proof of Theorem $\ref{prop:formula}$]
We set
$$\gamma'_\ell=\sum_{hk-\text{ht}(\alpha)=h+1}t^{\ell k}e_\alpha.$$
%$$\gamma'_\ell=t^{2\ell}\sum_{\alpha\in\Phi_{h-1}}e_\alpha+t^\ell\sum_{\alpha\in\Phi_{-1}}e_\alpha.$$
First, assume that $(h,\ell)=1$.
The regular semi-simple elements $t^{-1}\gamma'_\ell$ and $l_{(h+1)\ell-h}$ are elliptic and
homogeneous of the same weight.
Hence, we have $H^\bullet(\Gr_{t^{-1}\gamma'_\ell})\simeq H^\bullet(\Gr_{l_{(h+1)\ell-h}})$,
and \eqref{sommers} yields
\begin{align}\label{f1}
\begin{split}
\dim H^\bullet({}^0\Gr_{\gamma'_\ell})&=
\dim H^\bullet(\Gr_{t^{-1}\gamma'_\ell})\\
&=\dim H^\bullet(\Gr_{l_{(h+1)\ell-h}})\\
&=\frac{1}{|W|}\prod_{i=1}^r\big((h+1)\ell-h+e_i\big).
\end{split}
\end{align}
Since $\LG$ is simply connected, we have $\Lrho\in\LLbd$.
Further, the action of the element $(z^{-\ell\Lrho},z^h)$ in $T\times\bbG_m$
takes $\gamma'_\ell$ to $z^{(h+1)\ell}\gamma'_\ell$. 
Hence $\gamma'_\ell$ is fixed by the $\mu_\ell$-action by loop rotation, and
we may consider both projective varieties  ${}^0\Gr_{\gamma'_\ell}$
and ${}^0\Gr_{\gamma'_\ell}^\zeta$.
Since $(h,\ell)=1$, the $\mu_\ell$-action on ${}^0\Gr$
by loop rotation extends to a $\bbG_m$-action which preserves ${}^0\Gr_{\gamma'_\ell}$.
By \cite[thm.~1.2]{GKM1}, the variety ${}^0\Gr_{\gamma'_\ell}$ has a $\bbG_m$-equivariant
affine paving.
Therefore, the Grothendieck group $K_{\bbG_m}({}^0\Gr_{\gamma'_\ell})$ is free as an 
$R_{\bbG_m}$-module, the Chern character yields isomorphisms of $\bfk$-vector spaces
$$H^\bullet({}^0\Gr^\zeta_{\gamma'_\ell})=K({}^0\Gr_{\gamma'_\ell}^\zeta),
\qquad
H^\bullet({}^0\Gr_{\gamma'_\ell})=K({}^0\Gr_{\gamma'_\ell}),$$
and for any closed subgroup $H\subset\bbG_m$, we have
$$K_H({}^0\Gr_{\gamma'_\ell})=
K_{\bbG_m}({}^0\Gr_{\gamma'_\ell})\otimes_{R_{\bbG_m}}R_H.$$
Thus, the localization theorem in $K$-theory implies that
\begin{align}\label{fb}\dim H^\bullet({}^0\Gr^\zeta_{\gamma'_\ell})=
\dim K({}^0\Gr_{\gamma'_\ell}^\zeta)=\dim K({}^0\Gr_{\gamma'_\ell})=\dim H^\bullet({}^0\Gr_{\gamma'_\ell}).
\end{align}
Next, from Lemma \ref{lem:isom2}(a) we deduce that
\begin{align*}
H^\bullet({}^0\Gr^\zeta_{\gamma_\ell})^{W_{\ell,\ex}}
&=\bigoplus_{\omega\in \Xi}H^\bullet({}^0\Fl^\omega_{\ell,\gamma_\ell})^{W_{\ell,\ex}}
=\bigoplus_{\omega\in \Xi}H^\bullet({}^0\Fl^\omega_{\gamma})^{W_\ex}\\
H^\bullet({}^0\Gr^\zeta_{\gamma'_\ell})
&=\bigoplus_{\omega\in \Xi}H^\bullet({}^0\Fl^\omega_{\ell,\gamma'_\ell})
=\bigoplus_{\omega\in \Xi}H^\bullet({}^0\Fl^\omega_{\gamma'}).
\end{align*}
Thus, Corollary \ref{cor:compare} implies that
\begin{align}\label{fa}
H^\bullet({}^0\Gr^\zeta_{\gamma_\ell})^{W_{\ell,\ex}}=H^\bullet({}^0\Gr^\zeta_{\gamma'_\ell}).
\end{align}
From \eqref{f1},  \eqref{fb} and \eqref{fa} we deduce that if $(h,\ell)=1$ then
\begin{align}\label{prime}
\dim H^\bullet({}^0\Gr^\zeta_{\gamma_\ell})^{W_{\ell,\ex}}=
\frac{1}{|W|}\prod_{i=1}^r\big((h+1)\ell-h+e_i\big).
\end{align}
Now, we claim that the integer $\dim H^\bullet({}^0\Gr^\zeta_{\gamma_\ell})^{W_{\ell,\ex}}$ is a polynomial in $\ell$.
Thus the equality \eqref{prime} holds for any $\ell$.
To prove the claim, recall first that
$$\dim H^\bullet({}^0\Gr^\zeta_{\gamma_\ell})^{W_{\ell,\ex}}=
\sum_{\omega\in\Xi}\dim H^\bullet({}^0\Fl^\omega_{\ell,\gamma_\ell})^{W_{\ell,\ex}}.$$
Since the dimension of $H^\bullet({}^0\Fl^\omega_{\ell,\gamma_\ell})^{W_{\ell,\ex}}$ does not depend on $\ell$,
we must check that for each subset $J\subset\Sigma_\af$, the cardinal of the set 
$\{\omega\in\Xi\,;\,J_\omega=J\}$ is a polynomial in $\ell$.
We have $\Xi_\sc\simeq \bar A_{\ell,1}\cap\LLbd$ and
$\pi_1$ acts freely on $\Xi_\sc$ with the quotient equal to $\Xi$, see \S\ref{sec:Xi}.
Hence, it is enough to check that the cardinal of the set 
$$\{\mu\in\frakt_\bbR\,;\, 0\leqslant \langle\alpha,\mu\rangle\leqslant \ell\,,\,\forall\alpha\in\Phi^+\,,\,
\mu\ \text{of\ type}\ J\}\cap\LLbd$$ 
is a polynomial in $\ell$.
This number is given by the Ehrhart polynomial of the interior of the face of type $J$ of the alcove
(which is an integral polytope), proving the claim.

\end{proof}

\bigskip

%%%%%%%%%%%%%%%%%%%%%%%%%%%%%%%%%%%%%%%%%%%%%%%%

\section{Quantum groups and their representations}\label{s:hybrid}

From now on we assume $\LG$ to be quasi-simple, connected and simply connected.
Thus, we have $\Lbd=Q$, $\LLbd=\LP$ and $\pi_1=\LP/\LQ=\X^*(\LZ)=\Hom(\LZ,\bbG_m)$.
We abbreviate $d_i=d_{\Lalpha_i}$.

\subsection{Quantum groups}

\subsubsection{Quantum groups}
Let $q$ be a formal variable.
We abbreviate $q_i=q^{d_i}$ for each $i$.
The quantum group $\scrU_q$ associated with 
$\LG$ is the $\bfF$-algebra
generated by $E_i$, $F_i$, $K_\lambda$ for $i\in I$, $\lambda\in\LLbd$, subject to the
quantum Chevalley relations,  and
\begin{align*}
&K_\lambda K_\mu=K_{\lambda+\mu}, \qquad K_0=1,\qquad K_i=K_{\Lalpha_i},\\
&K_\lambda E_i K_{-\lambda}=q_i^{(\lambda,\alpha_i)} E_i=q^{(\lambda,\Lalpha_i)} E_i,\qquad
K_\lambda F_i K_{-\lambda}=q_i^{-(\lambda,\alpha_i)} F_i=q^{-(\lambda,\Lalpha_i)} F_i,\\
&[E_i, F_j]=\delta_{i,j}\frac{K_i-K_i^{-1}}{q_i-q^{-1}_i}.
\end{align*}
%where the quantum binomial coefficient is
%$$\begin{bmatrix} m\\ n
%\end{bmatrix}_{q_i}=\frac{\prod_{r=1}^{m}(q_i^{m+1-r}-q_i^{r-1-m})}{\prod_{r=1}^n(q_i^r-q_i^{-r})}.$$
Let $\scrU_q^+$, $\scrU_q^0$, $\scrU_q^-$ be respectively subalgebras generated by 
$E_i$'s, $K_\lambda$'s, $F_i$'s. 
Multiplication yields a canonical isomorphism of vector spaces
$\scrU_q^-\otimes\scrU_q^0\otimes\scrU_q^+= \scrU_q.$
Given any weight $\mu\in\LLbd$, let $\varepsilon_\mu:\scrU_q^0\to\bfK$ be the algebra homomorphism taking
$K_\lambda$ to $q^{(\lambda,\mu)}$ for each $\lambda\in\LLbd.$

\smallskip

\subsubsection{Integral forms}
For each $m\in\bbZ$ and $n\in\bbN$ we introduce the following elements in $\scrU_q^0$
\begin{align*}
\begin{bmatrix}K_i;m\\ n\end{bmatrix}_{q_i}=
\prod_{r=1}^n\frac{K_i q_i^{m+1-r}-K_i^{-1}q_i^{r-m-1}}{q_i^r-q_i^{-r}}.
%\qquad [K_i;m]_{q_i}=\begin{bmatrix}K_i;m\\ 1\end{bmatrix}_{q_i}.
\end{align*}
We have a chain of $\bfA$-subalgebras of $\scrU_q$
$$\frakU_q\subset \U_q^{\hyb, n}\subset \U_q^{\hyb, b}\subset \U_q$$
defined as follows:
\begin{itemize}[leftmargin=8mm]
\item $\frakU_q$ is generated by $E_i,\ F_i,\ K_\lambda$ for $i\in I$, $\lambda\in\LLbd$,
\item $\U_q^{\hyb, n}$ is generated by  $E_i^{(n)},\ F_i,\ K_\lambda$ for $i\in I$, $\lambda\in\LLbd$, $n>0$,
\item $\U_q^{\hyb, b}$ is generated by  
$E_i^{(n)},\ F_i,\ K_\lambda,\ \left [\!\begin{smallmatrix}K_i;m\\ n\end{smallmatrix}\!\right ]_{q_i}$ for 
$i\in I$, $\lambda\in\LLbd$, $n>0$, $m\geqslant 0$,
\item $\U_q$ is generated by $E_i^{(n)},\ F_i^{(n)}, K_\lambda$ for $i\in I$, $\lambda\in\LLbd$, $n>0$.
\end{itemize}
Write $\U_q^{\hyb,\flat}=\U_q^{\hyb,b}$ or $\U_q^{\hyb,n}$.
We call them the hybrid quantum groups.
We have the triangular decompositions
\begin{align*}
&\U_q^-\otimes\U_q^0\otimes\U_q^+= \U_q,\qquad
\frakU_q^-\otimes\frakU_q^0\otimes\frakU_q^+= \frakU_q,\\
&\frakU_q^-\otimes\U_q^0\otimes\U_q^+= \U^{\hyb, b}_q,\qquad
\frakU_q^-\otimes\frakU_q^0\otimes\U_q^+= \U^{\hyb, n}_q
\end{align*}
where
\begin{align*}
&\U_q^+=\langle E_i^{(n)}\rangle,\quad
\U_q^-=\langle F_i^{(n)}\rangle,\quad
\U_q^0=\langle K_\lambda, \left [\!\begin{smallmatrix}K_i;m\\ n\end{smallmatrix}\!\right ]_{q_i}\rangle,\\
&\frakU_q^+=\langle E_i\rangle,\quad
\frakU_q^-=\langle F_i\rangle,\quad
\frakU_q^0=\langle K_\lambda, [K_i;0]_{q_i}\rangle.
\end{align*}
Note that $\zeta$ and $\zeta_e$ are both primitive $\ell$-th roots of unity. 
For all $\bfA$-algebra $A_q$ above, let the $\bfk$-algebra
$A_\zeta=\bfk\otimes_\bfA A_q$ be the specialisation to $q_e=\zeta_e$.
For $A=A_q$ or $A_\zeta$ we have a triangular decomposition
$A=A^-\otimes A^0\otimes A^+$. 
We abbreviate $A^\leqslant=A^-A^0$ and $A^\geqslant=A^0A^+$.
The algebra homomorphism $\varepsilon_\mu$ restricts to a map
$\varepsilon_\mu:\U_q^0\to\bfA$.
It specializes to an algebra homomorphism $\varepsilon_\mu:A^0_\zeta\to\bfk$.

\medskip

\subsection{Centers}

\subsubsection{The center of $\scrU_q$}\label{sec:HC}
We consider the $\circ$-action of $W$ on $\scrU_q^0$ by algebra automorphisms such that
$w\circ K_\lambda=q^{(w\lambda-\lambda,\Lrho)}K_{w\lambda}$ for all $\lambda\in\LLbd$.
The Harish-Chandra map is an $\bfF$-algebra isomorphism
\begin{equation*}
Z(\scrU_q)\to (\scrU_q^{0,ev})^{W\circ},\qquad
\scrU_q^{0,ev}=\bigoplus_{\lambda\in\LLbd}\bfF K_{2\lambda},
\end{equation*}
such that the following diagram commutes, see, e.g., \cite{T12},
\begin{align*}
\xymatrix{Z(\scrU_q)\ar@{^{(}->}[r]\ar[d]_-\iota&\scrU_q=\scrU_q^-\otimes\scrU_q^0\otimes\scrU_q^+
\ar[d]^-{\varepsilon\otimes 1\otimes \varepsilon}\\
(\scrU_q^{0,ev})^{W\circ}\ar@{^{(}->}[r]&\scrU^0_q
}
\end{align*}
where $\varepsilon$ is the co-unit of $\scrU_q$.
Write
$\bfF[\LT]=\bfF[e^\lambda\,;\,\lambda\in\LLbd]$.
We have the $\bfF$-algebra isomorphisms 
\begin{align}\label{cq}(\scrU_q^{0,ev})^{W\circ}\to\bfF[\LT]^{W\circ}\to\bfF[\LT/W],
\qquad f(K_{2\lambda})\mapsto f(e^\lambda)\mapsto f(q^{-2(\lambda,\Lrho)}e^\lambda)
\end{align}
where the $\circ$-action of $W$ on $\LT$ is the Weyl group action centered at $q^{2\Lrho}$.
Composing the Harish-Chandra map with \eqref{cq}
we get an isomorphism
\begin{align}\label{phiq}
\bfhc:Z(\scrU_q)\to\bfF[\LT/W].\end{align}

\smallskip

\subsubsection{The center of $\frakU_\zeta$}
The centers of the integral forms $\frakU_q$ and $\U_q$ are given by
$$Z(\frakU_q)=Z(\scrU_q)\cap \frakU_q,\qquad Z(\U_q)=Z(\scrU_q)\cap \U_q.$$
The Harish-Chandra center of $\frakU_\zeta$ is the subalgebra of $Z(\frakU_\zeta)$ given by
$$Z_\HC=Z(\frakU_q)\,/\,(q_e-\zeta_e)Z(\frakU_q).$$ 
The map \eqref{phiq} yields the Harish-Chandra isomorphisms
\begin{align}\label{phizeta}
\bfhc:Z(\frakU_q)\to\bfA[\LT/W],\qquad
\bfhc:Z_\HC\to\bfk[\LT/W]\end{align}
The inverse Harish-Chandra homomorphism for $\frakU_q$ is the map
\begin{align}\label{HC00}\xymatrix{\overline\bfhc:\bfA[\LT/W]\ar[r]&Z(\frakU_q)},\qquad
 f(e^\lambda)\mapsto\iota^{-1}f(q^{2(\lambda,\Lrho)}K_{2\lambda}).\end{align}
The inverse Harish-Chandra homomorphism for $\frakU_\zeta$ is the map
\begin{align}\label{HC0}\xymatrix{\bfk[\LT/W]\ar[r]^-{\overline\bfhc}&Z_\HC\ar@{^{(}->}[r]&Z(\frakU_\zeta)},\qquad
 f(e^\lambda)\mapsto\iota^{-1}f(\zeta^{2(\lambda,\Lrho)}K_{2\lambda}).\end{align}
For each positive coroot $\Lalpha\in\LPhi^+$ there are coroot vectors 
$E_\Lalpha, $ $F_\Lalpha$ in $\frakU_\zeta$ such that
$E_{\Lalpha_i}=E_i$ and $F_{\Lalpha_i}=F_i$.
The Frobenius center is the $\bfk$-subalgebra $Z_\Fr$ of $Z(\frakU_\zeta)$ generated by
$\{ F_\Lalpha^\ell, K_{\ell\lambda},E_\Lalpha^\ell\,;\,\Lalpha\in\LPhi^+,\ \lambda\in\LLbd\,\}.$ 
Let $\LG^\ast= \LN^-\times\LT\times\LN$ be the Poisson dual group and
$\kappa : \LG^\ast\to \LT/W$ be the map taking $(n_1,t,n_2)$ to the $W$-orbit of the semi-simple
part of $n_2t^2n_1^{-1}$.
There is a $\bfk$-algebra isomorphism $Z_\Fr= \bfk[\LG^\ast]$.
According to De Concini-Kac-Procesi, we have an isomorphism
$$Z_\Fr\cap Z_\HC=\bfk[e^{\ell\lambda}\,;\,\lambda\in\LLbd]^W$$
such that the inclusion $Z_\Fr\cap Z_\HC\subset Z_\Fr$ takes $f(e^{\ell\lambda})$ to $\kappa^*(e^{2\lambda})$,
and the inclusion $Z_\Fr\cap Z_\HC\subset Z_\HC$ takes $f(e^{\ell\lambda})$ to itself.
Further, we have
\begin{eqnarray*}
Z(\frakU_\zeta)=Z_\Fr\otimes_{Z_\Fr\cap Z_\HC}Z_\HC=\bfk[\LG^\ast\times_{\LT/W}\LT/W].
\end{eqnarray*}

The obvious inclusion of $Z(\frakU_q)$ into $Z(\U_q)$ specializes to a map $Z(\frakU_\zeta)\to Z(\U_\zeta)$.
Composing these maps with \eqref{HC00} and \eqref{HC0}, 
we get  the inverse Harish-Chandra homomorphisms for $\U_q$ and $\U_\zeta$
\begin{align}\label{IHC2}
\xymatrix{\overline\bfhc:\bfA[\LT/W]\ar[r]&Z(\U_q)},\qquad
\xymatrix{\overline\bfhc:\bfk[\LT/W]\ar[r]&Z(\U_\zeta)}.\end{align}
The obvious map $Z(\frakU_\zeta)\to Z(\U_\zeta)$
is the composition of the restriction $\bfk[\LG^\ast\times_{\LT/W}\LT/W]\to\bfk[\{1\}\times\LOm]$ and
the inverse Harish-Chandra homomorphism $\bfk[\LOm]\to Z(\U_\zeta)$.
The inverse Harish-Chandra homomorphism for $\U_\zeta$ factorizes through the restriction
$\bfk[\LT/W]\to\bfk[\LOm]$.

\smallskip

\subsubsection{The center of $\u_\zeta$}
The small quantum group $\u_\zeta$ is the image of the obvious map $\frakU_\zeta\to \U_\zeta$.
It is isomorphic to the specialization $\frakU_\zeta\otimes_{Z_\Fr}\bfk$ with respect to the character 
of $Z_\Fr$ such that $E_\Lalpha^\ell,\ F_\Lalpha^\ell\mapsto 0,\ K_{\ell\lambda}\mapsto 1.$
Let $U(\Lg)$ be the enveloping algebra of $\Lg$.
The quantum Frobenius map is the algebra homomorphism given by
\begin{align}\label{Fr1}
\Fr:\U_\zeta\to U(\Lg),\qquad
 E_\Lalpha^{(n)}, F_\Lalpha^{(n)} \mapsto \begin{cases} E_\Lalpha^{(n/\ell)},  
 F_\Lalpha^{(n/\ell)} &\text{ if } \ell\,|\, n,\\
0 &\text{ otherwise,}\end{cases}\qquad
K_\lambda \mapsto  1.
\end{align}
The Hopf adjoint action of $\U_\zeta$ on itself preserves the small quantum group $\u_\zeta$.
When restricted to the center $Z(\u_\zeta)$, this action factorizes through $U(\Lg)$.
Hence, there is a $\LG$-action on $Z(\u_\zeta)$. We have
\begin{equation}\label{Ginv}
Z(\u_\zeta)^\LG=Z(\U_\zeta)\cap \u_\zeta.
\end{equation}

\smallskip

\subsubsection{The center of $\frakU_\zeta^b$}
Consider the ring homomorphisms
\begin{align*}
&Z_\Fr\to \bfk[\LT],\quad E_\Lalpha^\ell\,,\,F_\Lalpha^\ell\,,\,
K_{\ell\lambda}\mapsto 0\,,\,0\,,\,K_{\ell\lambda},\\
&Z_\Fr\to\bfk[\LB^-],\quad E_\Lalpha^\ell\,,\,F_\Lalpha^\ell\,,\,
K_{\ell\lambda}\mapsto 0\,,\, F_\Lalpha^\ell\,,\, K_{\ell\lambda}.
\end{align*}
We define the following algebras
\begin{align}
\label{hyb2}\frakU_\zeta^t=\frakU_\zeta\otimes_{Z_\Fr}\bfk[\LT],\qquad
\frakU_\zeta^b=\frakU_\zeta\otimes_{Z_\Fr}\bfk[\LB^-].
\end{align}
Note that $\frakU_\zeta^b$ naturally identifies with the image 
of the De Concini-Kac quantum group  $\frakU_\zeta$ in 
the hybrid quantum group $\U_\zeta^{\hyb,n}$. 
The canonical surjection $\frakU_\zeta\to \frakU_\zeta^b$ induces an algebra homomorphism
$Z(\frakU_\zeta)\to Z(\frakU_\zeta^b)$, which factorizes through the restriction
$$\bfk[\LG^\ast\times_{\LT/W}\LT/W]\to\bfk[\LB^-\times_{\LT/W}\LT/W].$$
The inverse Harish-Chandra isomorphism for $\frakU_\zeta^b$ is the map
\begin{align}\label{HC4}\bfk[\LT\times_{\LT/W}\LT/W]=\bfk[\LLbd]\otimes_{\bfk[\ell\LLbd]^W}\bfk[\LLbd]^W\to Z(\frakU^b_\zeta)
\end{align}
given by the composition of the map 
$\overline\bfhc$ with the obvious map $\bfk[\LT]\to \bfk[\LB^-]$ and \eqref{hyb2}.
%Here the map $\LT/W\to\LT/W$ is $[\ell]$ and the map $\LT\to \LT/W$ is the categorical quotient.
We may also consider the image $\U_\zeta^{\frac{1}{2}}$ of the hybrid quantum group $\U_\zeta^{\hyb, b}$ in 
the Lusztig quantum group $\U_\zeta$, and the image $\frakU_\zeta^n$ of 
$\frakU_\zeta$ in $\Ub_\zeta$. 
We have the triangular decompositions
\begin{align}\label{1/2}
\begin{split}
&\u_\zeta^-\otimes\U_\zeta^0\otimes\U_\zeta^+=\U_\zeta^{\frac{1}{2}},\qquad
\u^-_\zeta\otimes\frakU^0_\zeta\otimes\u^+_\zeta=\frakU_\zeta^t,\\
&\frakU^-_\zeta\otimes\frakU^0_\zeta\otimes\u^+_\zeta=\frakU_\zeta^b,\qquad
\frakU_\zeta^-\otimes\u_\zeta^0\otimes\u_\zeta^+=\frakU_\zeta^n.
\end{split}
\end{align}
%Let $\widehat\frakU_\zeta^\sharp$ be the corresponding completion of $\frakU_\zeta^\sharp$
%where $\sharp$ is either $t$ or $b$.
%The $\bfk$-algebra $\widehat\frakU_\zeta^\sharp$ is topologically free over the central subalgebra $S^\sharp$. 
We write
$\frakU^\sharp_\zeta=\frakU_\zeta^t$,
$\frakU_\zeta^b$ or $\frakU_\zeta^n$ with
$\sharp=\frakt$, $\frakb$ or $\frakn$.

\medskip

\subsection{Representations}\label{sec:rep}
Let $S^t=\bfk[\LT]_{\hat 1}$ and $S^b=\bfk[\LB^-]_{\hat 1}$ 
be the completions of $\bfk[\LT]$ and $\bfk[\LB^-]$ at $1$.
We may abbreviate $S=S^t$ and $\bfk=S\,/\,\frakt S$.
Let $\iota:\frakU^0\to S$ be the algebra homomorphism such that $K_\lambda\mapsto e^\lambda$.
By an $S$-algebra $R$ we'll always mean a commutative, Noetherian $S$-algebra. 
Unless specified otherwise, we'll also assume that  $R$ is a local integral domain with fraction field $K$,
residue field $F$ and structure morphism $\pi:S\to R$. 
We'll be mainly interested by the case $F=\bfk$.

\smallskip

\subsubsection{Module categories}\label{sec:mod}

First, let $A$ be any of the algebras $\U_\zeta$, $\U^\leqslant_\zeta$, $\U^\geqslant_\zeta$, $\U_\zeta^{\frac{1}{2}}$ or 
$\U_\zeta^0$.
A $\LLbd$-graded $A$-module $M=\bigoplus_{\mu\in\LLbd}M_\mu$ is integrable if 
for any $m\in M$, we have $E_i^{(r)}m=0=F_i^{(r)}m$ for $r$ large enough, and if
\begin{align}\label{weightsp}
M_\mu=\{m\in M\,;\,um=\varepsilon_\mu(u)m\,,\, u\in A^0\}.
\end{align}
Let $\Rep(A)$ be the category of integrable finite dimensional modules.

Next, let $\Rep(\u_\zeta)$ be the category of all finite dimensional $\u_\zeta$-modules.
A finite dimensional representation of $\LG$ is the same as a finite dimensional representation of $U(\Lg)$.
Hence, the pullback by the quantum Frobenius map yields an exact functor, called the Frobenius twist,
\begin{align}\label{Fr2}\Fr^\ast: \Rep(\LG) \to \Rep(\U_\zeta).\end{align}
Thus, the category $\Rep(\U_\zeta)$ is a tensor category over the tensor category $\Rep(\LG)$.
By \cite{AG} we have an equivalence 
\begin{align}\label{AG}\Rep(\u_\zeta)=\Vect(\bfk)\otimes_{\Rep(\LG)}\Rep(\U_\zeta).\end{align}

Finally, let $A=\U^{\hyb, n}_\zeta$ or $\frakU^\sharp_\zeta$.
Recall that $A^0=\bfk[\LT]$ and that $S$ is a completion of $A^0$ via the inclusion
$\iota:A^0\to S$.
For each $\mu\in\LLbd$, the automorphism  $f\mapsto{}^{\tau_\mu}\!f$ of $A^0$ 
in \S\ref{sec:RD} 
extends to an automorphism of $S$ such that
${}^{\tau_\mu}\!f(t)=f(\zeta^\mu t)$ for each $f\in S$, $t\in\LT$ and
$fa=a\,{}^{\tau_\mu}\!f$ for each homogeneous element $a\in A$ of weight $\mu$ and each $f\in A^0.$
Following \cite[\S\S 1.4, 2.3]{AJS}, we define $A\mod^\LLbd_R$
to be the category of $\LLbd$-graded finitely generated $(A, R)$-bimodules $M$ 
such that for each $\mu\in\LLbd$ the subset $M_\mu\subset M$ is an $R$-submodule killed by
all elements of $A\times R$ of the form
\begin{align}\label{AJS}(f,-\pi({}^{\tau_\mu}\!f)),\qquad
f\in\frakU^0.\end{align}
The right $R$-action on any object yields a ring morphism $R\to Z( A\mod^\LLbd_R)$, i.e., 
the category $A\mod^\LLbd_R$ is $R$-linear.
There is an obvious $R$-linear functor
from the idempotent completion of the $R$-linear category obtained from
$A\mod^\LLbd_S$ by base change on morphisms to $A\mod^\LLbd_R.$
We'll write $-\otimes_SR$ for the corresponding functor
$A\mod^\LLbd_S\to A\mod^\LLbd_R.$
We'll abbreviate 
$\u_\zeta\mod^\LLbd=\frakU_\zeta^t\mod^\LLbd_\bfk$.
It is the category of the $\LLbd$-graded finite dimensional $\u_\zeta$-modules, i.e., 
the finite dimensional modules on which the action of the $K_\lambda$'s comes from a $\LLbd$-grading.
We'll also need the following categories: 

\begin{itemize}[leftmargin=8mm]

\item set $\calO^{\hyb,n}_{\zeta,R}$ to be the full subcategory of $\U^{\hyb,n}_\zeta\mod^\LLbd_R$ 
of modules with a locally unipotent $\U_\zeta^+$-action, 

\item set $\calO^{\hyb,b}_\zeta$ to be the category of finitely generated $\U^{\hyb, b}_\zeta$-modules with integrable 
$\U_\zeta^\geqslant$-action.
\end{itemize}
We abbreviate $\U_\zeta^\hyb=\U_\zeta^{\hyb,n}$, $\calO^\hyb_{\zeta,R}=\calO^{\hyb,n}_{\zeta,R}$ and
$\calO^\hyb_\zeta=\calO^{\hyb,n}_{\zeta,\bfk}$.
The following lemma is obvious.

\smallskip

\begin{lemma}\label{lem:resb}
There is a canonical equivalence of categories
$\calO^{\hyb,b}_\zeta\simeq \calO^\hyb_\zeta$.
\qed
\end{lemma}

\smallskip
\subsubsection{Standard and simple modules}
First, we consider the algebra $A=\U^\hyb_\zeta$ or $\frakU^\sharp_\zeta$.
For each weight $\mu\in \LLbd$, let $R_\mu$ be the $\Lbd$-graded $(A^0,R)$-bimodule such that 
$$(R_\mu)_\lambda=\delta_{\lambda,\,\mu} R,\qquad
K_\lambda s=\zeta^{(\lambda,\mu)} s\pi(e^{\lambda}),\qquad
s\in R_\mu.$$
We may view $R_\mu$ as an $(A^\geqslant, R)$-bimodule under the projection $A^\geqslant\to A^0$.
We define
$$\Delta^A(\mu)=A\otimes_{A^\geqslant}R_\mu\in A\mod^\LLbd_R.$$ 
Now we describe the standard and simple modules case by case.
\begin{itemize}[leftmargin=8mm]

\item If $A=\U^\hyb_\zeta$ then $\Delta^A(\mu)=\M(\mu)_R$ is a deformed Verma module in 
$\calO^{\hyb}_{\zeta,R}$. 
We have $\M(\mu)_R= \frakU^-_\zeta\otimes R_\mu$ as $A^\leqslant$-module. 
If $R=F$ is a field then $\M(\mu)_F$ has a unique simple quotient $\E(\mu)_F$. 
The set $\{\E(\mu)_F\,;\,\mu\in\LLbd\}$ is a complete non redundant list of representative of isomorphism 
classes of simple modules in $\calO^{\hyb}_{\zeta,F}$.
We abbreviate $\M(\mu)=\M(\mu)_\bfk$ and $\E(\mu)=\E(\mu)_\bfk$. 

%\item If $A=\U^{\hyb,b}_\zeta$ then $\Delta^A(\mu)=\M(\mu)_\k$ is a Verma module in $\calO^\hyb_\zeta$.
%We abbreviate $\M(\mu)=\M(\mu)_\bfk$ and $\E(\mu)=\E(\mu)_\bfk$. 

\item If $A=\frakU_\zeta^t$ then 
$\Delta^A(\mu)=\Delta(\mu)^t_R$ is a deformed baby Verma module in $\frakU_\zeta^t\mod^\Lbd_R$
for each weight $\mu\in\LLbd$.
We have $\Delta(\mu)^t_R= \u^-_\zeta\otimes R_\mu$ as $A^\leqslant$-module.
 If $R=F$ is a field, 
the module $\Delta(\mu)^t_F$ has a unique simple quotient $\L(\mu)^t_F$.
The set $\{\L(\mu)^t_F\,;\,\mu\in\LLbd\}$ is a complete non redundant 
list of simple objects in $\frakU_\zeta^t\mod^\Lbd_F$.

\item If $A=\frakU_\zeta^b$ we write $\Delta^A(\mu)=\Delta(\mu)^b_R$.
We have $\Delta(\mu)^t_R=\Delta(\mu)^b_R\otimes_{S^b}S^t$ because 
$\frakU_\zeta^t=\frakU_\zeta^b\otimes_{S^b}S^t$.
We have $\Delta(\mu)^b_R= \frakU^-_\zeta\otimes R_\mu$ as $A^\leqslant$-module.
If $R=F$ is a field, the module $\Delta(\mu)^b_F$ has a unique simple quotient $\L(\mu)^b_F$. 
The set $\{\L(\mu)^b_F\,;\,\mu\in\LLbd\}$ is a complete non redundant 
list of simple objects in $\frakU_\zeta^b\mod^\Lbd_F$.

\item 
If $A=\U_\zeta$ then $\Delta^A(\mu)=\V(\mu)$ is the Weyl module in $\Rep(\U_\zeta)$
for each dominant weight $\mu\in\LLbd^+$.
It has a unique simple quotient $\D(\mu)$.
The set $\{\D(\mu)\,;\,\mu\in\LLbd^+\}$ is a complete non redundant list of representatives of isomorphism classes of
simple modules in $\Rep(\U_\zeta)$.

\item
If $A=\u_\zeta$ then $\Delta^A(\mu)=\Delta(\mu)_\bfk^t$ is a baby Verma module in $\Rep(\u_\zeta)$. 
We abbreviate $\Delta(\mu)=\Delta(\mu)^t_\bfk$ and $\L(\mu)=\L(\mu)_\bfk^t$ for each
restricted dominant weight $\mu\in\LLbd^+_\ell$. 
The set $\{\L(\mu)\,;\,\mu\in\LLbd^+_\ell\}$ is a complete non redundant 
list of simple objects in $\Rep(\u_\zeta)$. 
\end{itemize}
According to \cite{APW1}, the Weyl module $\V(\mu)$ lifts to a module over the $\bfA$-algebra $\U_q$.
To avoid confusions we may write $\V(\mu)_\zeta$ for the Weyl module over $\U_\zeta$ and
$\V(\mu)_q$ for the Weyl module over $\U_q$. We may also write $\D(\mu)_\zeta$ for the simple
$\U_\zeta$-module $\D(\mu)$. In particular $\D(\mu)_1$ denotes the simple $\LG$-module with highest weight $\mu$.

\smallskip

\subsubsection{Deformation of the quantum parameter}\label{sec:generification}
We'll need $\bfk[[\hbar]]$-linear 
version of the $\bfk$-categories considered above.
The first one is the category $\Rep(\U_{\hat\zeta})$ of integrable $\U_q$-modules
which are free of finite rank over $\bfk[[\hbar]]$.
We view it as a deformation of the category $\Rep(\U_\zeta)$ whose generic point 
is the $\bfF$-category $\Rep(\scrU_q)$ of integrable finite dimensional $\scrU_q$-modules. 
The second category is a deformation of the category $\calO^\hyb_{\zeta,S}$.
Recall that $\frakU_q^0=\bfA[\LT]$.
Let $S[[\hbar]]$ be the completion of $\bfA[\LT]$ at the point $(t,q_e)=(1,\zeta_e)$.
Fix any $S[[\hbar]]$-algebra $R$.
Let
$\calO^{\hyb}_{q,R}$ 
be the full subcategory of $\U^{\hyb}_q\mod^\LLbd_R$ 
consisting of those modules with locally unipotent $\U_q^+$-action. 
Let $\V(\mu)_{\hat\zeta}$ be the Weyl module in $\Rep(\U_{\hat\zeta})$, and
$\M(\mu)_R$ the Verma module in $\calO^{\hyb}_{q,R}$. We have
$$\M(\mu)_R=\U_q^\hyb\otimes_{\U^{\hyb,\geqslant}_q}R_\mu= \frakU^-_q\otimes R_\mu$$ 
as a $\U^{\hyb,\leqslant}_q$-module. 

Given an $S[[\hbar]]$-algebra $F$ which is a field,
let the set $\Phi_F$ be as in \eqref{PhiF}.
The algebra $\U^\hyb_q$ receives an obvious map from $\frakU_q$ which takes
the center $Z(\frakU_q)$ to $Z(\U^\hyb_q)$. 
Hence, the algebra $\bfA[\LT/W]$
acts on the module $\M(\mu)_R$ by scalar multiplication 
by a character $\chi_{q,\mu}$ of $\bfA[\LT/W]$
under the inverse Harish-Chandra homomorphism 
$\overline\bfhc$ in \eqref{HC00}.
The algebra homomorphism 
$\bfA[\LT]\to R$ composed of the inclusion $\bfA[\LT]\subset S[[\hbar]]$
and the map $\pi:S[[\hbar]]\to R$ yields the 
$R$-point $\xi_\pi$ in $\LT(R)$ such that
$e^\lambda(\xi_\pi)=\pi(K_\lambda)$ for all $\lambda\in\LLbd$.
The central character $\chi_{q,\mu}$ is identified with the point
$W(q^{2(\mu+\Lrho)}\xi_\pi)$ in $\LT(R)/W$.

\begin{lemma}\label{lem:mxss}
If $\Phi_F=\emptyset$, then
the category $\calO^{\hyb}_{q,F}$ is semi-simple.
\end{lemma}

\begin{proof}
By the linkage principle, we have
$$[\M(\mu)_F:\E(\nu)_F]\neq 0\Rightarrow q^{2(\mu+\Lrho)}\xi_\pi\in W(q^{2(\nu+\Lrho)}\xi_\pi).$$ 
Since $\Phi_F=\emptyset$, we have $(\xi_\pi)^{2\ell\alpha}\neq 1$ for any root $\alpha$,
hence $(\xi_\pi)^\ell$ is regular in $\LT(F)$, thus $\mu=\nu$.
\end{proof}

\smallskip

\subsubsection{Compatibilities}
In this section we compare the various module categories introduced above.
First, observe that for $\sharp=t,n$ or $b$, the algebra $\frakU_\zeta^\sharp$ 
is a central extension of the small quantum group $\u_\zeta$.
By base change we get the specialization functor 
$$\R_\sharp^\u: \frakU_\zeta^\sharp\mod\to \Rep(\u_\zeta).$$
%M\mapsto \R_\hyb^n(M)\otimes_{\bfk[N^-]}\bfk.$$
Next, we have the following.

\begin{lemma}\label{lem:res}
\hfill
\begin{itemize}[leftmargin=8mm]
\item[$\mathrm{(a)}$] 
There is a right exact functor
$\R_\hyb^t: \calO_{\zeta,S}^{\hyb}\to \frakU_\zeta^t\mod^\LLbd_S$. 
It maps  standards to standards.
\item[$\mathrm{(b)}$] 
There is a right exact functor
$\R_\hyb^\u: \calO^\hyb_\zeta\to \Rep(\u_\zeta)$. 
It maps  standards to standards.
\end{itemize}
\end{lemma}

\begin{proof}
The algebra $\frakU_\zeta^b$ embeds into $\U_\zeta^\hyb$.
We claim that each object $M\in \calO_{\zeta,S}^{\hyb}$ is finitely generated over 
$(\frakU_\zeta^b,S)$.
Hence, the restriction  yields an exact functor
\begin{align}\label{resb}\R_\hyb^b:\calO_{\zeta,S}^{\hyb}\to \frakU_\zeta^b\mod^\LLbd_S.\end{align}
Composing it with base change 
$$\R_b^t: \frakU_\zeta^b\mod^\LLbd_S\to \frakU_\zeta^t\mod^\LLbd_S,\quad
M\mapsto M\otimes_{S^b}S,$$
we get the functor
$$\R_\hyb^t: \calO_{\zeta,S}^{\hyb}\to \frakU_\zeta^t\mod^\LLbd_S,\quad
M\mapsto \R_\hyb^b(M)\otimes_{S^b}S.$$
The functor $\R_\hyb^t$ maps standards to standards.

To prove the claim, observe that, since $M$ is a finitely generated $\U_\zeta^\hyb$-module over which 
$\U_\zeta^+$ acts locally unipotently, 
it admits a finite filtration
$0=M_0\subset M_1\subset...\subset M_n=M$ such that $M_i/M_{i-1}$ is a quotient of a Verma module 
$\M(\mu_i)_S$.
Moreover $\R_\hyb^b(\M(\mu_i)_S)=\Delta(\mu_i)^b$ and the restriction functor is exact. 
Thus $\R_\hyb^b(M)$ also has such a filtration. 
Since standard modules are finitely generated, so is their quotient and a finite extension of them.

Part (a) is proved. Part (b) is proved in a similar way. By \eqref{1/2}, the restriction yields a functor
$$\R_\hyb^n:\calO^\hyb_\zeta\to \frakU_\zeta^n\mod^\LLbd.$$
Composing it with the specialization functor $\R_n^\u$ we get the functor 
$$\R_\hyb^\u: \calO^\hyb_\zeta\to \Rep(\u_\zeta),\quad
M\mapsto \R_\hyb^n(M)\otimes_{\bfk[N^-]}\bfk.$$

\end{proof}

\smallskip

The pullback by the obvious map $\U_\zeta^{\hyb,b}\to \U_\zeta$ yields an exact functor 
$\R_\U^\hyb:\Rep(\U_\zeta)\to \calO^\hyb_\zeta.$
Similarly, there is a restriction functor 
$\R_\U^\hyb:\Rep(\U_{\hat\zeta})\to \calO^{\hyb}_{q,\bfk[[\hbar]]}$.

\begin{lemma}\label{lem:QF}
\hfill
\begin{itemize}[leftmargin=8mm]
\item[$\mathrm{(a)}$] 
The functor $\R_\U^\hyb:\Rep(\U_\zeta)\to \calO^\hyb_\zeta$ has a left adjoint $\I^\U_\hyb$. 
The counit $\varepsilon:\I^\U_\hyb\circ\R_\U^\hyb\to 1$ is invertible.
\item[$\mathrm{(b)}$] 
$\I^\U_\hyb(\M(\mu))=\V(\mu)$ if $\mu\in\Lbd^+$ and 0 otherwise.
\item[$\mathrm{(c)}$] 
The functor
$\R_\U^\hyb:\Rep(\U_{\hat\zeta})\to \calO^{\hyb}_{q,\bfk[[\hbar]]}$
has a left adjoint $\I^\U_\hyb$.
The counit $\varepsilon:\I^\U_\hyb\circ\R_\U^\hyb\to 1$ is invertible.
\end{itemize}
\end{lemma}

\begin{proof}   
The ind-completions $\Ind\Rep(\U_\zeta),$ $\Ind\calO^\hyb_\zeta$ of the categories
$\Rep(\U_\zeta),$ $\calO^\hyb_\zeta$ are locally finitely presentable.
More precisely, the category $\Ind\Rep(\U_\zeta)$ consists of all integrable $\U_\zeta$-modules,
while $\Ind\calO^\hyb_{\zeta,R}$ consists of all $\LLbd$-graded $(\U^\hyb_\zeta,R)$-bimodules 
with a locally unipotent $\U_\zeta^+$-action. 
The functor $\R_\U^\hyb$ extends to a restriction functor 
$\Ind\Rep(\U_\zeta)\to\Ind\calO^\hyb_\zeta$ which preserves filtered colimits and 
small limits. Thus, it admits a left adjoint $\I^\U_\hyb$ by the adjoint functor theorem. 
The composed functor
$$\xymatrix{
\Ind\Rep(\U_\zeta)\ar[r]^-{\R_\U^\hyb} &\Ind\calO^\hyb_\zeta \ar[r]^-{\R_\hyb^\geqslant} 
&\Ind\Rep(\U_\zeta^\geqslant) }
$$
is the obvious restriction. Its left adjoint is the induction functor considered in \cite[\S 1]{APW2},
which takes the integrable one-dimensional $\U_\zeta^\geqslant$-module $\bfk_\mu$ to
the Weyl module $\V(\mu)$ in $\Rep(\U_\zeta)$ if $\mu\in\LLbd^+$, and 0 otherwise.
It is the composed functor $\I^\U_\hyb\circ\I^\hyb_{\geqslant}$.
We have $\I^\hyb_{\geqslant}(\bfk_\mu)=\M(\mu)$ for each $\mu\in\LLbd$.
Hence $\I^\U_\hyb(\M(\mu))=\V(\mu)$ if $\mu$ is dominant and 0 otherwise, proving (b).

We claim that the functor $\I^\U_\hyb$ maps 
$\calO^\hyb_\zeta$ into $\Rep(\U_\zeta)$, proving the first statement in part (a). 
The claim is a consequence of part (b) and of the right exactness of the functor $\I^\U_\hyb$.
Next, we prove the second part of (a).
It is enough to check that the functor $\R_\U^\hyb$ is fully faithful.
To do that we consider the chain of obvious maps
$$\xymatrix{\U^{\hyb,b}_\zeta\ar@{->>}[r]&\U_\zeta^{\frac{1}{2}}\ar@{^{(}->}[r]&\U_\zeta}$$
which decomposes the functor $\R_\U^\hyb$ as the composition of the following restrictions
$$\xymatrix{\Rep(\U_\zeta)\ar[r]^-F& \Rep(\U_\zeta^{\frac{1}{2}})\ar[r]^-H&\calO^\hyb_\zeta}.$$
The functor $H$ is obviously fully faithful.
The functor $F$ is also fully faithful, proving the lemma.
Indeed, the monoidal category $\Rep(\LG)$ can be viewed as an algebra object in the category of all categories.
Hence, the restriction $\Rep(\LG)\to\Rep(\LB)$ and the quantum Frobenius $\Fr^*:\Rep(\LG)\to\Rep(\U_\zeta)$
equip $\Rep(\LB)$ and $\Rep(\U_\zeta)$ with $\Rep(\LG)$-module structures in the category of all categories, and we can 
consider the tensor product of categories
$$\Rep(\LB)\otimes_{\Rep(\LG)}\Rep(\U_\zeta).$$
By \eqref{AG} this tensor product is equivalent to the category $\Rep(\U_\zeta^{\frac{1}{2}})$.
Since the forgetful functor $\Rep(\LG)\to\Rep(\LB)$ is fully faithful, we deduce that
$F$ is also fully faithful.
\end{proof}

\begin{remark}\label{rk:defQF}
Similarly, in the deformed setting, the pullback by $\U_{\hat\zeta}^{\hyb,b}\to \U_{\hat\zeta}$ yields an exact functor 
$\R_\U^\hyb:\Rep(\U_{\hat\zeta})\to \calO^\hyb_{\hat\zeta}.$
The same proof shows that Lemma \ref{lem:QF} holds in this setting as well.
\end{remark}

Let $\R_\U^\u$ be the restriction by the obvious embedding $\u_\zeta\subset\U_\zeta$.
We have the following diagram
\begin{equation}\label{diagram}
\begin{split}
\xymatrix{
\calO^{\hyb}_{\zeta,S}\ar[d]_{\R_\hyb^t}\ar[rr]^{-\otimes_S\bfk} &&\calO^\hyb_\zeta \ar[d]_{\R_\hyb^\u}
\ar@{}[d]\ar[rr]^{\I^\U_\hyb} 
&&\Rep(\U_\zeta)
 \ar[dll]^{\R_\U^\u} \\
\frakU_\zeta^t\mod^\Lbd_S\ar[rr]^{-\otimes_{S}\bfk}&& \Rep(\u_\zeta).&&
}
\end{split}
\end{equation}

\smallskip
\subsubsection{The Steinberg module}
For all categories of $A$-modules above, we define the Steinberg module to be 
$$\St^A=\Delta^A(\ell\Lrho-\Lrho).$$
We'll abbreviate $\St^\sharp_R=\St^{\frakU^\sharp_\zeta}$ in $\frakU^\sharp_\zeta\mod^\LLbd_R$.
The functors in \eqref{diagram} map Steinberg modules to Steinberg modules.

\begin{lemma}
The Steinberg module is projective in the categories
$\Rep(\U_\zeta)$, $\frakU^\sharp_\zeta\mod^\LLbd_R$ and $\Rep(\u_\zeta)$.
\end{lemma}

\begin{proof}
For $A=\U_\zeta$, $\Ut$, $\u_\zeta$ this is a consequence of the strong linkage principle,
see \cite{APW1}, \cite{AJS}.
For $A=\frakU_\zeta^b$, we have
$\St_R^b=\frakU_\zeta^b\otimes_{\Ut} \St_R^t.$
The induction functor $\frakU_\zeta^b\otimes_{\Ut}-$ is left adjoint to the restriction functor
$\frakU_\zeta^b\mod^\LLbd_R\to\frakU_\zeta^t\mod^\LLbd_R$, which is exact.
Thus it maps projective modules to projective ones. We deduce that $\St_R^b$ is projective.
\end{proof}

\smallskip
\subsubsection{Projective modules}\label{sec:proj}
We consider each of the module categories above separately.
The category $\Rep(\U_\zeta)$ has enough projectives, see \cite[\S 9]{APW2}.
Let $(R,K,F)$ be as above.
The module $\L(\mu)_F^\sharp$ has a unique projective cover $\P(\mu)_R^\sharp$ in 
$\frakU^\sharp_\zeta\mod^\LLbd_R$ for each $\mu\in\LLbd$, and we have
$\P(\mu)_R^\sharp\otimes_RF\simeq \P(\mu)_F^\sharp$. The modules
$\P(\mu)_R^\sharp$ with $\mu\in \LLbd$ form a non-redundant set of representatives of indecomposable
projective objects in $\frakU^\sharp_\zeta\mod^\LLbd_R$.
The category $\calO^{\hyb}_{\zeta,R}$ does not have enough projectives.
Projective objects only exist in truncated categories.
For each weight $\nu\in\LLbd$, let $\calO^{\hyb,\leqslant\nu}_{\zeta,R}$ be the full 
subcategory consisting 
of objects $M$ such that 
$M_\mu=0$ if $\mu\not\leqslant \nu$.
Any object in  $\calO^{\hyb}_{\zeta,R}$ belongs to a subcategory $\calO^{\hyb,\leqslant\nu}_{\zeta,R}$ for some 
weight $\nu$. For each $\mu\in\LLbd$, we set 
$$\Q_R(\mu)=\U_\zeta^\hyb\otimes_{\frakU_\zeta^b} P_R^b(\mu).$$

\begin{lemma}\label{lem:proj} 
\hfill
\begin{itemize}[leftmargin=8mm]
 \item[$\mathrm{(a)}$] 
The functor 
$\Hom_{(\U_\zeta^\hyb,R)}(\Q(\mu)_R\,,\,-):\calO^{\hyb}_{\zeta,R}\to\mathrm{Vect}_\bfk$
is exact.% and takes $\E(\lambda)_F$ to $\delta_{\lambda,\mu}F$.
\item[$\mathrm{(b)}$] 
For each $\nu\geqslant \mu$, there is a quotient $\Q(\mu)_R^{\leqslant\nu}$ of $\Q(\mu)_R$, which is projective in $\calO^{\hyb,\leqslant\nu}_{\zeta,R}$. It maps onto
 the simple module $\E(\mu)_F$, and admits a finite filtration by Verma modules.
 \end{itemize}
\end{lemma}

\begin{proof}
For any $N\in \calO^{\hyb}_{\zeta,R}$,
by \eqref{resb}, the module $\R_\hyb^b(N)$ belongs to the category $\frakU_\zeta^b\mod^\LLbd_R$, 
in which $P_R^b(\mu)$ is projective. The canonical adjunction between induction and restriction yields a natural 
isomorphism
$$\Hom_{ (\U_\zeta^{\hyb},R)}(\Q_R(\mu), N)=\Hom_{(\frakU_\zeta^b,R)}(P_R^b(\mu), \R_\hyb^b(N)),
\qquad \forall N\in \calO^{\hyb}_{\zeta,R}.$$
Thus the functor $\Hom_{(\U_\zeta^{\hyb},R)}(\Q_R(\mu),-)$ is exact on $\calO^{\hyb}_{\zeta,R}$. This proves part (a).

Next, by the PBW theorem, the functor $\U_\zeta^\hyb\otimes_{\frakU_\zeta^b}-$ is exact. The module $P_R^b(\mu)$ admits a finite filtration by the standard modules of the form $\Delta_R^b(\lambda)$. So $\Q_R(\mu)$ admits a finite filtration by modules of the form $\U_\zeta^\hyb\otimes_{\frakU_\zeta^b}\Delta_R^b(\lambda)$.
Further, the quotient of $\U_\zeta^{\hyb,\geqslant}$ by the two-sided ideal generated by the augmentation ideal of the subalgebra $\frakU^{b,\geqslant}_\zeta$ is isomorphic to $\U^+_1$.
So we have an $R$-linear isomorphism
\begin{align}\label{ind0}\U^+_1\otimes R_{\lambda}=
\U_\zeta^{\hyb,\geqslant}\otimes_{\frakU_\zeta^{b,\geqslant}}R_{\lambda}.
\end{align}
%where $(\frakU^{\frakb,\times}_\zeta)$ is the two sided ideal generated by the augmentation ideal of
%$\frakU_\zeta^b$. 
%Since the Hopf algebra $\frakU_\zeta^b$ is normal in $\U_\zeta^\hyb$, the multiplication yields an isomorphism
%$\U_1^+\otimes\frakU_\zeta^{b,\geqslant}=\U_\zeta^{\hyb,\geqslant}.$
We deduce that
\begin{eqnarray*}
\U_\zeta^\hyb\otimes_{\frakU_\zeta^b}\Delta_R^b(\lambda)=\U_\zeta^\hyb\otimes_{\U^{\hyb,\geqslant}_{\zeta}}(\U^+_1\otimes R_{\lambda}),
\end{eqnarray*}
where the $\U^{\hyb,\geqslant}_{\zeta}$-action in the right hand side is given by \eqref{ind0}.
Since $\U^+_1$ is unipotent, 
$\U_\zeta^\hyb\otimes_{\frakU_\zeta^b}\Delta_R^b(\lambda)$ admits a filtration by Verma modules of highest weights
$\lambda+\ell\beta$ with $\beta\in Q^+$, with $\M(\lambda)_R$ appearing only once as a quotient.
So $\Q_R(\mu)$ has a filtration by Verma modules.
Moreover, $\Delta_R^b(\mu)$ is a quotient of  $P_R^b(\mu)$, so $\Q_R(\mu)$ surjects onto $\M(\mu)_R$, which surjects onto $\E(\mu)_F$. Finally, since $\Ext^1(\M(\lambda)_R, \M(\lambda')_R)=0$ unless $\lambda' > \lambda$, the module $\Q_R(\mu)$ admits a unique quotient $\Q(\mu)_R^{\leqslant\nu}$ in $\calO^{\hyb,\leqslant\nu}_{\zeta,R}$, which is filtered by Verma modules appearing in the filtration of $\Q_R(\mu)$ with highest weight $\leqslant \nu$.
We have, for any $N\in \calO^{\hyb,\leqslant\nu}_{\zeta,R}$, 
$$\Hom_{ (\U_\zeta^{\hyb},R)}(\Q_R(\mu), N)=\Hom_{ (\U_\zeta^{\hyb},R)}(\Q(\mu)_R^{\leqslant\nu}, N).$$
So $\Q(\mu)_R^{\leqslant\nu}$ is projective in $\calO^{\hyb,\leqslant\nu}_{\zeta,R}$.
\end{proof}

\smallskip

Note that the $\U_\zeta^\hyb$-module $\Q(\mu)_R$ does not belong to the category $\calO^{\hyb}_{\zeta,R}$.
\smallskip

\iffalse %%%%%%%%%%%%%%
\begin{corollary}\label{cor:proj1}
The module $\Q(\mu)_R$ has a Verma filtration, in particular it is free over $R$.
%\item[$\mathrm{(c)}$] $\Q(\mu+\ell\beta)_R\subset \Q(\mu)_R$for any $\beta\in Q^+$.
\end{corollary}

\begin{proof}
Follows from the proof of Lemma \ref{lem:proj} and the following \cite[\S 5.3.2]{G18}
$$\Ext_{\calO_{\zeta,F}^\hyb}^{>0}(\M(\lambda)_F,\M(\mu)_F^\vee)=0,\qquad
\Hom_{\calO_{\zeta,F}^\hyb}(\M(\lambda)_F,\M(\mu)_F^\vee)=\delta_{\lambda,\mu} F$$
where $\M(\mu)_F^\vee$ is the dual Verma module in $\calO^\hyb_{\zeta,F}$.
\end{proof}
\fi %%%%%%%%%%%%%%%%%%
%\smallskip
%\begin{remark}
%For $\U_\zeta$ and $\u_\zeta$, the block of the Steinberg module is semi-simple. This is not the case for 
%$\calO^\hyb_\zeta$. Indeed, let $v_\mu\in \M(\mu)$ be a highest weight vector.
%Then, there is a homomorphism $\M(\mu)\to \M(\mu+2\ell\Lrho)$ such that
%$v_\mu\mapsto (\prod_{\alpha\in\Phi^+}F_\alpha^\ell) \,v_{\mu+2\ell\Lrho}$.\end{remark}

\begin{corollary}\label{cor:proj2}
\hfill
\begin{itemize}[leftmargin=8mm]
\item[$\mathrm{(a)}$]
The unit $\eta:1\to\R_\U^\hyb\circ\I^\U_\hyb$ is surjective.
\item[$\mathrm{(b)}$]
The category $\calO^\hyb_\zeta$ is a module category over the tensor category $\Rep(\U_\zeta)$.
The functors $R^\U_\hyb$, $I^\U_\hyb$ commute with tensor product with objects in $\Rep(\U_\zeta)$.
\end{itemize}
\end{corollary}

\begin{proof}
To prove part (a), it is enough to show that $\eta_N: N\to \R_\U^\hyb\circ\I^\U_\hyb(N)$ is surjective for any $N\in \calO^{\hyb,\leqslant\nu}_{\zeta,R}$. Note that by Lemma \ref{lem:QF}, the surjectivity holds when $N$ is a Verma module. Since $\Q(\mu)^{\leqslant\nu}$ has a finite filtration by Verma modules, 
by the Five lemma, the morphism $\eta_{\Q(\mu)}$ is also surjective.
Finally, any object in $N\in \calO^{\hyb,\leqslant\nu}_{\zeta,R}$ is a quotient of a finite direct sum of 
$\Q(\mu)^{\leqslant \nu}$'s, thus $\eta_N$ is always surjective.

Now, we prove the claim (b). The algebra $\scrU_q$ is a Hopf algebra, and $\frakU_q$, $\U^{\hyb,b}_q$, $\U_q$ are Hopf subalgebras. So it makes sense to tensor objects in $\Rep(\U_\zeta)$ and those in  $\calO^\hyb_\zeta$, we obtain again objects in $\calO^\hyb_\zeta$.
By definition, the restriction functor $R^\U_\hyb$ commutes with tensor product. Now, let us consider the functor $\I^\U_\hyb$.

%We must define a functorial isomorphism
%\begin{align*}\I^\U_\hyb(M\otimes \Fr^*V)\to \I^\U_\hyb(M)\otimes \Fr^*V,\qquad M\in\calO^\hyb_\zeta,
%\qquad V\in\Rep(\LG).
%\end{align*}
First, observe that for any $M\in\calO^\hyb_\zeta$, $L\in \Rep(\U_\zeta)$ and any morphism 
$\varphi:M\to\R_\U^\hyb(L)$, it must
factorize through the map $\eta(M):M\to\R_\U^\hyb\I^\U_\hyb(M)$ and yields a morphism $\R_\U^\hyb\I^\U_\hyb(M)\to \R_\U^\hyb(L)$.
Because, by definition of the adjunction $(\R_\U^\hyb,\I^\U_\hyb)$,
the map $\varphi$ decomposes as
$$\xymatrix{M\ar[r]^-{\eta(M)}&\R_\U^\hyb\I^\U_\hyb(M)\ar[rr]^-{\R_\U^\hyb\I^\U_\hyb(\varphi)}&&
\R_\U^\hyb\I^\U_\hyb\R_\U^\hyb(L)=\R_\U^\hyb(L).}$$
Now, for any $N\in \Rep(\U_\zeta)$, apply this to $L= \I^\U_\hyb(M)\otimes N$ and the surjective map
$$\eta(M)\otimes\id_N:M\otimes N\to\R_\U^\hyb\I^\U_\hyb(M)\otimes N=\R_\U^\hyb(\I^\U_\hyb(M)\otimes N),$$
we obtain a surjective map
$$\psi:\R_\U^\hyb\I^\U_\hyb(M\otimes N)\to\R_\U^\hyb(\I^\U_\hyb(M)\otimes N).$$
Apply $\I^\U_\hyb$ and Lemma \ref{lem:QF}. We get a surjective map
$$\I^\U_\hyb(\psi):\I^\U_\hyb(M\otimes N)\to\I^\U_\hyb(M)\otimes N.$$
Let $K$ be it kernel, and $I_K$ be the injective hull of $K$ in $\Rep(\U_\zeta)$. We must show that $K=0$.
To do this, we apply the functor $\Hom_{\U_\zeta}(-,I_K)$ and obtain the short exact sequence
$$0\to\Hom_{\U_\zeta}(\I^\U_\hyb(M)\otimes N,I_K)\to\Hom_{\U_\zeta}(\I^\U_\hyb(M\otimes N),I_K)\to\Hom_{\U_\zeta}(K,I_K)\to 0.$$
By adjunction, the second map is identified with the map
$$\Hom_{\calO^\hyb_\zeta}(M,\R_\U^\hyb(I_K\otimes N^*))\to
\Hom_{\calO^\hyb_\zeta}(M,\R_\U^\hyb(I_K)\otimes N^*)$$
which is an isomorphism, because $\R_\U^\hyb$ commutes with $-\otimes N^\ast$. We deduce that the second map in the short exact sequence is an isomorphism. Thus $\Hom_{\U_\zeta}(K,I_K)=0$. Hence $K=0$. So $\I^\U_\hyb(\psi)$ is an isomorphism.
\end{proof}

\smallskip

\subsubsection{The $\pi_1$-gradings}
Let $\Rep(\U_\zeta)'$ be the full subcategory of $\Rep(\U_\zeta)$ consisting of the modules $M$ such that
the weight subspace $M_\lambda$ vanishes for all $\lambda\in\LLbd\setminus\LQ$. 
Since $\pi_1=\X^*(\LZ)$, the category $\Rep(\U_\zeta)$ is equivalent to the category of
$\pi_1$-graded objects in $\Rep(\U_\zeta)'$. 
Thus, there is a decomposition
$Z(\Rep(\U_\zeta))=Z(\Rep(\U_\zeta)')^{\oplus\pi_1}$.
Hence, the obvious map
$Z(\U_\zeta)\to Z(\Rep(\U_\zeta)')$ yields a canonical map
$$Z(\U_\zeta)^{\oplus\pi_1}\to Z(\Rep(\U_\zeta)).$$
Composing this map with the inverse Harish-Chandra homomorphism \eqref{IHC2}
we get an algebra homomorphism
\begin{align}\label{HC1}
\bfk[\LT/W]^{\oplus\pi_1}\to Z(\Rep(\U_\zeta)).
\end{align}
In a similar way we define algebra homomorphisms
\begin{align}\label{HC2}
\begin{split}
&\bfA[\LT/W]^{\oplus\pi_1}\to Z(\Rep(\U_q)),\\
&\bfA[\LT/W]^{\oplus\pi_1}\to Z(\calO^{\hyb}_{q,S[[\hbar]]}),\\
&\bfk[\LT/W]^{\oplus\pi_1}\to Z(\calO^{\hyb}_{\zeta,\bfk}).
\end{split}
\end{align}

\smallskip

\subsubsection{Central characters and blocks}
We consider the central characters
$$\chi_{\zeta,\mu}^\sharp:Z(\frakU_\zeta^\sharp\mod^\LLbd_S)\to S,\qquad 
\chi_{\zeta,\mu}^\hyb:Z(\calO^{\hyb}_{\zeta,S})\to S,\qquad
\chi_{\zeta,\lambda}:Z_\HC\to \bfk,$$
of the standard modules $\Delta(\mu)^\sharp_S$, $\M(\mu)_S$ 
and $\V(\lambda)_\zeta$, respectively, for $\mu\in\LLbd$, $\lambda\in\LLbd^+$.
Under the isomorphism $Z_\HC=\bfk[\LT/W]$
the central character $\chi_{\zeta,\lambda}$ is identified with the point
$W(\zeta^{2(\lambda+\Lrho)})$ in $\LT(\bfk)/W$.
The products of the central characters of all standard modules yield the maps
\begin{align*}
&\chi_\zeta^\sharp:Z(\frakU_\zeta^\sharp\mod^\LLbd_S)\to\Fun(\LLbd,S),\quad
\chi_\zeta^\hyb:Z(\calO^{\hyb}_{\zeta,S})\to\Fun(\LLbd,S),\\
&\chi_\zeta:Z_\HC\to\Fun(\LLbd^+,\bfk).
\end{align*}
Similarly, if $q_e=\zeta_e+\hbar$ with $\hbar$ closed to 0,
the central characters of the standard modules $\M(\mu)_{S[[\hbar]]}$ and $\V(\lambda)_{\hat\zeta}$
yield the maps 
\begin{align*}
\chi^\hyb_{\hat\zeta,\mu}:Z(\calO^{\hyb}_{q,S[[\hbar]]})\to S[[\hbar]],\quad
\chi_{\hat\zeta,\lambda}:Z_\HC\to\bfk[[\hbar]].
\end{align*}
Taking the product over all $\mu$ and $\lambda$ we get the maps
\begin{align}\label{ccqmix}
\chi^\hyb_{\hat\zeta}:Z(\calO^{\hyb}_{q,S[[\hbar]]})\to\Fun(\LLbd,S[[\hbar]]),\quad
\chi_{\hat\zeta}:Z_\HC\to\Fun(\LLbd^+,\bfk[[\hbar]]).
\end{align}
The maps in \eqref{ccqmix} are injective,  because the categories
$\calO^{\hyb}_{q,S[[\hbar]]}$
and $\Rep(\U_q)$ are generically  semi-simple.

Let the algebra $A$ be either $\U_\zeta$, $\u_\zeta$, $\U^{\hyb,\flat}_\zeta$ or $\frakU_\zeta^\sharp$.
Hence $A$ receives a map from $\frakU_\zeta$ which takes
the Harish-Chandra center $Z_\HC$ to the center of $A$.
The algebra $Z_\HC$ acts on the standard module $\Delta^A(\mu)$ by the scalar multiplication 
with the character $\chi_{\zeta,\mu}$.
Under the isomorphism $Z_\HC=\bfk[\LT/W]$
the central character $\chi_{\zeta,\mu}$ is identified with the point
$W(\zeta^{2(\mu+\Lrho)})$ in $\LT(\bfk)/W$.
Hence, if the weight $\mu$, $\mu'$ do not belong to the same $W_{\ell,\ex}\bullet$-orbit, 
there is no extension between any homomorphic images of the standard modules
of $\Delta^A(\mu)$ and $\Delta^A(\mu')$.

\bigskip

\section{Center and cohomology}

In this section we relate the center of the quantum groups to the cohomology of affine springer fibers.
To do so, we first fix a $W$-equivariant $\bfk$-linear isomorphism $\frakt=\frakt^*$.
This yields the isomorphisms
\begin{align}\label{pairing1} \begin{split}
(H^\bullet_T)_{\hat 0}&=\bfk[\frakt]_{\hat 0}=\bfk[\frakt^*]_{\hat 0}=\bfk[\LT]_{\hat 1}=S,\\
(H^\bullet_{T\times\bbG_m})_{\widehat{0,0}}&=\bfk[\frakt\times\bbG_a]_{\widehat{0,0}}
=\bfk[\frakt^*\times\bbG_a]_{\widehat{0,0}}=\bfk[\LT\times\bbG_m]_{\widehat{1,\zeta}}=S[[\hbar]],\\
\end{split}
\end{align}
Recall that
\begin{align*}
&\Om=\{1\}\times_{T/W}T/W,\qquad\LOm=\{1\}\times_{\LT/W}\LT/W,\\
&\Gamma=T\times_{T/W}T/W,\qquad\LGamma=\LT\times_{\LT/W}\LT/W.
\end{align*}
This yields also scheme isomorphisms 
\begin{align}\label{pairing2}\Om=\LOm,\quad
\widetilde{N}_{\Omega}(T/W)=
\widetilde{N}_{\LOm}(\LT/W),\quad
\widetilde{N}_{\Gamma}(T\times T/W)_{\widehat{1,0}}=
\widetilde{N}_{\LGamma}(\LT\times\LT/W)_{\widehat{1,0}}.
\end{align}
Here, the subscript $\widehat{1,0}$ denotes the formal neighbourhoods
at the points $(1,0)$ in $T\times\bbG_a$ and $\LT\times\bbG_a$ respectively.

\smallskip

\subsection{The center of $\frakU_\zeta^t$}
We first concentrate on the center of $\frakU_\zeta^t$. We'll prove the following theorem.

\begin{theorem} \label{prop:P2}\hfill
\begin{itemize}[leftmargin=8mm]
\item[$\mathrm{(a)}$] 
There is a commutative diagram of algebra homomorphisms
\begin{align*}
\xymatrix{H^\bullet_T({}^0\Gr^\zeta_{\gamma_\ell})_{\hat 0}\ar[rr]^-\bfa_-\sim\ar@{_{(}->}[dr]_-\res
&&Z(\Ut\mod^\LLbd_S)\ar@{^{(}->}[dl]^-{\chi_\zeta^t}\\
&\Fun(\LLbd,S) &
}
\end{align*}

\item[$\mathrm{(b)}$] 
The map $\bfa$
specializes to an algebra embedding
$$\bfa:H^\bullet({}^0\Gr^\zeta_{\gamma_\ell})\hookrightarrow Z(\u_\zeta\mod^\LLbd)$$
which restricts to an embedding $H^\bullet({}^0\Gr^\zeta_{\gamma_\ell})^{\ell\LLbd}\hookrightarrow Z(\u_\zeta)^\LT.$
\end{itemize}
\end{theorem}

\smallskip

We propose the following conjecture:

\begin{conjecture}\label{conj-G1T}
The algebra embedding 
$$\bfa:H^\bullet({}^0\Gr^\zeta_{\gamma_\ell})\hookrightarrow Z(\u_\zeta\mod^\LLbd)$$
is an isomorphism. As a consequence, it restricts to an isomorphism
$$H^\bullet({}^0\Gr^\zeta_{\gamma_\ell})^{\ell\LLbd}\simeq Z(\u_\zeta)^\LT.$$
\end{conjecture}

Here the fact that the second isomorphism is a consequence of the first one is given by \eqref{uu} below.

\smallskip

\begin{remark}
Since the map $\bfa$ in (a)
is a ring isomorphism, it yields a bijection between the blocks of $\Ut\mod^\LLbd_S$ and the 
connected components of ${}^0\Gr^\zeta_{\gamma_\ell}$.
Let us briefly recall the blocks decomposition for the simply connected small quantum group $\u_\zeta$.
Note that \cite{LQ2} considered the adjoint small quantum group.
For each restricted dominant weight $\mu\in\LLbd^+_\ell$ 
let $\Rep(\u_\zeta)^\mu$ be the full subcategory of $\Rep(\u_\zeta)$ 
generated by the simple modules $\L(\mu')$ with $\mu'\in W_{\ell,\ex}\bullet\mu\cap\LLbd^+_\ell$.
It is an indecomposable category which consists of all modules
killed by some power of the character $\chi_{\zeta,\mu}$.
Let $\u_\zeta^\omega\subset\u_\zeta$ be the two-sided ideal such that
$\Rep(\u_\zeta^\omega)=\Rep(\u_\zeta)^\omega.$
We have 
$$\u_\zeta=\bigoplus_{\omega\in\Xi}\u_\zeta^\omega,\qquad
Z(\u_\zeta)=Z(\Rep(\u_\zeta)),\qquad 
Z(\u_\zeta^\omega)=Z(\Rep(\u_\zeta)^\omega).$$
The map 
$\bfa:H^\bullet({}^0\Gr^\zeta_{\gamma_\ell})^{\ell\LLbd}\to Z(\u_\zeta)^\LT$ in (b)
is compatible with the block decomposition, meaning that, under the decomposition
${}^0\Gr^\zeta_{\gamma_\ell} =\bigsqcup_{\omega\in \Xi} {}^0\!\Fl^\omega_{\ell,\gamma_\ell}$
in Lemma \ref{lem:isom2}, it takes the summand
$H^\bullet({}^0\!\Fl^\omega_{\ell,\gamma_\ell})^{\ell\LLbd}$ into $Z(\u_\zeta^\omega)^\LT$.
Note that the ring $H^\bullet({}^0\!\Fl^\omega_{\ell,\gamma_\ell})^{\ell\LLbd}$ is indecomposable,
although $H^\bullet({}^0\!\Fl^\omega_{\ell,\gamma_\ell})$ is not since 
${}^0\!\Fl^\omega_{\ell,\gamma_\ell}$ is not connected.

\end{remark}

\smallskip

\subsubsection{The center of $\Ut$}
From now on, let $\pi$, $S$, $R$, $F$, $K$ be as in \S\ref{sec:rep}.
We'll abbreviate
$$\calC_R=\Ut\mod^\LLbd_R.$$
We first gather a few facts on the category $\calC_R$, most of them being proved in \cite{AJS}. 
We'll use the notation from the appendix. In particular
$\calP_R$ is the full subcategory of projective objects in $\calC_R$,
we have the base change $\calP_S\to\calP_R$, $M\mapsto M\otimes_SR$, 
and
$$\Hom_{\calP_S}(M,N)\otimes_SR=\Hom_{\calP_R}(M\otimes_SR,N\otimes_SR),
\qquad M,N\in \calP_S.$$ 
Further $\Hom_{\calC_S}(M,N)$ is projective over $S$ whenever $M\in\calP_S$ and $N$ is projective over $S$. 
Hence, Lemmas \ref{lem:a}, \ref{lem:b} hold and we have an inclusion of the centers
$Z(\calC_R)\subset  Z(\calC_K).$
For each root $\alpha$, we set 
$h_\alpha=(K_\Lalpha)^{2\ell}-1$ in $S.$
We abbreviate $h_\alpha=\pi(h_\alpha)$ in $R$.
If $R=F$ is a field, we define the set
\begin{align}\label{PhiF}\LPhi_F=\{\Lalpha\in\LPhi\,;\, h_\alpha= 0\}.\end{align}
It is a root subsystem of $\LPhi$. 
If $R$ is a local ring we set
$\LPhi_R=\LPhi_F$ and $\LPhi_R^+=\LPhi^+\cap\LPhi_R$.
Given $\mu\in\LLbd$ and $\Lalpha\in\LPhi_R$, there is a unique integer 
$n_\alpha(\mu)$ such that, see \cite[\S 6.1]{AJS},
$$0<n_\alpha(\mu)\leqslant\ell
\quad,\quad
\pi(K_{\Lalpha})^2=\zeta^{2d_\Lalpha(n_\alpha(\mu)-\langle\alpha,\mu+\Lrho\rangle)}.$$
Let $W_{\ell,\af}^R\subset W_\af$ be the subgroup 
generated by the affine reflections $s_{\alpha+m\ell\delta}$ with $\Lalpha\in\LPhi_R$ and $m\in\bbZ$. 
For each weight $\mu\in\LLbd$, let $\calC_R^\mu$ 
be the Serre subcategory of $\calC_R$ generated by the baby Verma modules $\Delta(\lambda)^t_R$ 
with $\lambda\in W_{\ell,\af}^R\bullet\mu$, and 
$\calP_R^\mu=\calP_R\cap \calC_R^\mu$.
The categories $\calC^\mu_R$ and $\calP_R^\mu$ 
only depend on the class of the weight $\mu$ in $\LLbd/W^R_{\ell,\af}$.
They may not be indecomposable.
We have
\begin{align}\label{split}
\calC_R=\bigoplus_{\omega\in\LLbd/W^R_{\ell,\af}}\calC_R^\omega.
\end{align}
%We have $n_\alpha(\mu)=\ell$ if and only if there is an integer 
%$m\in\bbZ$ such that $s_{\alpha,m\ell}\in W_{\ell,\mu}$.
Set $\alpha\downarrow \mu=\mu-n_\alpha(\mu)\Lalpha$ if $n_\alpha(\mu)\neq\ell$, and
$\alpha\downarrow \mu=\mu$ otherwise. 
Write $\nu\preccurlyeq\mu$ if there is a chain $\nu=\mu_0,\mu_1,....,\mu_n=\mu$ and $\beta_i\in\Phi^+_R$ 
such that 
$\beta_i\downarrow\mu_i=\mu_{i-1}$.
The strong linkage principle implies that 
$$[\Delta(\mu)^t_F:\L(\lambda)^t_F]\neq 0
\Rightarrow \lambda\preccurlyeq\mu.$$
In particular, if  $\LPhi_F$ is empty then the category $\calC_F$ is semi-simple.
For each $\alpha\in\Phi^+$ we consider the $S$-algebra
$$S_\alpha=S[h_\beta^{-1}\,;\,\beta\in\Phi^+\,,\,\alpha\neq\beta].
$$
Assume that $R=S_\alpha$. 
The maximal ideal of $R$ is generated by $h_\alpha$. We have $\LPhi_R^+=\{\Lalpha\}$.
Set
\begin{align*}\mu_0=\mu,\qquad
\mu_{k-1}=\alpha\downarrow\mu_{k},\qquad\mu\in\LLbd,\qquad k\in\bbZ.
\end{align*}
The category $\calC_R^\mu$ is 
generated by the set of modules
$\{\Delta(\mu_k)^t_R\,;\, k\in\bbZ\}$.
If $n_\alpha(\mu)=\ell$, then $\mu_k=\mu$ for all $k$ 
and  $\calP_R^\mu$ is a sum of categories equivalent $\Vect(R)$.
If $n_\alpha(\mu)\neq \ell$, then $\mu_{2k}=\mu+k\ell\Lalpha$ and 
$\mu_{2k-1}=\mu+(k\ell-n_\alpha(\mu))\Lalpha$ for all $k$.
Let $B_R$ be the quotient of the path $R$-algebra of the quiver 
$$
......
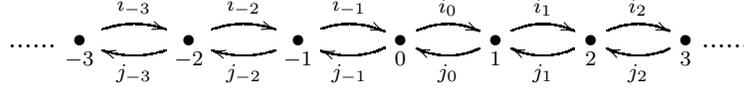
\xymatrix{
\underset{-3}\bullet \ar@/^/[r]^{i_{-3}} 
& \underset{-2}\bullet \ar@/^/[r]^{i_{-2}}\ar@/^/[l]^{j_{-3}}  
& \underset{-1}\bullet\ar@/^/[r]^{i_{-1}}\ar@/^/[l]^{j_{-2}} 
& \underset{0}\bullet \ar@/^/[r]^{i_{0}}\ar@/^/[l]^{j_{-1}} 
& \underset{1}\bullet \ar@/^/[r]^{i_{1}}\ar@/^/[l]^{j_{0}}  
& \underset{2}\bullet \ar@/^/[r]^{i_{2}}\ar@/^/[l]^{j_{1}} 
& \underset{3}\bullet\ar@/^/[l]^{j_{2}} 
}......
$$
modulo the following relations
\begin{align*}
&i_{k+1}\circ i_k= j_{k}\circ j_{k+1}=
i_{k-1}\circ j_{k-1}-j_{k}\circ i_k-(-1)^kh_\alpha1_{k}=0,
\qquad k\in\bbZ.
\end{align*}
Here $1_k$ is the idempotent at the vertex $k$. 
We have the following lemmas.

\begin{lemma}\label{lem:L1}
Let $R=S_\alpha$ with $\alpha\in\Phi^+$ such that $n_\alpha(\mu)\neq \ell$.
The category $\calP_R^\mu$ is equivalent to the category of projective $B_R$-modules. 
This equivalence takes the projective module $\P(\mu_k)^t_R$
to the $B_R$-module $B_R1_k$ for each $k\in\bbZ$.
\qed
\end{lemma}

\smallskip

\begin{lemma}\label{lem:L2} 
Let $R=S_\alpha$ with $\alpha\in\Phi^+$. 
\hfill
\begin{itemize}[leftmargin=8mm]
\item[$\mathrm{(a)}$]
$ Z(\calC_K)= \Fun(\LLbd,K)$. 
The $\lambda$-th component is the central character of the module $\Delta(\lambda)^t_K$.
\item[$\mathrm{(b)}$]
%as a subalgebra of $ Z(\calC_R\otimes_RK)=\prod_{\lambda\in\Lambda} K 1_\lambda$, we have
$ Z(\calC_R)= \{ (a_\lambda)\in \Fun(\LLbd,R)\,;\, a_\lambda\equiv 
a_{\alpha\downarrow\lambda} \modulo h_\alpha\}$.
\item[$\mathrm{(c)}$]
$ Z(\calC_F)= Z(\calC_R)\otimes_R F$.
\end{itemize}
\end{lemma}

\begin{proof}
Since $\LPhi_K=\emptyset$,
the category $\calC_K$ is  semi-simple, hence we have
\begin{align*}
Z(\calC_K)=\prod_{\lambda\in\LLbd}\End_{\calC_K}(\Delta(\lambda)^t_K)=\Fun(\LLbd,K).
\end{align*} 
Let us concentrate on (b).
From the proof of Lemma \ref{lem:L1}, for each integer $k$ we consider the element
$x_k=j_k\circ i_k$  in $\End_{\calC_R}(P_R(\mu_k)^t)$. We have
$\End_{\calC_R}(P_R(\mu_k)^t)=R 1_k\oplus Rx_k$ and
\begin{align*}
x_{k+1}\circ i_k=0,\quad
i_k\circ x_k=(-1)^{k+1}h_\alpha i_k,\quad
j_k\circ x_{k+1}=0,\quad
x_k\circ j_k=(-1)^{k+1}h_\alpha j_k.
\end{align*}
An element in $Z(\calC_R^\mu)$ is a tuple $z=(z_k)$ of morphisms $z_k\in\End_{\calC_R}(P_R(\mu_k)^t)$ 
such that 
$i_k\circ z_k=z_{k+1}\circ i_k$, $j_k\circ z_{k+1}=z_k\circ j_k$, and $z_k=a_k 1_k+b_k x_k$ for some $a_k, b_k\in R$. Then we have
\begin{align}\label{EQN}a_{k}=a_{k+1}+(-1)^{k}b_kh_\alpha, \qquad k\in\bbZ
\end{align}
and the multiplication in the center is given by
$$(zz')_k=a_ka'_k 1_k+(a_kb'_k+a'_kb_k+(-1)^{k+1}h_\alpha b_kb'_k)x_k.$$ 
for $z=(z_k)$, $z'=(z'_k).$
Since $h_\alpha\neq 0$ and $R$ is a domain, the element $z$ is uniquely determined by the tuple 
$(a_k)$  by \eqref{EQN}.
Thus the map $Z(\calC_R^\mu)\to R^\bbZ$ such that $z\mapsto(a_k)$
is an $R$-algebra embedding, whose image consists of the tuples
$(a_k)$ such that
$a_k\equiv a_{k+1} \modulo h_\alpha$ for all $k$.
\end{proof}

\smallskip

Now,  we consider the case $R=S$.

\begin{lemma}\label{lem:L3}
Let $R=S$. We have $\LPhi_S=\LPhi$. Further, the following hold.
\hfill
\begin{itemize}[leftmargin=8mm]
\item[$\mathrm{(a)}$]
$ Z(\calC_K)= \Fun(\LLbd,K)$. 
The $\lambda$-th component is the central character of the module $\Delta(\lambda)^t_K$.
\item[$\mathrm{(b)}$]
$ Z(\calC_S)=\{ (a_\lambda)\in \Fun(\LLbd,S)\,;\,
a_\lambda\equiv a_{\alpha\downarrow\lambda}\modulo h_\alpha,\ \forall\alpha\in\Phi^+\}$.
\end{itemize}
\end{lemma}

\begin{proof}
%For any  $S$-module $M$ we set $M_\alpha=M\otimes_SS_\alpha$. 
%We have $S=\bigcap_{\alpha\in\Phi^+}S_\alpha$.
%If $M$ is flat then we have
%$M=\bigcap_{\alpha\in\Phi^+}M_\alpha$. 
By Lemma \ref{lem:b} we have
$Z(\calC_S)=\bigcap_{\alpha\in\Phi^+} Z(\calC_{S_\alpha})$.
Further $Z(\calC_S)$ embeds in $ Z(\calC_K)$ in the obvious way.
Hence, Lemma \ref{lem:L2} yields
$$ Z(\calC_S)
=\bigcap_{\alpha\in\Phi^+} \Big\{ (a_\lambda)\in \Fun(\LLbd,S_\alpha)\,;\, a_\lambda\equiv 
a_{\alpha\downarrow\lambda} \modulo h_\alpha\Big\}
.$$
\end{proof}

Recall that $W^S_{\ell,\af}=W_{\ell,\af}$, that $\Xi_\sc=\LLbd/\W_{\ell,\af}$, and that $\pi_1$ acts on $\Xi_\sc$
with orbit space $\Xi=\Xi_\sc/\pi_1$. For each $\omega\in\Xi$ we set
\begin{align}\label{blocks3}
\calC^\omega_S=\bigsqcup_{\substack{\omega_\sc\in\Xi_\sc\\ \omega_\sc\in\omega}}\calC^{\omega_\sc}_S.
\end{align}

\smallskip

\begin{corollary} \label{cor:C1}
For each $\omega\in \Xi$
there is an $S$-algebra isomorphism
$Z(\calC_S^\omega)\to H^\bullet_T({}^0\!\Fl_{\ell,\gamma_\ell}^\omega)_{\hat 0}$.
\end{corollary}

\begin{proof}
From \eqref{split} and Lemma \ref{lem:L3} we deduce that
\begin{align*}
 Z(\calC_S^\omega)
&=\{ (a_\lambda)\in \Fun(W_{\ell,\ex}\bullet\omega\,,S)\,;\,
a_\lambda\equiv a_{\alpha\da\lambda}\modulo h_\alpha,\, \forall \alpha\in\Phi^+\},\\
&=\{ (a_x)\in \Fun(W_{\ell,\ex}^\omega\,,S)\,;\,
a_x\equiv a_{s_{\alpha+\ell m\delta}x}\modulo h_\alpha\, 
\,\operatorname{if}\,n_\alpha(x\bullet\omega)\neq\ell\}.
\end{align*}
Further, we have 
$$n_\alpha(x\bullet\omega)\neq\ell
\iff\langle\alpha,x\bullet\omega+\Lrho\rangle\notin\ell\bbZ
\iff x^{-1}s_{\alpha+\ell n\delta}x\notin W_{\ell,\omega},\quad\forall n\in\bbZ.$$
Thus, the lemma follows from Lemma \ref{lem:GKM0}.
\end{proof}

\smallskip

\subsubsection{Proof of Theorem $\ref{prop:P2}$}

Part (a) follows from Lemma \ref{lem:isom2}, the decomposition in \eqref{split} and Corollary \ref{cor:C1}.
Note that the map $\chi_\zeta^t$ is injective because the category $\calC_K$ is
semi-simple by the strong linkage principle.
Now, we prove (b). The ind-variety ${}^0\Gr^\zeta_{\gamma_\ell}$ is equivariantly formal.
Hence, we have
$H^\bullet_T({}^0\Gr^\zeta_{\gamma_\ell})_{\hat 0}\otimes_S\bfk=H^\bullet({}^0\Gr^\zeta_{\gamma_\ell})$.
Thus, the isomorphism $\bfa$ in (a)
specializes to an isomorphism 
$$H^\bullet({}^0\Gr^\zeta_{\gamma_\ell})\to Z(\Ut\mod^\LLbd_S)\otimes_S\bfk
\subset Z(\u_\zeta\mod^\LLbd).$$
The category $\u_\zeta\mod^\LLbd$ is $\ell\LLbd$-graded.
Let $\langle-\rangle$ be the grading shift.
In the category $\u_\zeta\mod^\LLbd$ we set
$$\Q=\bigoplus_{\lambda\in\LLbd_\ell}\P(\lambda)^t,\qquad
\A_{\ell\lambda}=\Hom_{\u_\zeta\mod^\LLbd}(\Q,\Q\langle\ell\lambda\rangle),\qquad
\A=\bigoplus_{\lambda\in\LLbd}\A_{\ell\lambda}.
$$
The $\ell\LLbd$-graded category $\u_\zeta\mod^\LLbd$ 
is equivalent to the $\ell\LLbd$-graded category $\A\mod^{\ell\LLbd}$ of all 
$\ell\LLbd$-graded finite dimensional $\A$-modules.
Since any element of $Z(\A)\cap \A_0$ acts on any object in $\A\mod^{\ell\LLbd}$, we have an inclusion
$$Z(\A)\cap \A_0\subset Z(\A\mod^{\ell\LLbd})=Z(\u_\zeta\mod^\LLbd).$$
The group $\ell\LLbd$ acts on the ring $Z(\u_\zeta\mod^\LLbd)$ as in \S\ref{sec:A5}.
The subring $Z(\A)\cap \A_0$ 
coincides with the ring of $\ell\LLbd$-invariant elements $Z(\u_\zeta\mod^\LLbd)^{\ell\LLbd}$. 
Forgetting the grading, the image $\bar\Q$ of $\Q$ in $\Rep(\u_\zeta)$ is a projective generator.
Since $\A=\End_{\u_\zeta}(\bar\Q)$, we have $Z(\u_\zeta)=Z(\A)$.
The group $\LT$ acts on $Z(\u_\zeta)$ as in \eqref{Ginv}.
Thus, we have
\begin{align}\label{uu}Z(\u_\zeta\mod^\LLbd)^{\ell\LLbd}= Z(\u_\zeta)^\LT.\end{align}

\medskip

\subsection{The center of $\calO^{\hyb}_{\zeta,S}$}
The goal of this section is to prove the following proposition.

\begin{proposition}\label{prop:P1}
There is a commutative diagram of algebra homomorphisms
\begin{align*}
\xymatrix{
H^\bullet_{T\times\bbG_m}(\Gr^\zeta)_{\widehat{0,0}}\ar@{^{(}->}[r]^-{\hat\bfb}\ar[d]&
%\ar@{^{(}->}@/^2pc/[rr]^-D
Z(\calO^{\hyb}_{q,S[[\hbar]]})\ar[d]\ar[r]^-{\chi_{\hat\zeta}^\hyb}&\Fun(\LLbd,S[[h]])\ar[d]\\
H^\bullet_T(\Gr^\zeta)_{\hat 0}\ar@{^{(}->}[r]^-\bfb&%\ar@{_{(}->}@/_2pc/[rr]_-D
Z(\calO^{\hyb}_{\zeta,S})\ar[r]^-{\chi_\zeta^\hyb}
&\Fun(\LLbd,S)
}
\end{align*}
The composed maps $\chi^\hyb_{\hat\zeta}\circ \hat\bfb$ and $\chi^\hyb_\zeta\circ \bfb$ 
coincide with the restriction maps
$$\res:H^\bullet_{T\times\bbG_m}({}^0\Gr^\zeta)\to\Fun(\LLbd,H^\bullet_{T\times\bbG_m}),\qquad
\res:H^\bullet_T({}^0\Gr^\zeta)\to\Fun(\LLbd,H^\bullet_T)$$
in $\S\ref{sec:GKM}$, up to a shift by $\Lrho$. In particular, they are injective.
\end{proposition}

\begin{proof}
By Proposition \ref{prop:N}, to define the commutative diagram it is enough to construct the maps 
\begin{align}\label{BU1}
\begin{split}
&\bfk[\widetilde N_{\Gamma}(T\times T/W)]^{\oplus\pi_1}_{\widehat{1,0}}\to Z(\calO^\hyb_{q,S[[\hbar]]}),\\
&\bfk[N_{\Gamma}(T\times T/W)]^{\oplus\pi_1}_{\hat 1}\to Z(\calO^\hyb_{\zeta,S}).
\end{split}
\end{align}
Composing them with \eqref{BUa}, we get the algebra homomorphisms
\begin{align}\label{BU2}
\hat\bfb:H^\bullet_{T\times\bbG_m}(\Gr^\zeta)_{\widehat{0,0}}
\to Z(\calO^\hyb_{q,S[[\hbar]]}),\qquad
\bfb:H^\bullet_T(\Gr^\zeta)_{\hat 0}
\to Z(\calO^\hyb_{q,S}).
\end{align}

%The proposition follows, because $\Gr^\zeta$ is equivariantly formal.

Now, we define the maps \eqref{BU1}.
It is enough to construct the first of them.
The inverse Harish-Chandra isomorphism \eqref{HC2} yields an $\calA$-algebra homomorphism
\begin{align}\label{map4}
\calA[\LT/W]^{\oplus\pi_1}_{\hat 0}\to Z(\calO^\hyb_{q,S[[\hbar]]}).
\end{align}
%The center $Z(\calO^\hyb_{q,S[[\hbar]]})$ is an $S[[\hbar]]$-algebra.
The map \eqref{map4} extends by $S[[\hbar]]$-linearity to an $S[[\hbar]]$-algebra homomorphism
\begin{align}\label{map3}
\calA[\LT\times\LT/W]^{\oplus\pi_1}_{\widehat{1,0}}
\to Z(\calO^\hyb_{q,S[[\hbar]]}).
\end{align}
We claim that \eqref{map3} extends further to a map
$$
\bfk[\widetilde{N}_{\LGamma}(\LT\times\LT/W)]^{\oplus\pi_1}_{\widehat{1,0}}
\to Z(\calO^\hyb_{q,S[[\hbar]]}).$$
Composing this map with the isomorphism \eqref{pairing2} we get an algebra homomorphism
$$\bfk[\widetilde{N}_{\Gamma}(T\times T/W)]^{\oplus\pi_1}_{\widehat{1,0}}
\to Z(\calO^\hyb_{q,S[[\hbar]]}).$$

To prove the claim we must check that \eqref{map3} maps $(I_{\LGamma})^{\oplus\pi_1}$
into $\hbar\,Z(\calO^{\hyb}_{q,S[[\hbar]]})$.
To do so we specialize \eqref{map3} to the point $\hbar=0$.
We get a map
\begin{align}\label{map8}
\bfk[\LT\times\LT/W]^{\oplus\pi_1}_{\hat 1}\to Z(\calO^\hyb_{\zeta,S}).
\end{align}
For any $\bfk[[\hbar]]$-algebra $A$ which is torsion free as a $\bfk[[\hbar]]$-module, we have an inclusion
$Z(A)/\hbar Z(A)\subset Z(A/\hbar A)$.
From Lemma \ref{lem:proj} we deduce that
$$Z(\calO^{\hyb}_{q,S[[\hbar]]})/\hbar\,Z(\calO^{\hyb}_{q,S[[\hbar]]})\subset Z(\calO^\hyb_{\zeta,S}).$$
Thus, we must check that the map \eqref{map8} takes $(I_{\LGamma})^{\oplus\pi_1}$ to $\{0\}$.
Note that \eqref{map8} factorizes through the inverse Harish-Chandra isomorphism \eqref{HC4}
and the obvious map $Z(\frakU^b_\zeta)\to Z(\calO^\hyb_{\zeta,S})$.
Therefore the map \eqref{map8} factorizes through the restriction
$\bfk[\LT\times\LT/W]^{\oplus\pi_1}_{\hat 1}\to\bfk[\LGamma]^{\oplus\pi_1}_{\hat 1}$,
by definition of \eqref{HC4}.
Hence, it kills the ideal $(I_{\LGamma})^{\oplus\pi_1}$.

Now, we compare the composed maps $\chi^\hyb_{\hat\zeta}\circ \hat\bfb$ and $\chi^\hyb_\zeta\circ \bfb$ 
coincide with the restriction maps. It is enough to consider the map $\chi^\hyb_{\hat\zeta}\circ \hat\bfb$.
The central character of the Verma module $\M(\mu)_{S[[\hbar]]}$ 
in $\calO^{\hyb}_{q,S[[\hbar]]}$ yields the map
\begin{align}\label{chimix}
\xymatrix{\bfA[\LT/W]\ar[r]^-{\text{diag}}&\bfA[\LT/W]^{\oplus\pi_1}\ar[r]^-{\eqref{HC2}}&Z(\calO_{q,S[[\hbar]]}^\hyb)
\ar[r]^-{\chi^\hyb_{\hat\zeta,\mu}}&S[[\hbar]],}
\end{align}
such that $g(t)\mapsto g(t^2q^{2(\mu+\Lrho)})$.
Thus, composing the map \eqref{map3} with the $Z(\calO^\hyb_{q,S[[\hbar]]})$-action on $\M(\mu)_{S[[\hbar]]}$
we get the map
$$\calA[\LT\times\LT/W]^{\oplus\pi_1}_{\widehat{1,0}}\to S[[\hbar]],\qquad
f\otimes g\mapsto f\cdot{}^{\tau_{\mu+\Lrho}}\!g.$$
By Lemma \ref{lem:extension}, 
the K-theoretic restriction to the point $\delta_\mu$ of the affine Grassmannian yields the map
$$\calA[\LT\times\LT/W]^{\oplus\pi_1}_{\widehat{1,0}}\to S[[\hbar]],\qquad
f\otimes g\mapsto f\cdot{}^{\tau_{\mu}}\!g.$$
Thus the composed map $\chi^\hyb_{\hat\zeta}\circ\hat\bfb$ coincides with
the cohomological restriction $\res$ in Proposition \ref{prop:N} up to a shift by $\Lrho$.
The maps $\bfb$ are injective because the restriction maps in \eqref{incl1} are injective.

\iffalse%%%%%%%%%%%
we consider
the action of the center on the deformed Verma modules.
It yields a map
$$\chi^\hyb_{\hat\zeta}:Z(\calO^{\hyb}_{q,S[[\hbar]]})\to\Fun(\LLbd,S[[\hbar]]).$$
We compose the maps \eqref{map3} and $\chi^\hyb_{\hat\zeta}$. We get a map
\begin{align}\label{map4}\bfk[T\times T/W\times\bbG_a]_{\widehat{1,0}}\to \Fun(\LLbd,S[[\hbar]]).
\end{align}
From the expression of the central character $\chi_\mu_{S[[\hbar]]}$ 
of the deformed Verma module $\M(\mu)_{S[[\hbar]]}$ given above,
we deduce that this map coincides with the map composed of \eqref{pi1} and \eqref{pi3}.
Hence it extends to a map
$$\bfk[\widetilde{N}_{\Gamma}(T\times T/W)]_{\widehat{1,0}}
\to \Fun(\LLbd,S[[\hbar]])$$
by Lemma \ref{lem:extension}, which coincides with \eqref{pi4} up to some completion.
Equivalently \eqref{map4} maps any element $x\in I_{\Gamma}$
into an element $x'\in \hbar\Fun(\LLbd,S[[\hbar]])$.
The map $\chi^\hyb_{\hat\zeta}$ is injective, because any element in $Z(\calO^{\hyb}_{q,S[[\hbar]]})$ is completely 
determined by its action on Verma modules by Lemma \ref{lem:mxss}.
Thus, to conclude it is enough to observe that 
$$\hbar\Im(\chi^\hyb_{\hat\zeta})=\Im(\chi^\hyb_{\hat\zeta})\cap\hbar\Fun(\LLbd,S[[\hbar]]).$$
\fi%%%%%%%%%%%%

\end{proof}

\smallskip

\begin{remark}
In fact, the embedding of $H^\bullet_T(\Gr^\zeta)_{\hat 0}$ into 
$Z(\calO^{\hyb}_{\zeta,S})$ is an isomorphism. A proof of this is given by Quan Situ in \cite{Situ}.
\end{remark}

\medskip

\subsection{The center of  $\U_\zeta$}\label{sec:ZUzeta}

Let $Z(\Rep(\U_{\hat\zeta}))$ be the completion of the algebra $Z(\Rep(\U_q))$ at $\hbar=0$.
%Recall from \eqref{Ginv} that we have $Z(\u_\zeta)^G=Z(\U_\zeta)\cap \u_\zeta$.

\begin{proposition} \label{prop:P0} There is an algebra homomorphism
$$\hat\bfc:\bfk[\widetilde N_{\Om}(T/W)]_{\hat 0}\to Z(\U_{\hat\zeta})$$
which specializes to an algebra homomorphism
$$\bfc:\bfk[N_{\Om}(T/W)]\to Z(\u_\zeta)^\LG.$$
\end{proposition}

\begin{proof}
The inverse Harish-Chandra homomorphism \eqref{IHC2} yields a map
\begin{align}\label{map6.1}\overline\bfhc:\calA[\LT/W]_{\hat 0}\to Z(\U_{\hat\zeta}).
\end{align}
Its fiber at $\hbar=0$ is the inverse Harish-Chandra homomorphism 
$\bfk[\LT/W]\to Z(\U_\zeta)$. Recall that 
$$\bfk[\widetilde{N}_{\LOm}(\LT/W)]=\calA[\LT/W][\hbar^{-1} I_{\LOm}].$$
We claim that $\overline\bfhc(I_{\LOm})\subset\hbar\,Z(\U_{\hat\zeta}).$
Thus, by \eqref{pairing2} the map \eqref{map6.1} extends to an algebra homomorphism
$$\hat\bfc:
\bfk[\widetilde N_{\Om}(T/W)]_{\hat 0}
=\bfk[\widetilde N_{\LOm}(\LT/W)]_{\hat 0}
\to Z(\U_{\hat\zeta}).$$
Taking the fiber at $\hbar=0$, we get the map $\bfc$.
To prove the claim, we compose \eqref{map6.1} with the specialization to $\hbar=0$.
We get a map
\begin{align}\label{map7.1}\calA[\LT/W]_{\hat 0}\to Z(\U_\zeta).
\end{align}
Since $Z(\U_q)$ is $\bfA$-flat, we have the inclusion
$$Z(\U_{\hat\zeta})\,/\,\hbar\,Z(\U_{\hat\zeta}) \subset Z(\U_\zeta).$$
Thus, it is enough to check that \eqref{map7.1} maps the ideal $I_\LOm$ to $\{0\}$.
To do so, we'll prove that the inverse Harish-Chandra map
$\overline\bfhc$ takes the ideal
$I_\LOm$ into $\frakU_{\hat\zeta}\cap\hbar\,\U_{\hat\zeta}.$
This follows from the factorization of the map $\overline\bfhc$ as the composition of a chain of maps
$$\calA[\LT/W]_{\hat 0}\to \bfk[\LG^\ast\times_{\LT/W} \LT/W]
\to\bfk[\{1\}\times\LOm]\to Z(\U_\zeta).$$
 
Now we prove that $\bfc$ maps into $Z(\u_\zeta)^\LG$.
To do so, by \eqref{Ginv} it is enough to prove that
\begin{align}\label{PT}
\overline\bfhc(I_\LOm)\subset \frakU_{\hat\zeta}\cap\hbar^2\U_{\hat\zeta}+\hbar\,\frakU_{\hat\zeta}.
\end{align}
To do that, we use another construction of the inverse Harish-Chandra map using
the universal $R$-matrix of $\scrU_q$.
For each dominant weight $\mu\in\LLbd^+$, let $\varphi_{\mu}$ be the character of the Weyl module 
$\V(\mu)_\zeta$
and $\psi_\mu$ the character of the simple module
$\D(\ell\mu)_\zeta=\Fr^*(\V(\mu)_1)$.
We consider the normalized characters which are defined as follow 
$$\bar\varphi_\mu=\varphi_\mu-\varphi_\mu(1)
\quad,\quad
\bar\psi_\mu=\psi_\mu-\varphi_\mu(1).$$
The algebra $\bfk[\LT/W]$ is generated by the characters $\varphi_\mu$,
and the ideal $I_\LOm$ by the normalized characters $\bar\psi_\mu$.
Let $R$ be the universal $R$-matrix of $\scrU_q$.
Given a finite dimensional $\U_q$-module $M$
we define the following element in $\scrU_q$
$$z_{q,M}=(1\otimes\Tr_M\,,\,R_{21}R(1\otimes K_{2\Lrho})).$$
The subscript $21$ indicates the flip of the factors and $\Tr_M$ is the trace of $M$.
Since $z_{q,M}$  and $\Tr_M$
only depend on the character $\varphi_M$ of $M$, we may write $z_{q,\varphi_M}=z_{q,M}$
and
$\Tr_{\varphi_M}=\Tr_M$.
The element $z_{q,\mu}=z_{q,\bar\psi_\mu}$ in $\scrU_q$ is well-defined, 
and it was observed by Drinfeld
that this element is central.
The formula for $R$ implies that  $z_{q,\mu}$ belongs to $Z(\U_q)$, see, e.g., \cite[\S8.3]{C94}.
Let $z_{\hat\zeta,\mu}$ be its image in $Z(\U_{\hat\zeta})$.
We have
\begin{align}\label{RR2}
z_{\hat\zeta,\mu}=\overline\bfhc(\bar\psi_\mu),
\end{align}
see, e.g., \cite[prop.~5]{B98}.
To prove \eqref{PT}, we must check the following
$$z_{\hat\zeta,\mu}\in\frakU_{\hat\zeta}\cap\hbar^2\U_{\hat\zeta}+\hbar\,\frakU_{\hat\zeta}.$$
Recall that $R=\Theta\cdot A$ where $\Theta$ is the quasi-R-matrix  which takes then following form
\begin{align*}
&\Theta=\sum_{n\in\bbN^{\LPhi^+}}\Theta_n,\qquad
\Theta_n=c_n\,E^n\otimes F^{(n)},\\ 
&E^n=\prod_{\Lalpha\in\LPhi^+}E_\Lalpha^{n_\Lalpha},
\quad F^{(n)}= \prod_{\Lalpha\in\LPhi^+}F_\Lalpha^{(n_\Lalpha)},
\quad c_n\in\bfA,\qquad c_0=1,
\end{align*}
where the products run over all positive roots in a fixed order, and $A$ is the linear operator such that
$$A(v_1\otimes v_2)=q^{(\mu_1,\mu_2)}v_1\otimes v_2,$$
whenever $v_i$ is a weight $\mu_i$ element of an integrable $\scrU_q$-module for each $i=1,2$.
We'll write $\ell \mid n$ if and only if $n\in(\ell\bbN)^{\LPhi^+}$.
Write 
$$\Theta=\Theta^{(\ell)}+\Theta^{(\ell')}+\Theta^{(0)},\qquad 
\Theta^{(\ell)}=\sum_{\substack{\ell\mid n\\n\neq 0}}\Theta_n,\qquad
\Theta^{(0)}=1.$$
We decompose accordingly
$z_{\hat\zeta,\mu}=z_{\hat\zeta,\mu}^{(\ell)}+z_{\hat\zeta,\mu}^{(\ell')}+z_{\hat\zeta,\mu}^{(0)}.$
Let $\lambda_n$ be the weight of the monomial $E^n$ and $a_\lambda$ 
be the projection to the weight $\lambda$ subspace. We have 
\begin{align*}
z_{\hat\zeta,\mu}^{(\ell)}&=\sum_{\substack{\ell\mid n\\n\neq 0}}\sum_{\lambda\in\LLbd}
c_n\Tr_{\bar\psi_\mu}(F^{(n)}E^{(n)}a_{\ell\lambda})\,E^nF^nK_{2\ell(\lambda-\lambda_n)},\\
z_{\hat\zeta,\mu}^{(\ell')}&=\sum_{\substack{\ell \nmid n\\n\neq 0}}\sum_{\lambda\in\LLbd}
c_n\Tr_{\bar\psi_\mu}(F^{(n)}E^{(n)}a_{\ell\lambda})\,E^nF^nK_{2\ell(\lambda-\lambda_n)},\\
z_{\hat\zeta,\mu}^{(0)}&=\sum_{\lambda\in\LLbd}
\dim(\V(\mu)_{1,\lambda})\,(K_{2\ell\lambda}-1).
\end{align*}
We deduce that $z_{\hat\zeta,\mu}^{(\ell)}$ belongs to $\hbar^2\U_{\hat\zeta}$.
Using \eqref{Fr1} we deduce also that
$z_{\hat\zeta,\mu}^{(\ell')}$ to $\hbar\,\frakU_{\hat\zeta}$.
Further $z_{\hat\zeta,\mu}^{(0)}$ is the image of the normalized character $\bar\varphi_\mu$ by the 
map  $f(e^\lambda)\mapsto f(K_{2\ell\lambda})$.
Since $\bar\varphi_\mu$ belongs to the square of the maximal ideal of 1 in $\bfk[\LT]$, we deduce that
$z_{\hat\zeta,\mu}^{(0)}$ belongs to $\hbar^2\U_{\hat\zeta}$, proving the proposition.

\end{proof}

\begin{corollary}\label{cor:C}
There is a commutative diagram of algebra homomorphisms
\begin{align*}
\xymatrix{
H^\bullet_{\bbG_m}(\Gr^\zeta)_{\hat 0}\ar[d]\ar[r]^-{\hat\bfc}& Z(\Rep(\U_{\hat\zeta}))\ar[d]\\
H^\bullet(\Gr^\zeta)\ar[r]^-\bfc& Z(\Rep(\U_\zeta)).}
\end{align*}
\end{corollary}

\begin{proof}
The inverse Harish-Chandra homomorphism \eqref{map6.1} yields an algebra homomorphism
\begin{align}\label{map6}\calA[\LT/W]^{\oplus\pi_1}_{\hat 0}\to Z(\Rep(\U_{\hat\zeta})).
\end{align}
By Proposition \ref{prop:P0} it extends to an algebra homomorphism
$$\hat\bfc:
H^\bullet_{\bbG_m}(\Gr^\zeta)_{\hat 0}=\bfk[\widetilde N_{\Om}(T/W)]^{\oplus\pi_1}_{\hat 0}
=\bfk[\widetilde N_{\LOm}(\LT/W)]^{\oplus\pi_1}_{\hat 0}
\to Z(\Rep(\U_{\hat\zeta})).$$
Taking the fiber at $\hbar=0$, we get the map $\bfc$. 

\end{proof}

\medskip

\subsection{The compatibility theorem}\label{sec:main}
In this section we put together the geometric constructions of the centers of 
$\frakU_\zeta^t$, $\U_\zeta$ and $\calO_{\zeta,S}^\hyb$.
First, we compare the maps $\bfb$ and $\bfc$.

\begin{proposition}\label{prop:P00}
There are commutative diagrams of algebra homomorphisms
\begin{align*}
\xymatrix{
H^\bullet_{T\times\bbG_m}(\Gr^\zeta)_{\widehat{0,0}}\ar[d]\ar[r]^-{\hat\bfb}& Z(\calO^\hyb_{q,S[[\hbar]]})\ar[d]\\
H^\bullet_{\bbG_m}(\Gr^\zeta)_{\hat 0}\ar[r]^-{\hat\bfc}& Z(\Rep(\U_{\hat\zeta}))}
\qquad\qquad
\xymatrix{
H^\bullet_T(\Gr^\zeta)_{\hat 0}\ar[d]\ar[r]^-\bfb& Z(\calO^\hyb_{\zeta,S})\ar[d]\\
H^\bullet(\Gr^\zeta)\ar[r]^-\bfc& Z(\Rep(\U_\zeta)).}
\end{align*}
\end{proposition}

\begin{proof} 
The left vertical maps in both diagrams are the specialization maps at 0 in $\frakt_\bfk$.
The right vertical maps are given by the composition of the specialization
$S\to\bfk$ and $Z(\I^\U_\hyb)$, where $Z(\I^\U_\hyb)$ are the maps associated with
the quotient functors
$$
\I^\U_\hyb: \calO^{\hyb}_{q,\bfk[[\hbar]]}\to\Rep(\U_{\hat\zeta}),\qquad
\I^\U_\hyb: \calO^\hyb_\zeta\to\Rep(\U_\zeta)
$$
as in \eqref{ZQF}.
To prove the proposition, it is enough to prove that
$$\xymatrix{H^\bullet_{\bbG_m}(\Gr^\zeta)_{\hat 0}\ar[r]^-{\hat\bfb}\ar[rd]_-{\hat\bfc}& 
Z(\calO^\hyb_{q,\k[[\hbar]]})\ar[d]^-{Z(\I^\U_\hyb)}\\
&Z(\Rep(\U_{\hat\zeta}).
}$$
commutes. Note that the morphisms $\bfb$, $\bfc$ are defined via the maps
$$\bfk[\widetilde N_{\LOm}(\LT/W)]^{\oplus\pi_1}_{\hat{0}}\to Z(\calO^\hyb_{q,\k[[\hbar]]}),\qquad
\bfk[\widetilde N_{\LOm}(\LT/W)]^{\oplus\pi_1}_{\hat{0}}\to Z(\Rep(\U_{\hat\zeta})),$$
which are generically the inverse Harish-Chandra map.
We consider the following diagram
$$\xymatrix{\bfk[\widetilde N_{\LOm}(\LT/W)]^{\oplus\pi_1}_{\hat{0}}\ar[r]^-{\hat\bfb}\ar[rd]_-{\hat\bfc}& 
Z(\calO^\hyb_{q,\k[[\hbar]]})\ar[d]^-{Z(\I^\U_\hyb)}\ar@{^{(}->}[r]&\prod_{M\in\calO^\hyb_{q,\k[[\hbar]]}}\End(M)
\ar[d]^-{\I^\U_\hyb}
\\
&Z(\Rep(\U_{\hat\zeta}))\ar@{^{(}->}[r]&\prod_{N\in\Rep(\U_{\hat\zeta})}\End(N).
}$$
The right horizontal arrows are given by the action of the center on the corresponding objects.
The right square is commutative by \eqref{sec:A4}.
The external diagram commutes because both maps from $\bfk[\widetilde N_{\LOm}(\LT/W)]^{\oplus\pi_1}_{\hat{0}}$ to $\prod_{N\in\Rep(\U_{\hat\zeta})}\End(N)$ are generically given by the central action via the inverse Harish-Chandra map.
\end{proof}

\smallskip

Now, let us proceed to the last main result of the paper.
In Theorem \ref{prop:P2} and Corollary \ref{cor:C} we constructed the ring homomorphisms 
$$\bfa: H^\bullet({}^0\Gr^\zeta_{\gamma_\ell})^{\ell\LLbd} \to Z(\u_\zeta)^\LT,\qquad
\bfc: H^\bullet(\Gr^\zeta)\to Z(\Rep(\U_\zeta)).$$
Recall that restriction yields a $\pi_1$-equivariant ring homomorphism 
$i^\ast: H^\bullet(\Gr^\zeta)\to H^\bullet({}^0\Gr^\zeta_{\gamma_\ell})^{W_{\ell,\af}}$, see \eqref{invariant-res}.
These maps together with the obvious embedding $Z(\u_\zeta)^\LG\subset Z(\u_\zeta)^\LT$ 
fit into a commutative diagram.

\smallskip

\begin{theorem}\label{thm:main2}
There is a commutative diagram of algebra homomorphisms
\begin{align*}
\xymatrix{
H^\bullet(\Gr^\zeta)\ar[rr]^{\bfc}
\ar[d]^{i^\ast} &&Z(\u_\zeta)^\LG\ar@{^{(}->}[d]\\
H^\bullet({}^0\Gr^\zeta_{\gamma_\ell})^{\ell\LLbd}\ar@{^{(}->}[rr]^-\bfa &&Z(\u_\zeta)^\LT
}
\end{align*}
All the morphisms are compatible with the block decompositions.
\end{theorem}

\smallskip

The image of $i^\ast$ is actually contained in the $W_{\ell,\af}$-invariant part of the target by \eqref{invariant-res}. 
Thus the theorem implies the following result.

\smallskip

\begin{corollary}
The morphism $\bfa$ restricts to an algebra embedding
$$H^\bullet({}^0\Gr^\zeta_{\gamma_\ell})^{W_{\ell,\ex}} \hookrightarrow Z(\u_\zeta)^\LG.$$
\end{corollary}

\begin{proof}
We need to show $\bfa\circ i^\ast=\bfc$. This will be given in several steps.

\smallskip

\emph{Step 1:} We first construct a morphism $\phi$ such that $\bfa\circ i^\ast=\phi\circ\bfb$.
By Theorem \ref{prop:P2} and Proposition \ref{prop:P1}, we have a commutative diagram
of $S$-algebras
\begin{align}\label{diagf}
\begin{split}
\xymatrix{
H^\bullet_T(\Gr^\zeta)_{\hat 0}\ar@{^{(}->}[rr]^\bfb \ar@{_{(}->}[d]_{i^\ast} &&
Z(\calO^{\hyb}_{\zeta,S})\ar[r]^-{\chi_\zeta^\hyb}&\Fun(\LLbd,S).\\
H^\bullet_T({}^0\Gr^\zeta_{\gamma_\ell})_{\hat 0}\ar[rr]^-\bfa_{\sim} &&Z(\Ut\mod^\LLbd_S)
\ar@{^{(}->}[ru]_-{\chi_\zeta^t}&
}
\end{split}
\end{align}
Let $Z(\calO^{\hyb}_{\zeta,S})'$ be the image of the map $\bfb$.
There is a unique algebra morphism $\phi_S$ which fits into the following commutative diagram
\begin{align}\label{square1}
\begin{split}
\xymatrix{
H^\bullet_T(\Gr^\zeta)_{\hat 0}\ar[rr]^-\bfb_-\sim \ar@{_{(}->}[d]_{i^\ast} &&
Z(\calO^{\hyb}_{\zeta,S})'\ar[d]^-{\phi_S}\\
H^\bullet_T({}^0\Gr^\zeta_{\gamma_\ell})_{\hat 0}\ar[rr]^-\bfa_-{\sim} &&Z(\Ut\mod^\LLbd_S).
}
\end{split}
\end{align}
Now, we apply the base change functor $-\otimes_S\bfk$ to the diagram \eqref{square1}.
By the equivariant formality of the ind-schemes $\Gr^\zeta$ and ${}^0\Gr^\zeta_{\gamma_\ell}$, the restriction
$H^\bullet_T(\Gr^\zeta)\to H^\bullet_T({}^0\Gr^\zeta_{\gamma_\ell})$ specializes to the restriction
$H^\bullet(\Gr^\zeta)\to H^\bullet({}^0\Gr^\zeta_{\gamma_\ell})$. 
Its image is contained in the $\ell\LQ$-invariant part by \eqref{invariant-res}.
On the center side, we have a
morphism $Z(\Ut\mod^\LLbd_S)\otimes_S\bfk\to Z(\u_\zeta\mod^\LLbd)$. 
Set $Z(\calO^{\hyb}_\zeta)'=Z(\calO^{\hyb}_{\zeta,S})'\otimes_S\bfk$ and $\phi=\phi_S\otimes_S\bfk$.
The base change yields the following commutative diagram 
$$\xymatrix{
H^\bullet(\Gr^\zeta)\ar[r]^-\bfb_-\sim\ar[d]_-{i^*}& 
Z(\calO^{\hyb}_\zeta)'\ar[d]^-\phi
\\
H^\bullet({}^0\Gr^\zeta_{\gamma_\ell})^{\ell\LQ}\ar[r]^-\bfa& 
Z(\u_\zeta\mod^\LLbd).
}$$
This completes the first step. 
Note that all the maps in the square above are $\pi_1$-equivariant. Moreover, we have 
$(H^\bullet({}^0\Gr^\zeta_{\gamma_\ell})^{\ell\LQ})^{\pi_1}=H^\bullet({}^0\Gr^\zeta_{\gamma_\ell})^{\ell\LLbd}$. 
Further, the image of $\bfa$ is contained in $Z(\u_\zeta\mod^\LLbd)^{\ell\LQ}$, whose $\pi_1$-invariant part is $Z(\u_\zeta\mod^\LLbd)^{\ell\LLbd}=Z(\u_\zeta)^\LT$ by \eqref{uu}. 
So taking the $\pi_1$-invariant of the diagram above yields
$$\xymatrix{
H^\bullet(\Gr^\zeta)^{\pi_1}\ar[r]^-\bfb_-\sim\ar[d]_-{i^*}& 
{Z(\calO^{\hyb}_\zeta)'}^{\pi_1}\ar[d]^-\phi
\\
H^\bullet({}^0\Gr^\zeta_{\gamma_\ell})^{\ell\LLbd}\ar[r]^-\bfa& 
Z(\u_\zeta)^\LT.
}$$

\smallskip

\emph{Step 2:} In order to compare $\phi\circ \bfb$ with $\bfc$, we'll express $\phi$ in other terms. 
For this purpose, we first construct a morphism of functors
$$\xi : \R^\u_b\circ\R^b_\hyb\to\R^\u_\U\circ\I^\U_\hyb.$$
as follows. Consider the unit $\eta: 1\to\R^\hyb_\U\circ\I^\U_\hyb$ associated with the adjoint pait $(\I^\U_\hyb, \R^\hyb_\U)$. Compose both sides with $\R^\u_b\circ\R^b_\hyb$, we get a morphism $\R^\u_b\circ\R^b_\hyb\to \R^\u_b\circ\R^b_\hyb\circ \R^\hyb_\U\circ\I^\U_\hyb$.
But since the composed morphism $\frakU^b_\zeta \to \U^{\hyb}_\zeta\to \U_\zeta$ factorizes through
$\frakU^b_\zeta \twoheadrightarrow \u_\zeta\hookrightarrow \U_\zeta$, we have an isomorphism of functors $\R^\u_b\circ\R^b_\hyb\circ \R_\U^\hyb\to\R^\u_\U$.
So $\R^\u_b\circ\R^b_\hyb\circ \R_\U^\hyb\circ \I^\U_\hyb\simeq \R^\u_\U\circ \I^\U_\hyb$. We get the desired morphism $\xi$.
The proof of the following lemma will be postponed to the end of the whole proof.

\smallskip

\begin{lemma}\label{lem:eta} Fix a projective generator $P$ in $\Rep(\u_\zeta)$. There is a module
$P^\hyb$ in $\calO^\hyb_\zeta$ such that
$\xi(P^\hyb)$ is an isomorphism $\R^\u_b\R^b_\hyb(P^\hyb)\to\R^\u_\U\I^\U_\hyb(P^\hyb)$.
Both modules $\R^\u_b\R^b_\hyb(P^\hyb)$ and $\R^\u_\U\I^\U_\hyb(P^\hyb)$ are isomorphic to $P$.
\end{lemma}

\smallskip

We'll abbreviate $P_b=\R^\u_b\R^b_\hyb(P^\hyb)$ and
$P_\U=\R^\u_\U\I^\U_\hyb(P^\hyb).$
The conjugation by $\xi(P^\hyb)$ yields an algebra isomorphism
$c_\xi:\End(P_b)\to\End(P_\U)$. 

\smallskip

\emph{Step 3: }
Now, we show that the following diagram commute
\begin{align}\label{diaga}\begin{split}
\xymatrix{
Z(\calO^{\hyb}_\zeta)'\ar[r]^-{\can}\ar[d]_{\phi}&\End(P^\hyb)\ar[d]_-{\R^\u_b\circ\R^b_\hyb} \ar[dr]^-{\R^\u_\U\circ\I^\U_\hyb}&\\
Z(\u_\zeta)
\ar@{^{(}->}[r]^-{\can_b}&\End(P_b) \ar[r]^{c_\xi}_{\sim} &\End(P_\U).
}\end{split}
\end{align}
Here $\can$, $\can_b$ are the natural morphisms given by the action of the center on objects in the corresponding categories. The commutativity of the right triangle in the diagram is given by the definition of $c_\xi$. 

Let us prove the commutativity of the left square.
To achieve this, we use a deformation argument.
Fix a module $P_S$ in $\frakU^t_\zeta\mod^\LLbd_S$ which specializes to the projective generator $P$.
The proof of Lemma \ref{lem:eta} (see below) implies that there is a module $P_S^\hyb$ in $\calO^\hyb_{\zeta,S}$
which specializes to $P^\hyb$ and such that $\R^t_b\R^b_\hyb(P^\hyb_S)$ is isomorphic to $P_S$.
Thus we have the following diagram which specializes to \eqref{diaga}
\begin{align}\label{diage}\begin{split}
\xymatrix{
H^\bullet_T(\Gr^\zeta)_{\hat 0}\ar[r]_-\sim^\bfb \ar@{_{(}->}[d]_{i^\ast} &
Z(\calO^{\hyb}_{\zeta,S})'\ar[r]^-{\can}&\End(P^\hyb_S)\ar[d]^-{\R^t_b\circ\R^b_\hyb}\\
H^\bullet_T({}^0\Gr^\zeta_{\gamma_\ell})_{\hat 0}\ar[r]^-\bfa &Z(\frakU^t_\zeta\mod^\LLbd_S)
\ar@{^{(}->}[r]^-{\can_b}&\End(P_S)
}\end{split}
\end{align}
It is enough to prove that \eqref{diage} commutes. This follows from the commutativity of 
the diagram \eqref{diagf}.

\smallskip

\emph{Step 4: }
We show the diagram 
\begin{align}\label{diagb}\begin{split}
\xymatrix{
H^\bullet(\Gr^\zeta)\ar[r]_-\sim^\bfb \ar[dr]_-\bfc &
Z(\calO^{\hyb}_\zeta)'\ar[r]^-{\can}&\End(P^\hyb)\ar[d]^-{\R^\u_\U\circ\I^\U_\hyb}\\
&Z(\u_\zeta)^\LG
\ar@{^{(}->}[r]^-{\can_\U}&\End(P_\U)
}\end{split}
\end{align}
commutes. Here $\can_\U$ is again the natural one given by the central action. It is injective because $P_\U$ is a projective generator for $\u_\zeta\mod$.

To this end,  we decompose the diagram as follows
\begin{align*}
\xymatrix{
H^\bullet(\Gr^\zeta)\ar[r]_-\sim^\bfb \ar[ddr]_-\bfc &
Z(\calO^{\hyb}_\zeta)'\ar[r]^-{\can}\ar[d]^-{Z(\I^\U_\hyb)}&\End(P^\hyb)\ar[d]^-{\I^\U_\hyb}\\
&Z(\U_\zeta)\ar[r]&\End(\I^\U_\hyb(P^\hyb))\ar[d]^-{\R^\u_\U}\\
&Z(\u_\zeta)^\LG\ar@{^{(}->}[u]
\ar@{^{(}->}[r]^-{\can_\U}&\End(P_\U)
}
\end{align*}
The upper square commutes by \S\ref{sec:A4}, the triangle commutes by  Proposition \ref{prop:P00},
and the lower square commutes because $\R^\u_\U$ is a restriction functor.

\smallskip

\emph{Step 5: }
We can now complete the proof. Using subsequently Steps 4, 3 and 1, we get
$$\begin{aligned}
\can_\U\circ\,\bfc&=(\R^\u_\U\circ\I^\U_\hyb)\circ\can\circ\bfb\\
&=c_\xi\circ\can_b\circ\,\phi\circ\bfb\\
&=c_\xi\circ\can_b\circ\,\bfa\circ i^\ast.
\end{aligned}$$
Now, since $\xi(P^\hyb)$ is a morphism of $\u_\zeta$-modules,
it intertwines the $Z(\u_\zeta)$-action. Hence $c_\xi\circ\can_b=\can_\U.$
We deduce 
$\can_\U\circ\,\bfc=\can_\U\circ\,\bfa\circ i^\ast.$
But $\can_\U$ is injective, hence $\bfc=\bfa\circ i^\ast.$ The proof for the theorem is complete.

\end{proof}

\medskip

\noindent \emph{Proof of Lemma $\ref{lem:eta}$.}
Recall that the Steinberg module $\St^\hyb$ in $\calO^\hyb_\zeta$ is the standard module 
$\M(\ell\Lrho-\Lrho)$. We first observe that the morphism
$$\xi(\St^\hyb):\R^\u_b\R^b_\hyb(\St^\hyb)\to\R^\u_\U I^\U_\hyb(\St^\hyb)$$
is invertible and that both
$\R^\u_b\R^b_\hyb(\St^\hyb)$ and $\R^\u_\U I^\U_\hyb(\St^\hyb)$
are isomorphic to the Steinberg module $\St^\u$ in $\Rep(\u_\zeta)$. 
To prove this, note that $\I^\U_\hyb(M)$ is the maximal quotient of $M$ such that the
$\U^\hyb_\zeta$-action on $M$ lifts to a $\U_\zeta$-action, for each module $M$ in $\calO^\hyb_\zeta$,
and that $\R^\u_\U\I^\U_\hyb(M)$ is canonically isomorphic to $\I^\U_\hyb(M)$ as a vector space.
On the other hand, $\R^\u_b\R^b_\hyb(M)$ is canonically isomorphic to 
$M\otimes_{\k[B^-]}\k$ as a vector space. Thus $\xi(M)$ can be viewed as a surjective linear map
$$\xi(M):M\otimes_{\k[B^-]}\k\to \I^\U_\hyb(M).$$
Under this identification, the morphism $\xi(\St^\hyb)$ is a surjective linear map 
from the space $\St^\hyb\otimes_{\k[B^-]}\k$ to the 
Weyl module $\V(\ell\Lrho-\Lrho)$. 
Since this Weyl module is cyclically generated by a highest vector $v_{\ell\Lrho-\Lrho}$,
modulo the subspace spanned by the elements $xv_{\ell\Lrho-\Lrho}$ where $x$ is any monomial in the
$F_\Lalpha^{(n\ell)}$'s with $n$ a positive integer and $\Lalpha$ a positive root, the map $\xi(\St^\hyb)$ is invertible.

Next, note that the categories $\Rep(\u_\zeta)$, $\frakU^b_\zeta\mod^\LLbd_S$, $\calO^\hyb_\zeta$
are module categories over the tensor category $\Rep(\U_\zeta)$, with the action given by the tensor product.
The functors $\R^\u_b$, $\R^b_\hyb$, $\R^\u_\U$ naturally commute with tensor products, 
because they are either the specialization or the pullback along some algebra homomorphisms.
The functor $I^\U_\hyb$ also commutes with tensor products by Corollary \ref{cor:proj2}.
From the description of the morphism $\xi$ above, we deduce that 
for each module $V\in\Rep(\U_\zeta)$, the following diagram commutes
$$\xymatrix{
\R^\u_b\R^b_\hyb(\St^\hyb\otimes V)\ar[rr]^{\xi(\St^\hyb\otimes V)}\ar[d]
&&\R^\u_\U I^\U_\hyb(\St^\hyb\otimes V)\ar[d]\\
\R^\u_b\R^b_\hyb(\St^\hyb)\otimes V\ar[rr]^{\xi(\St^\hyb)\otimes V}
&&\R^\u_\U I^\U_\hyb(\St^\hyb)\otimes V.}$$
We deduce that the morphism $\xi(Q)$ is invertible for each direct summand $Q$ of 
$\St^\hyb\otimes V$.
The lemma follows because any projective module in $\Rep(\u_\zeta)$ is a direct factor of 
$\St^\u\otimes V$ for some integrable $\U_\zeta$-module $V$.

\qed

\smallskip

The corollary above is motivated by a conjecture of Bezrukavnikov-Qi-Shan-Vasserot, which claims that,
in type $A$, there should be an algebra isomorphism 
$H^\bullet(\Fl_\gamma)^\LLbd\to Z(\u_\zeta^0)$.
We extend it to the following conjecture.

\smallskip

\begin{conjecture}\label{conj:B}
The morphism $\bfa$ restricts to an isomorphism
$$H^\bullet({}^0\Gr^\zeta_{\gamma_\ell})^{W_{\ell,\ex}} \simeq Z(\u_\zeta)^\LG.$$
\end{conjecture}

\iffalse%%%%%%%%%%%%%%%%%%%%%%%
\bigskip

\section{Positive characteristic analogue}

\hfill

In this section, let $\bfk$ be an algebraic closed field of positive characteristic.
{\color{red} COMPLETE}

\fi%%%%%%%%%%%%%%%%%%%%%%%%%%

\bigskip

\appendix

\section{The center of a category}

Let $\bfk$ be a field and $R,$ $S$ be $\bfk$-algebras which are commutative Noetherian integral domain.
All categories are assumed to be small and $\bfk$-linear. All functors are $\bfk$-linear.
The center $Z(\calC)$ of a category $\calC$ is the endomorphism ring of the identity functor.
An element $z\in Z(\calC)$ may be viewed as a tuple $z=(z_M)$ of compatible endomorphisms
$z_M\in\End_\calC(M)$ for each object $M$ in $\calC$.
%The center of an additive category is canonically isomorphic to the center of its idempotent completion.

\smallskip

\subsection{}
Let $\calC_S$ be an abelian $S$-linear category,
$\calP_S$ be the full subcategory of all compact projective objects in $\calC_S$,
and $\P_S$ be a faithful set of objects in $\calP_S$.
Consider the associative $S$-algebra
$$A_S=\bigoplus_{P,Q\in\P_S}\Hom_{\calP_S}(P, Q),$$
with the multiplication given by the opposite of morphism compositions in $\calP_S$.
The $S$-algebra $A_S$ is locally unital and the set
$\{1_P\in \End_{\calC_S}(P)\,;\,P\in\P_S\}$ is a 
complete set of mutually orthogonal primitive idempotents.
We'll assume that the category $\calP_S$ is Schurian, i.e.,  the $S$-module $1_P A_S1_Q$ 
is a projective of finite rank for all $P$, $Q$, and that
the functor $\bigoplus_{P\in\P_S}\Hom_{\calP_S}(P,-)$ is an equivalence from
$\calC_S$ to the category of all left $A_S$-modules $B$  such that $B=\bigoplus_{P\in\P_S}1_P B$.
Under this equivalence, the full subcategory $\calC^{fg}_S$ of all compact objects of $\calC_S$ is taken to $A_S\mod$,
and the category $\calP_S$ is taken to the category $A_S$-$\proj$ of all projective modules in $A_S\mod$.
The following lemma holds.

\begin{lemma}\label{lem:a}
The restriction yields an $S$-algebra isomorphism $Z(\calC_S)\to Z(\calP_S).$
\qed
\end{lemma}

Now, let $R$ be an $S$-algebra.
Let $\calP_S\otimes_SR$ be the idempotent completion of the
$R$-linear category with the same set of objects as 
$\calP_S$ and with the morphisms
$\Hom_{\calP_R}(M,N)=\Hom_{\calP_S}(M,N)\otimes_SR$. 
We write $M\otimes_SR$ for $M$ viewed as an object of $\calP_S\otimes_SR$.
Since the category $\calP_S$ is Schurian, the category  $\calP_S\otimes_SR$ is Schurian as well.
For each prime ideal $\frakp\subset S$, let $S_\frakp$ be the localization of $S$ at $\frakp$.
The obvious embedding 
$$\prod_{P\in\P_S}\End_{\calP_S}(P)\to 
\prod_{P\in\P_S}\End_{\calP_S\otimes_SS_\frakp}(P\otimes_S S_\frakp)$$ 
yields the inclusions
$$Z(\calP_S)\subset Z(\calP_S\otimes_SS_\frakp)\subset Z(\calP_S\otimes_SK).$$
Let $\frakP$ be a set of prime ideals in $S$ such that $M=\bigcap_{\frakp\in\frakP}M_\frakp$
for any object $M\in\calP_S$. For example, the set of all prime ideals of 
codimension one in $S$ satisfies this condition. 
The following lemma holds.

\begin{lemma}\label{lem:b}
We have $Z(\calC_S)=\bigcap_{\frakp\in\frakP}Z(\calP_S\otimes_SS_\frakp)$ as $S$-subalgebras of $Z(\calP_S\otimes_SK)$.
\qed
\end{lemma}

\iffalse%%%%%%%%%%%%%%%%%
\begin{proof}
It is enough to prove that $Z(\calP_R)=\bigcap_{\frakp\in\frakP}Z(\calP_{R_\frakp})$ in
$Z(\calP_K)$. 
We have the following Cartesian diagram
\[\xymatrix{Z(\calP_R)\ar[rr]\ar[d] &&\bigcap_{\frakp\in\frakP}Z(\calP_{R_\frakp})\ar[d]\\
\prod_{P\in\calP}\End_{\calP}(P)\ar[rr] && 
\bigcap_{\frakp\in\frakP}\Big(\prod_{P\in\calP}\End_{\calP_{R_\frakp}}(P\otimes_{R}R_\frakp)\Big),}\]
where all the arrows are inclusions. The lower horizontal arrow is bijective, because
\begin{align*}
\bigcap_{\frakp\in\frakP}\Big(\prod_{P\in\calP_R}\End_{\calP_{R_\frakp}}(P\otimes_{R}R_\frakp)\Big)
&=\prod_{P\in\calP_R}\Big(\bigcap_{\frakp\in\frakP}\End_{\calP_{R_\frakp}}(P\otimes_{R}R_\frakp)\Big)\\
&=\prod_{P\in\calP_R}\Big(\bigcap_{\frakp\in\frakP}\End_{\calP_R}(P)\otimes_RR_\frakp\Big)\\
&=\prod_{P\in\calP_R}\End_{\calP_R}(P).
\end{align*}
Thus, the top arrow is also bijective.
\end{proof}
\fi%%%%%%%%%%%%%%%%%

%%%%%%%%%%%%%%%%%

\smallskip

\subsection{} \label{sec:A4}
Let $(D,E)$ be a pair of adjoint functors with $D:\calC\to\calD$ and $E:\calD\to\calC$.
Assume further that the co-unit $\varepsilon:ED\to 1$ is invertible, i.e., the functor $E$ is a quotient functor.
Then, we have a map
\begin{align}\label{ZQF}
Z(E):Z(\calD)\to Z(\calC),\quad z\mapsto z'=\varepsilon\circ(1_E z1_D)\circ\varepsilon^{-1}.
\end{align}
Further, for each object $M\in\calD$ the following square is commutative, 
where the horizontal maps are given by the action 
of the center on $M$ and $E(M)$, 
$$\xymatrix{
Z(\calD)\ar[r]\ar[d]_-{Z(E)}&\End(M)\ar[d]^-E\\ Z(\calC)\ar[r]&\End(E(M)).}
$$

\smallskip

\subsection{} \label{sec:A5}
Let $\Lbd$ be any abelian group. By a $\Lbd$-graded category we mean a category $\calC$ 
along with a collection of shift functors $M\mapsto M\langle\lambda\rangle$ for each $\lambda\in\Lbd$ 
such that $\langle 0\rangle=\id$,
with isomorphisms of functors
$\langle\lambda\rangle\circ\langle\mu\rangle=\langle\lambda+\mu\rangle$ and
for each $\lambda,$ $\mu$.
If $\calC$ is $\Lbd$-graded, then the maps $Z(\langle\lambda\rangle)\in\End(Z(\calC))$ as $\lambda$ varies in $\Lbd$  
define an action of the group $\Lbd$ on $Z(\calC)$ by ring automorphisms.

\iffalse%%%%%%%%%%%%%%
Finally, let $F:\calD\to\calC$ be a functor and $P$ a projective generator of $\calD$ such that $F(P)$ is a 
projective generator of $\calC$.
Then, we have a map
\begin{align}\label{ZP}
Z(F):Z(\calD)=Z(\End(P))\to Z(\calC)=Z(\End(F(P)),\quad z\mapsto F(z).
\end{align}
\fi%%%%%%%%%%%%%%%%

\bigskip

\section{A geometric approach to the map $\bfc$}\label{appendixB}

In this section we give another approach to define the morphism
$$\bfc|_{H^\bullet(\Fl)}:H^\bullet(\Fl)\to Z(\Rep(\U_\zeta)\!^{\,\hat 0}).$$
In order to be able to use Frobenius weights, we extend the previous setting and consider the groups and 
varieties above over both fields $\bbF$ and $\bfk$, where $\bfk=\bar\bbQ_l$ and
$\bbF$ is an algebraic closure $\bar\bbF_q$
 of a finite field of characteristic $p$ which is large enough and prime to $l$.
Let ${}^0\calI$ be the pro-unipotent 
radical of the Iwahori group $\calI$.
Let $D^b_{m,{}^0\calI}(\Gr)$ be the 
${}^0\calI$-equivariant derived category of mixed complexes of $l$-adic sheaves
on the affine Grassmannian $\Gr$.
Let $D^b_{m,IW}(\Fl)$ be the Iwahori-Whittaker derived category of mixed complexes of $l$-adic sheaves
on the affine flag manifold $\Fl$. In both cases, by mixed we mean mixed in the sense of 
Beilinson-Ginzburg-Soergel.
Let (1/2) denote half the Tate twist and $\langle 1\rangle=[1](1/2)$.
Forgetting the mixed structures we get the categories $D^b_{{}^0\calI}(\Gr)$ and $D^b_{IW}(\Fl)$.
According to \cite{BY}, there is an equivalence of categories
\begin{align}\label{koszul}D^b_{m,{}^0\calI}(\Gr)\to D^b_{m,IW}(\Fl)\end{align}
which intertwines $(-1/2)$ on the left hand side with $\langle 1\rangle$ on the right hand side.
Since the cohomology of the affine flag manifold $\Fl$ is pure, we have a graded vector space morphism
\begin{align*}
H^\bullet(\Fl)&\longrightarrow\Hom\big(\id_{D^b_{m,IW}(\Fl)}\,,\,\id_{D^b_{m,IW}(\Fl)}\langle\bullet\rangle\,\big)\\
&\overset{{\eqref{koszul}}}\longrightarrow\Hom\big(\id_{D^b_{m,{}^0\calI}(\Gr)}\,,\,\id_{D^b_{m,{}^0\calI}(\Gr)}(-\bullet\!/2)\,\big).
\end{align*}
Composing it with the degrading functor 
\begin{align}\label{degrading} D^b_{m,{}^0\calI}(\Gr)\to D^b_{{}^0\calI}(\Gr),\end{align} we get a map
\begin{align}\label{map1}H^\bullet(\Fl)\to Z(D^b_{{}^0\calI}(\Gr)).\end{align}
By \cite{ABG} there is an equivalence of triangulated categories 
\begin{align}\label{ABG}D^b_{{}^0\calI}(\Gr)\to D^b(\Rep(\U_\zeta)\!^{\,\hat 0}).\end{align}
Composing \eqref{map1} and \eqref{ABG}, we get an algebra homomorphism
\begin{align}\label{map2}
\bfc':H^\bullet(\Fl)\to Z(\Rep(\U_\zeta)\!^{\,\hat 0}).
\end{align}
One may also use  the Kazhdan-Lusztig equivalence
and the Kashiwara-Tanisaki localization theorem instead of \cite{ABG} to define the map $\bfc'$.
One can check that $\bfc'=\bfc|_{H^\bullet(\Fl)}$.
We'll not do it here.

%%%%%%%%%%%%%%%%%%%%%%%%%%%%%%%%%%%%%%%%%%%%%%%%%%

\end{document}